\documentclass[12pt]{amsart}
\usepackage{amssymb,latexsym,amsmath,enumerate,nicefrac}

\usepackage{amsaddr}
\usepackage[total={6.5in,9in}, top=1in, left=1in, headsep = 0in]{geometry}

\usepackage{cite}

\usepackage{fancyhdr}
\usepackage{color}

\definecolor{egyptianblue}{rgb}{0.06, 0.2, 0.65}
\definecolor{cobalt}{rgb}{0.0, 0.28, 0.67}
\definecolor{yaleblue}{rgb}{0.06, 0.3, 0.57}
\definecolor{maroon}{rgb}{0.5, 0.0, 0.0}
\definecolor{indigo}{rgb}{0.0, 0.25, 0.42}
\definecolor{armygreen}{rgb}{0.29, 0.33, 0.13}
\definecolor{huntergreen}{rgb}{0.21, 0.37, 0.23}
\definecolor{royalblue}{rgb}{0.0, 0.14, 0.4}
\definecolor{zaffre}{rgb}{0.0, 0.08, 0.66}
\definecolor{myurlcolor}{rgb}{0.5, 0.0, 0.0}%0.8941,0,0.1686}
\definecolor{mycitecolor}{rgb}{0.06, 0.2, 0.65}%{0.0627,0.3843,0.6039}
\definecolor{myrefcolor}{rgb}{0.06, 0.2, 0.65}

\usepackage[bookmarks=false,pagebackref=true]{hyperref}
\hypersetup{colorlinks,
linkcolor=myrefcolor,
citecolor=mycitecolor,
urlcolor=myurlcolor}

\renewcommand*{\backref}[1]{}%(Referred to on page~#1.)}

\usepackage[all]{xy}
\usepackage{tikz}
\usetikzlibrary{calc}
\usepackage{graphics}
\usepackage{dsfont}
\usepackage{amsfonts}
\usepackage{mathtools}
\usepackage{comment}
\usepackage{tabto}

\makeatletter
\newcommand{\citecomment}[2][]{\citen{#2}#1\citevar}
\newcommand{\citeone}[1]{\citecomment{#1}}
\newcommand{\citetwo}[2][]{\citecomment[,~#1]{#2}}
\newcommand{\citevar}{\@ifnextchar\bgroup{;~\citeone}{\@ifnextchar[{;~\citetwo}{]}}}
\newcommand{\citefirst}{\@ifnextchar\bgroup{\citeone}{\@ifnextchar[{\citetwo}{]}}}

\makeatother

\newcommand{\catname}[1]{{\mathsf{#1}}} %{{\text{\sffamily #1}}}

\DeclareMathOperator{\Hom}{Hom}
\DeclareMathOperator{\Vol}{Vol}
\DeclareMathOperator{\Pon}{Pon}
\DeclareMathOperator{\rad}{rad}

\newcommand{\LCAG}{\catname{LCAG}}
\newcommand{\Abelian}{\catname{Ab}}

\renewcommand{\d}{{\rm d}}

\pgfkeys{/tikz/.cd, arrowhead/.default=-1pt, arrowhead/.code={}}

\allowdisplaybreaks

\usepackage{graphics}
\usepackage{dsfont}
\usepackage{amsbsy}
\usepackage{amsmath}
\usepackage{amsfonts,amssymb}
\usepackage{mathrsfs}
\usepackage{mathtools}

\usepackage{tikz}
\usepackage{bm}
\usetikzlibrary{calc}

\newcommand{\wt}{\rm wt}

\newcommand{\pair}[1]{\big(#1,\widehat{#1}\big)}

\newcommand{\Ball}{\mathds{B}}
\newcommand{\Shell}{\mathds{S}}
\newcommand{\Ballm}{\mathds{B}^{(m)}}
\newcommand{\Shellm}{\mathds{S}^{(m)}}
\newcommand{\Ballo}{\mathds{B}^{(0)}}
\newcommand{\Shello}{\mathds{S}^{(0)}}

\newcommand{\BallK}{\Ball_{\Kfield}}

\newcommand{\Ballt}{\Ball_{2}}
\newcommand{\BBo}{\Ball_2^{(0)}}
\newcommand{\BBtm}{\Ball_2^{(m)}}

\newcommand{\Shelloc}{\Shello_{\rm c}}
\newcommand{\Shellof}{\Shello_{\rm f}}

\newcommand{\SStmc}{\Shell_{\rm c}^{(m)}}
\newcommand{\SStmf}{\Shell_{\rm f}^{(m)}}

\newcommand{\Shelltc}{\Shell_{\rm c}}
\newcommand{\Shelltf}{\Shell_{\rm f}}

\newcommand{\TT}{\mathds{T}}
\newcommand{\ZZ}{\mathds{Z}}
\newcommand{\Zp}{\mathds{Z}_p}
\newcommand{\Qp}{\mathds{Q}_p}
\newcommand{\QQ}{\mathds{Q}}
\newcommand{\Qpdual}{\widehat{\Qp}}

\newcommand{\blank}{\cdot}

\newcommand{\Fourier}{\mathcal{F}}
\newcommand{\invFourier}{\Fourier^{-1}}

\newcommand{\indicator}[1]{\mathds{1}_{#1} }

\newcommand{\Kfield}{\mathds{K}}
\newcommand{\dualKfield}{\widehat{\mathds{K}}}

\newcommand{\OK}{\mathcal{O}_\Kfield}
\newcommand{\PK}{\mathcal{P}_\Kfield}

\newcommand{\h}{h}
\newcommand{\ConXG}{C(\G)}%{{\color{red}C}}
\newcommand{\ConMG}{C(1)}%{{\color{blue}{C({\max})}}}%C
\newcommand{\ConTd}{C(h)}%{{\color{teal}{C(h)}}}%C

\DeclareMathOperator{\tr}{tr}

\newcommand{\Gm}{{\mathds G_m}}

\newcommand{\x}{[x]}
\newcommand{\vol}{{\rm Vol}}
\newcommand{\Prob}{{\rm Prob}}
\newcommand{\Pb}{{\mathds P}}
\newcommand{\Pbm}{{\mathds P}^m}
\newcommand{\PbG}{{\mathds P}^\ast}
\newcommand{\EX}{{\mathds E}}
\newcommand{\EXm}{{\mathds E}_m}

\newcommand{\EXG}{{\mathds E}_\ast}

\newcommand{\ve}[1]{{\mathbf{#1}}}%\mathbf

\newcommand{\G}{{\mathds G}}
\newcommand{\Gd}{{\widehat{\mathds G}}}

\newcommand{\NN}{{\mathds N}}
\newcommand{\No}{{\mathds N}_0}

\newcommand{\e}{{\rm e}}

\newcommand{\Z}{\tilde{Z}}

\newcommand{\C}{\mathscr C}

\DeclarePairedDelimiter\floor{\lfloor}{\rfloor}

\DeclarePairedDelimiter\Gmap{[}{]}

\newcommand{\y}{\ve{y}} %{\Gdmap{\ve{y}}}
\renewcommand{\x}{\Gmap{\ve{x}}}

\newtheorem{theorem}{Theorem}[section]
\newtheorem{lemma}{Lemma}[section]
\newtheorem{proposition}{Proposition}[section]
\newtheorem{corollary}{Corollary}[section]
\theoremstyle{definition}
\newtheorem{definition}{Definition}[section]

\theoremstyle{remark}

\makeatletter
\def\namedlabel#1#2{\begingroup
    #2%
    \def\@currentlabel{#2}%
    \phantomsection\label{#1}\endgroup
}
\makeatother

\begin{document}

\title{Geometry-induced criticality in $p$-adic scaling limits of random walks}

\author{Rahul Rajkumar$^1$ \and David Weisbart$^2$}
% ORCID of David Weisbart: http://orcid.org/0000-0003-4689-1788

\address{
\begin{tabular}[h]{cc}
 $^{1,2}$Department of Mathematics\\
  University of California, Riverside
  \end{tabular}
  }
\email{$^1$rahul.rajkumar@email.ucr.edu} \email{$^2$weisbart@math.ucr.edu}

%%%%%%%%%%%%%%%%%%%%%%%%%%%%%%%%%%%%%%%%%%%%%%%%%%%%%%
%%%%%%%%%%%%%%%%%%%%                             %%%%%%%%%%%%%%%%%%%%%%%%%
%%%%%%%%%%%%%%%%%%%%        Abstract        %%%%%%%%%%%%%%%%%%%%%%%%%
%%%%%%%%%%%%%%%%%%%%                             %%%%%%%%%%%%%%%%%%%%%%%%%
%%%%%%%%%%%%%%%%%%%%%%%%%%%%%%%%%%%%%%%%%%%%%%%%%%%%%%

\begin{abstract}

An anisotropy parameter $h$ in $(0,1]$ induces on $\Qp^2$ a duality-compatible, two-scale filtration that collapses to one scale at the right endpoint. This filtration defines shell-uniform transition laws for hierarchical random walks on a discrete group whose scaling limits are L\'{e}vy processes on $\Qp^2$. The diffusion constants of the coordinate processes jump at the right endpoint, even though the radial jump law depends continuously on $h$. This instance of geometry-induced criticality isolates a structural mechanism that should extend to locally compact abelian groups and suggests a route to studying critical behavior in ultrametric models.

\end{abstract}


\maketitle

\tableofcontents

%

%%%%%%%%%%%%%%%%%%%%%%%%%%%%%%%%%%%%%%%%%%%%%%%%%%%%%%%%%%%%%%%%%%%%%%
%%%%%%%%%%%%%%%%%%%%%%%%%%%%%%%%%%%%%%%%%%%%%%%%%%%%%%%%%%%%%%%%%%%%%%
%%%%%%%%%%%%%%%%%%%%%%%%%%%%%%%%%%%%%%%%%%%%%%%%%%%%%%%%%%%%%%%%%%%%%%
%%%%%%%%%%%%%%%%%%%%%%%%%%%%%%%%%%%%%%%%%%%%%%%%%%%%%%%%%%%%%%%%%%%%%%
%%%%%%%%%%%%%%%%%%%%%%%%%%%%%%%%%%%%%%%%%%%%%%%%%%%%%%%%%%%%%%%%%%%%%%
%%%%%%%%%%%%%%%%%%%%%%%%%%%%%%%%%%%%%%%%%%%%%%%%%%%%%%%%%%%%%%%%%%%%%%
%%%%%%%%%%%%%%%%%%%%%%%%%%%%%%%%%%%%%%%%%%%%%%%%%%%%%%%%%%%%%%%%%%%%%%
%%%%%%%%%%%%%%%%%%%%%%%%%%%%%%%%%%%%%%%%%%%%%%%%%%%%%%%%%%%%%%%%%%%%%%
%%%%%%%%%%%%%%%%%%%%%%%%%%%%%%%%%%%%%%%%%%%%%%%%%%%%%%%%%%%%%%%%%%%%%%
%%%%%%%%%%%%%%%%%%%%%%%%%%%%%%%%%%%%%%%%%%%%%%%%%%%%%%%%%%%%%%%%%%%%%%
%%%%%%%%%%%%%%%%%%%%%%%%%%%%%%%%%%%%%%%%%%%%%%%%%%%%%%%%%%%%%%%%%%%%%%
%%%%%%%%%%%%%%%%%%%%%%%%%%%%%%%%%%%%%%%%%%%%%%%%%%%%%%%%%%%%%%%%%%%%%%

\thispagestyle{empty}

%%%%%%%%%%%%%%%%%%%%%%%%%%%%%%%%%%%%%%%%%%%%%%%%%%%%%%%%%%%%%%%%%%%%%%
%%%%%%%%%%%%%%%%%%%%%%%%%%%%%%%%%%%%%%%%%%%%%%%%%%%%%%%%%%%%%%%%%%%%%%
%%%%%%%%%%%%%%%%%%%%%%%%%%%%%%%%%%%%%%%%%%%%%%%%%%%%%%%%%%%%%%%%%%%%%%
%%%%%%%%%%%%%%%%%%%%%%%%%%%%%%%%%%%%%%%%%%%%%%%%%%%%%%%%%%%%%%%%%%%%%%
%%%%%%%%%%%%%%%%%%%%%%%%%%%%%%%%%%%%%%%%%%%%%%%%%%%%%%%%%%%%%%%%%%%%%%
%%%%%%%%%%%%%%%%%%%%%%%%%%%%%%%%%%%%%%%%%%%%%%%%%%%%%%%%%%%%%%%%%%%%%%
%%%%%%%%%%%%%%%%%%%%%%%%%%%%%%%%%%%%%%%%%%%%%%%%%%%%%%%%%%%%%%%%%%%%%%
%%%%%%%%%%%%%%%%%%%%%%%%%%%%%%%%%%%%%%%%%%%%%%%%%%%%%%%%%%%%%%%%%%%%%%
%%%%%%%%%%%%%%%%%%%%%%%%%%%%%%%%%%%%%%%%%%%%%%%%%%%%%%%%%%%%%%%%%%%%%%
%%%%%%%%%%%%%%%%%%%%%%%%%%%%%%%%%%%%%%%%%%%%%%%%%%%%%%%%%%%%%%%%%%%%%%
%%%%%%%%%%%%%%%%%%%%%%%%%%%%%%%%%%%%%%%%%%%%%%%%%%%%%%%%%%%%%%%%%%%%%%
%%%%%%%%%%%%%%%%%%%%%%%%%%%%%%%%%%%%%%%%%%%%%%%%%%%%%%%%%%%%%%%%%%%%%%

\section{Introduction}

For any prime $p$, the field $\Qp$ is the completion of $\QQ$ under the ultrametric induced by the $p$-adic absolute value $|\cdot|_p$. With addition, $\Qp$ is a non-discrete, totally disconnected object in $\LCAG$, the category of locally compact Hausdorff abelian groups with continuous homomorphisms.

For any positive $b$, the Vladimirov operator $\Delta_b$ (see Section~\ref{subsec:LimitProcess1D}) is an analog of the Laplacian on $L^2(\Qp)$. Take $\delta_0$ to be the Dirac measure centered at $0$.  For positive $\sigma$, the Cauchy problem
\begin{equation}\label{Intro:HeatEquation}
\begin{cases}
\displaystyle \frac{\d}{\d t}\,\rho(t,x) \,=\, -\sigma\,\Delta_b\,\rho(t,x),\\[4pt]
\rho(0,\cdot)=\delta_0,
\end{cases}
\end{equation}
admits a unique solution $\rho(t,\cdot)$ of unit mass for each $t$ in $(0,\infty)$. The family $\{\rho(t,\cdot)\}_{t>0}$ forms a convolution semigroup of probability density functions on $\Qp$ that give rise by standard construction to a probability measure $\Pb$ on $D([0,\infty)\colon\Qp)$, the Skorokhod space of $\Qp$-valued c\`adl\`ag paths endowed with the $J_1$ topology~\cite{Varadarajan:LMP:1997, bil1, Kochubei:Book:2001}.  
Take $Y$ to be the function from $[0,\infty)\times D([0,\infty)\colon \Qp)$ to $\Qp$ that is given by %
\begin{equation}\label{Intro:Yt}
Y(t,\omega) = \omega(t),
\end{equation} %
and $Y_t$ to be the random variable \begin{equation}\label{Intro:Ydef}Y_t(\omega) = \omega(t).\end{equation} %
This current work refers to the stochastic process $(D([0,\infty)\colon \Qp), \Pb, Y)$ as a $p$-adic Brownian motion. Authors traditionally refer to diffusion processes of this type as ultrametric diffusion processes~\cite{Avetisov:JPA:2002, Bik:UAA:2010}, but such processes may potentially be more general.

Prior work shows that $p$-adic Brownian motion is a limit of a sequence of discrete-time random walks in $\Qp$~\cite{BW}; in fact, it is the scaling limit of a single discrete-time random walk on a discrete group via sequences of embeddings that densely and exhaustively approximate $\Qp$ \cite{BW}.  Analogous results hold for vector spaces over local fields~\cite{W:Expo:24}, and for Brownian motion on $\Zp$~\cite{Pierce_Weisbart:JSP:2025}, though in the latter case it is only known that the Brownian motion is a limit of a sequence of random walks. These motivate a general perspective on constructing scaling limits of random walks in $\LCAG$.

A central motivation in $p$-adic mathematical physics is that $\Qp$ exhibits hierarchical, ultrametric structure while remaining analytically tractable.  Researchers have pursued $p$-adic models as idealized descriptions of various systems displaying approximate hierarchical or ultrametric structures \cite{Dragovich:Biosystems:2021,Dragovich:PNUAA:2009,Dragovich:PNUAA:2017}, including models of contagion in the presence of social clustering  \cite{Volov:PNUAA:2020,Khrennikov:medRxiv:2020,Khrennikov:Entropy:2016}, neural networks \cite{ZunigaGalindo:JNMP:2024}, data science \cite{Bradley:arXiv:2024,Bradley:JoC:2025}, and geophysics \cite{Kochubei:UkrMathJ:2018, Khrennikov:JFAA:2016b}. Recent works that utilized the Vladimirov equation to describe protein molecule dynamics have received experimental confirmation~\cite{ABZ:ProcSteklov:2014, BZ:Physica:2021}. In addition, $p$-adic analogues of classical stochastic processes are of independent interest in number theory, where they promise to bring tools of random matrix theory to bear on arithmetic questions \cite{Peski:arXiv:2021, Peski:arXiv:2024a, Peski:arXiv:2023, Peski:arXiv:2024b}. 

Hierarchical models serve as toy models of statistical mechanics \cite{Easo:JMP:2024}. Possible links between $p$-adic diffusion and percolation criticality on hierarchical lattices further motivate the present study. Recently, $p$-adic constructions have been used as a starting point in percolation theory in the form of hierarchical lattices for long-range Bernoulli percolation \cite{Hutchcroft:arXiv:2025II,Hutchcroft:arXiv:2025I,Hutchcroft:arXiv:2025III,Abdesselam:arXiv:2018}. Arguably, the most famous example of ultrametricity is in the energy landscapes of spin glasses \cite{Panchenko:AnnMath:2013}, the theory of which has found widespread interdisciplinary application \cite{Macy:npjComplexity:2024, Megard:WSLNP:nd}.  Since near-optima are organized as an ultrametric tree \cite{Auffinger:arXiv:2022}, the study of $p$-adic stochastic processes may shed light on dynamics in spin glasses and related systems \cite{Avetisov_Bikulov_Kozyrev:JPA:1999,Avetisov:JPA:2002,Avetisov:JPA:2003}. 

Section~\ref{Sec:LimitingSpacesandDisApp} develops a framework on $\LCAG$ for choosing a discrete group, a primitive random walk, and a spatial scaling. Its review of quadratic field extensions motivates subsequent development in $\Qp^2$. Section~\ref{Sec:One-Comp} extends earlier work~\cite{WJPA} in the $\Qp$ setting by introducing a simple perturbation that adjusts the diffusion coefficient of the limit without altering space-time scaling. The principle result of this section is Theorem~\ref{Sec:Con:Theorem:MAIN}.  Section~\ref{Sec:Prim} reviews the basic facts about the specialization of the process discussed earlier \cite{W:Expo:24} to the present setting, but simplifies some of the proofs.

Section~\ref{Sec:Perturbed} introduces a one-parameter family of two-component hierarchical random walks with anisotropy parameter $h$ in $(0,1]$ and establishes their basic properties. It shows in Theorem~\ref{Sec:Con:Theorem:MAIN2} that each member has a scaling limit given by a L\'{e}vy process on $\Qp^2$. This extends the known setting in which diffusion in the non-Archimedean setting is a scaling limit.  Anisotropy in the discrete coordinate processes detects metric symmetries and controlled asymmetries of the limit. Section~\ref{Sec:Components} proves that as $h$ approaches the isotropic right endpoint, the probabilities for jumping into shells for the random walks depend continuously on $h$, but lead to non-uniformity within shells in limit.  Theorem~\ref{thm:componentCrit} quantifies an effect on the continuum limit by showing that the diffusion coefficients of the limiting coordinate processes exhibit a jump. The mechanism is geometric: an alternating two-scale filtration collapses when $h$ is equal to $1$, changing shell combinatorics and forcing the discontinuity. This furnishes a toy model for critical behavior. Although the exposition focuses on $\Qp^2$, the constructions should extend to objects in $\LCAG$ that admit a compact open subgroup.

%%%%%%%%%%%%%%%%%%%%%%%%%%%%%%%%%%%%%%%%%%%%%%%%%%%%%%%%%%%%%%%%%%%%%%
%%%%%%%%%%%%%%%%%%%%%%%%%%%%%%%%%%%%%%%%%%%%%%%%%%%%%%%%%%%%%%%%%%%%%%
%%%%%%%%%%%%%%%%%%%%%%%%%%%%%%%%%%%%%%%%%%%%%%%%%%%%%%%%%%%%%%%%%%%%%%
%%%%%%%%%%%%%%%%%%%%%%%%%%%%%%%%%%%%%%%%%%%%%%%%%%%%%%%%%%%%%%%%%%%%%%
%%%%%%%%%%%%%%%%%%%%%%%%%%%%%%%%%%%%%%%%%%%%%%%%%%%%%%%%%%%%%%%%%%%%%%
%%%%%%%%%%%%%%%%%%%%%%%%%%%%%%%%%%%%%%%%%%%%%%%%%%%%%%%%%%%%%%%%%%%%%%
%%%%%%%%%%%%%%%%%%%%%%%%%%%%%%%%%%%%%%%%%%%%%%%%%%%%%%%%%%%%%%%%%%%%%%
%%%%%%%%%%%%%%%%%%%%%%%%%%%%%%%%%%%%%%%%%%%%%%%%%%%%%%%%%%%%%%%%%%%%%%
%%%%%%%%%%%%%%%%%%%%%%%%%%%%%%%%%%%%%%%%%%%%%%%%%%%%%%%%%%%%%%%%%%%%%%
%%%%%%%%%%%%%%%%%%%%%%%%%%%%%%%%%%%%%%%%%%%%%%%%%%%%%%%%%%%%%%%%%%%%%%
%%%%%%%%%%%%%%%%%%%%%%%%%%%%%%%%%%%%%%%%%%%%%%%%%%%%%%%%%%%%%%%%%%%%%%
%%%%%%%%%%%%%%%%%%%%%%%%%%%%%%%%%%%%%%%%%%%%%%%%%%%%%%%%%%%%%%%%%%%%%%

\section{Limiting spaces and their discrete approximations}\label{Sec:LimitingSpacesandDisApp}

%%%%%%%%%%%%%%%%%%%%%%%%%%%%%%%%%%%%%%%%%%%%%%%%%%%%%%%%%%%%%%%%%%%%%%
%%%%%%%%%%%%%%%%%%%%%%%%%%%%%%%%%%%%%%%%%%%%%%%%%%%%%%%%%%%%%%%%%%%%%%
%%%%%%%%%%%%%%%%%%%%%%%%%%%%%%%%%%%%%%%%%%%%%%%%%%%%%%%%%%%%%%%%%%%%%%
%%%%%%%%%%%%%%%%%%%%%%%%%%%%%%%%%%%%%%%%%%%%%%%%%%%%%%%%%%%%%%%%%%%%%%
%%%%%%%%%%%%%%%%%%%%%%%%%%%%%%%%%%%%%%%%%%%%%%%%%%%%%%%%%%%%%%%%%%%%%%
%%%%%%%%%%%%%%%%%%%%%%%%%%%%%%%%%%%%%%%%%%%%%%%%%%%%%%%%%%%%%%%%%%%%%%
%%%%%%%%%%%%%%%%%%%%%%%%%%%%%%%%%%%%%%%%%%%%%%%%%%%%%%%%%%%%%%%%%%%%%%
%%%%%%%%%%%%%%%%%%%%%%%%%%%%%%%%%%%%%%%%%%%%%%%%%%%%%%%%%%%%%%%%%%%%%%
%%%%%%%%%%%%%%%%%%%%%%%%%%%%%%%%%%%%%%%%%%%%%%%%%%%%%%%%%%%%%%%%%%%%%%
%%%%%%%%%%%%%%%%%%%%%%%%%%%%%%%%%%%%%%%%%%%%%%%%%%%%%%%%%%%%%%%%%%%%%%
%%%%%%%%%%%%%%%%%%%%%%%%%%%%%%%%%%%%%%%%%%%%%%%%%%%%%%%%%%%%%%%%%%%%%%
%%%%%%%%%%%%%%%%%%%%%%%%%%%%%%%%%%%%%%%%%%%%%%%%%%%%%%%%%%%%%%%%%%%%%%

%%%%%%%%%%%%%%%%%%%%%%%%%%%%%%%%%%%%%%%%%%%%%%%%%%%%%%%%%%%%%%%%%%%%%%
%%%%%%%%%%%%%%%%%%%%%%%%%%%%%%%%%%%%%%%%%%%%%%%%%%%%%%%%%%%%%%%%%%%%%%

\subsection{The idea of a Pontryagin filtration}

%%%%%%%%%%%%%%%%%%%%%%%%%%%%%%%%%%%%%%%%%%%%%%%%%%%%%%%%%%%%%%%%%%%%%%
%%%%%%%%%%%%%%%%%%%%%%%%%%%%%%%%%%%%%%%%%%%%%%%%%%%%%%%%%%%%%%%%%%%%%%

For any object $G$ in $\LCAG$, its Pontryagin dual $\widehat{G}$ is the group $\Hom(\G\colon \TT)$ endowed with the compact-open topology, where $\TT$ is the group of unit complex numbers.  The elements of $\widehat{G}$ are the characters of $G$.  For any subset $U$ of $G$, the \emph{annihilator} of $U$ is %
\begin{equation}
U^\perp\coloneqq\Big\{\xi\in\widehat{G}\colon\xi(u)=1,\; \forall u\in U\Big\}.
\end{equation}
The \emph{annihilator operation} takes $U$ to $U^\perp$ and reverses inclusion.

\begin{definition}[Regular filtration]
A \emph{filtration} for $G$ is a non-empty countable family $S$ of compact open subgroups of $G$ forming a neighborhood base at the identity and totally ordered by inclusion.  Groups $A$ and $A^\prime$ in a filtration $S$ for $G$ are \emph{adjacent} if there is no $B$ in $S$ that satisfies \[A\subset B\subset A^\prime.\]  A filtration for $G$ is a \emph{regular filtration} if for any adjacent $A$ and $A^\prime$ in $S$, the order $[A^\prime\colon A]$ of $A^\prime\slash A$ is finite, and
\begin{equation*}
N(S)\coloneqq [A^\prime\colon A]
\end{equation*}
is independent of the adjacent pair.  Refer to $N(S)$ as the \emph{index factor} of the filtration. 
\end{definition}

\begin{definition}[Pontryagin filtration]
Regular filtrations $\Pon(G)$ for $G$ and $\Pon(\widehat{G})$ for $\widehat{G}$ are \emph{annihilator compatible} if $U^\perp$ is in $\Pon(\widehat{G})$ for every $U$ in $\Pon(G)$ and $V^\perp$ is in $\Pon(G)$ for every $V$ in $\Pon(\widehat{G})$.  A \emph{Pontryagin filtration} for $\pair{G}$ is a pair $\big(\!\Pon(G),\Pon(\widehat{G})\big)$ of annihilator compatible, regular filtrations for $G$ and $\widehat{G}$, respectively.
\end{definition}

\begin{proposition}\label{prop:indexoffiltration}
For any Pontryagin filtration $\big(\!\Pon(G),\Pon(\widehat{G})\big)$ for a dual pair $\pair{G}$ of objects in $\LCAG$,
\begin{equation}
N\big(\!\Pon(G)\big)=N\big(\!\Pon(\widehat{G})\big).
\end{equation}
\end{proposition}

%%%%%%%%%%%%%%%%%%%%%%%%%%%%%%%%%%%%%%%%%%%%%%%%%%%%%%%%%%%%%%%%%%%%%%
%%%%%%%%%%%%%%%%%%%%%%%%%%%%%%%%%%%%%%%%%%%%%%%%%%%%%%%%%%%%%%%%%%%%%%

\begin{proof}
Annihilator compatibility of the Pontryagin filtration and the reversal of inclusion under the annihilator operation together imply that the annihilator operation preserves adjacency.  

Take $A$ and $A^\prime$ to be adjacent in $\Pon(G)$, and $A$ to be a subgroup of $A^\prime$. Every homomorphism from a discrete group to $\TT$ is
continuous, so $\Hom_{\Abelian}(-,\TT)$, the $\Hom$ functor in the category $\Abelian$ of abelian groups, agrees with $\Hom(-,\TT)$ on discrete groups.  Universality of the quotient implies that %
\[
(A^\prime)^\perp \cong \Hom(G\slash A^\prime,\TT) \quad \text{and} \quad A^\perp \cong \Hom(G\slash A,\TT).
\] %
Since $A$ and $A^\prime$ are open, the groups $A^\prime\slash A$, $G\slash A$, and $G\slash A^\prime$ are discrete. Since $\TT$ is divisible, Baer's criterion implies that it is injective \cite[Theorem~21.1]{Fuchs:IAG:1970}, hence the functor $\Hom_{\Abelian}(-,\TT)$ is exact \cite[Proposition~4.6]{CartanEilenberg:HA:1956}, and therefore takes the exact sequence 
\[
0\longrightarrow A^\prime\slash A\longrightarrow G\slash A\longrightarrow G\slash A^\prime\to0
\]
to the exact sequence
\[
0\longrightarrow (A^\prime)^\perp\longrightarrow A^\perp \longrightarrow \Hom(A^\prime\slash A, \TT)\longrightarrow 0.
\]
Exactness of the sequence implies that $A^{\perp}\slash (A^\prime)^\perp$ is isomorphic to $\Hom(A^\prime\slash A, \TT)$ in $\Abelian$. Both $A$ and $A^\prime$ are compact open, so $A^\prime\slash A$ is finite, and Pontryagin duality for finite abelian groups implies that
\[
\#\big(A^{\perp}\slash (A^\prime)^\perp\big) =\#\Hom(A^\prime\slash A,\TT) = \#(A^\prime\slash A) = [A^\prime\colon A].
\]
\end{proof}

%%%%%%%%%%%%%%%%%%%%%%%%%%%%%%%%%%%%%%%%%%%%%%%%%%%%%%%%%%%%%%%%%%%%%%
%%%%%%%%%%%%%%%%%%%%%%%%%%%%%%%%%%%%%%%%%%%%%%%%%%%%%%%%%%%%%%%%%%%%%%

\begin{definition}[Self-dual Pontryagin filtration]
For any self-dual $G$ in $\LCAG$ with a fixed isomorphism $\Psi$ from $G$ to $\widehat{G}$, a regular filtration $\Pon(G)$, if it exists, is a \emph{Pontryagin filtration for $G$ relative to $\Psi$} if $\Psi(U^\perp)$ is in $\Pon(G)$ whenever $U$ is in $\Pon(G)$.
\end{definition}

The isomorphism $\Psi$ in the definition of a Pontryagin filtration $\Pon(G)$ for a self-dual group $G$ makes it possible to regard $\Pon(G)$ as closed under the annihilator operation. However, this notion of annihilator compatibility depends on the chosen isomorphism.  Indeed, if $\Psi_1$ and $\Psi_2$ are two isomorphisms from $G$ to $\widehat{G}$ and $\Pon(G)$ is a Pontryagin filtration for $G$ relative to $\Psi_1$, then for
\[
\Psi = \Psi_2^{-1} \!\circ \Psi_1
\]
the filtration
\[
\Psi(\Pon(G)) = \{\Psi(U)\colon U \in \Pon(G)\}
\]
is a Pontryagin filtration for $G$ relative to $\Psi_2$. For any set $S$, denote by $\indicator{S}$ its indicator function.

\begin{lemma}\label{lem:LCAG:Int}
For any $G$ in $\LCAG$, take $\mu_G$ to be a Haar measure on $G$. For any compact open subgroup $U$  of $G$ and any $\xi$ in $\widehat{G}$, %
\[
\int_U \xi(x) \,\d\mu_G(x)\;=\;\Vol(U)\,\indicator{U^\perp}(\xi).
\]
\end{lemma}

\begin{proof}
Since $U$ is compact open, its volume is positive and finite. For any $\xi$ in $\widehat{G}$, take \[I(\xi)\coloneqq\int_U \xi(x)\,\d\mu_G(x).\]
If $\xi$ is in $U^\perp$, then \[\xi\big|_U\equiv 1,\quad \text{hence} \quad I(\xi)=\int_U 1\,\d\mu_G=\Vol(U).\]  If $\xi$ is not in $U^\perp$, then there exists $u_0$ in $U$ with $\xi(u_0)$ not equal to $1$.  Translation invariance of the Haar measure implies that
\begin{align*}
I(\xi) &=\int_U \xi(x)\,\mathrm d\mu_G(x)\\
&=\int_U \xi(x+u_0)\,\mathrm d\mu_G(x)\\
&=\xi(u_0)\int_U \xi(x)\,\mathrm d\mu_G(x) =\xi(u_0)\,I(\xi).
\end{align*}
Since $\xi(u_0)$ is not equal to $1$, $I(\xi)$ is equal to $0$.
\end{proof}

%%%%%%%%%%%%%%%%%%%%%%%%%%%%%%%%%%%%%%%%%%%%%%%%%%%%%%%%%%%%%%%%%%%%%%
%%%%%%%%%%%%%%%%%%%%%%%%%%%%%%%%%%%%%%%%%%%%%%%%%%%%%%%%%%%%%%%%%%%%%%

\subsection{Harmonic analysis on $\Qp$}

%%%%%%%%%%%%%%%%%%%%%%%%%%%%%%%%%%%%%%%%%%%%%%%%%%%%%%%%%%%%%%%%%%%%%%
%%%%%%%%%%%%%%%%%%%%%%%%%%%%%%%%%%%%%%%%%%%%%%%%%%%%%%%%%%%%%%%%%%%%%%

See the classic book of Vladimirov, Volovich, and Zelenov \cite{vvz} for a general introduction to $p$-adic analysis and $p$-adic mathematical physics to provide further background.  

For any $k$ in $\ZZ$, the compact open sets %
\begin{equation}
\Ball(k) \coloneqq \{x\in\Qp\colon |x|_p \le p^{k}\} \quad \text{and} \quad \Shell(k) \coloneqq \{x\in\Qp\colon |x|_p = p^{k}\}
\end{equation}
are the centered closed ball and its shell of radius $p^{k}$, respectively. The ball $\Ball(0)$ is the ring of $p$-adic integers $\Zp$, the unique maximal compact subring of $\Qp$.  It contains the integers $\ZZ$ as a dense subset and its unique maximal ideal is $p\Zp$, the set $\{x\colon |x|_p<1\}$.  The residue field $\Zp/p\Zp$ has $p$ elements.

For any $x$ in $\Qp$ there is a unique function $a_x$ from $\ZZ$ to $\{0, 1, \dots, p-1\}$ and an integer $K$ so that $a_x$ has support in $[K, \infty)\cap\ZZ$ and
\begin{equation}
x\coloneqq\sum_{i\in\ZZ} a_x(i)\,p^{i}.
\end{equation}
If $x$ is nonzero, then the $p$-adic valuation of $x$ is the integer $v_p(x)$, where \[v_p(x)=\min\{i\in\ZZ\colon a_x(i)\ne 0\}.\]  The $p$-adic valuation determines the $p$-adic absolute value by \[|x|_p=p^{-v_p(x)} \quad \text{and} \quad |0|_p = 0.\]  For any $x$ in $\Qp$, the \emph{fractional part} of $x$ is the rational number \[\{x\}_p = \sum_{k< 0}a_x(k)p^k.\]  Although $\{\cdot\}_p$ is not additive, it is constant on cosets of $\Zp$.  Take $\chi$ to be the character \[\chi(x) = \exp\big(2\pi \sqrt{-1}\{x\}_p\big), \quad (\forall x\in\Qp)\] which is trivial on $\Zp$ and nontrivial on $p^{-1}\Zp$.  For any character $\psi$ on $\Qp$, there is a unique $y$ in $\Qp$ so that for any $x$ in $\Qp$,
\begin{equation}
\psi(x) = \chi(xy),
\end{equation}
which identifies $\Qpdual$ with $\Qp$ as objects in $\LCAG$.

Take $\mu$ to be the Haar measure on $(\Qp,+)$, normalized to give $\Zp$ unit mass. Translation invariance of $\mu$ together with the fact that the residue field $\Zp\slash p\Zp$ has $p$ elements implies that for any $i$ in $\ZZ$,
\begin{equation}\label{eq:Ball1DMeasure}
\mu(x+\Ball(i))=p^{\,i},\qquad
\mu(x+\Shell(i))=\mu((x+\Ball(i))\smallsetminus(x+\Ball(i-1)))=p^{i}(1-p^{-1}).
\end{equation}
Henceforth, for any subset $A$ of any $G$ in $\LCAG$ and any specified Haar measure $\mu_G$ on $G$, take $\Vol(A)$ to mean $\mu_G(A)$, and for any $g$ in $G$, compress notation by writing $\d g$ rather than $\d\mu_G(g)$ whenever there is no ambiguity in doing so.

The Fourier transform $\Fourier$ and its inverse $\invFourier$ are the unitary transforms on $L^{2}(\Qp)$ whose restrictions to $L^{1}(\Qp)\cap L^{2}(\Qp)$ are given by 
\begin{equation}
\Fourier f(y)\coloneqq\int_{\Qp}\chi(-xy)f(x)\,\d x \quad \text{and} \quad \invFourier g(x)\coloneqq\int_{\Qp}\chi(xy)g(y)\,\d y.
\end{equation}
Use $\chi$ to identify $\Ball(i)^{\perp}$ as a ball in $\Qp$, so that
\begin{equation}\label{Eq:dim1:AnnBallisBall}
\Ball(i)^{\perp}=\{y\in\Qp\colon \chi(xy)=1\; \text{ for all }x\in\Ball(i)\} = \Ball(-i).
\end{equation}

Lemma~\ref{lem:LCAG:Int} implies that the Fourier transform takes indicator functions on centered balls to weighted indicator functions on centered balls, and \eqref{Eq:dim1:AnnBallisBall} implies that the family of centered balls 
\begin{equation}
\Pon(\Qp)\coloneqq\{\Ball(i)\colon i\in\ZZ\}
\end{equation} %
forms a \emph{Pontryagin filtration} of $\Qp$ with respect to the isomorphism from $\Qp$ to its dual given by the (rank-$0$) character $\chi$.  The index factor of $\Pon(\Qp)$ is $p$, the order of the residue field $\Zp\slash p\Zp$.

%%%%%%%%%%%%%%%%%%%%%%%%%%%%%%%%%%%%%%%%%%%%%%%%%%%%%%%%%%%%%%%%%%%%%%
%%%%%%%%%%%%%%%%%%%%%%%%%%%%%%%%%%%%%%%%%%%%%%%%%%%%%%%%%%%%%%%%%%%%%%

\subsection{Quadratic field extensions of $\Qp$}

%%%%%%%%%%%%%%%%%%%%%%%%%%%%%%%%%%%%%%%%%%%%%%%%%%%%%%%%%%%%%%%%%%%%%%
%%%%%%%%%%%%%%%%%%%%%%%%%%%%%%%%%%%%%%%%%%%%%%%%%%%%%%%%%%%%%%%%%%%%%%

In the course of developing harmonic analysis on $\Qp^2$, first treat the case of a quadratic field extension $\Kfield\slash\Qp$; the field case fixes certain choices and makes the general theory transparent. In the field case, the multiplicative structure fixes two ingredients: an additive character via the trace pairing and, from the valuation, a canonical Pontryagin filtration compatible with the character. Transporting these structures through a suitable $\Qp$-linear isomorphism $T$ from $\Qp^2$ to $\Kfield$ yields on $\Qp^2$ an induced character and Pontryagin filtration; in the unramified case, the sets in the filtration are max-norm balls (with constant index factor $p^2$), while in the totally ramified case, they become weighted max-norm balls (with constant index factor $p$).

View $\Kfield$ as a vector space over $\Qp$ and any $a$ in $\Kfield$ as indicating the $\Qp$-linear transformation from $\Kfield$ to itself given by multiplication by $a$, so that it makes sense to define $\det(a)$ and $\tr(a)$ as the determinant and trace of that linear transformation, respectively.  Define
\[
|a|_\Kfield \coloneqq \bigl|\det(a)\bigr|_p^{\frac{1}{2}}
\]
to be the absolute value of $a$ in $\Kfield$.  This extends the absolute value $|\cdot|_p$ on $\Qp$, since $\det(a)$ is equal to $a^2$ for any $a$ in $\Qp$. The centered unit ball $\{x\in\Kfield\colon|x|_\Kfield\le1\}$ is the maximal compact subring of $\Kfield$, called the \emph{ring of integers} of $\Kfield$ and denoted $\OK$; normalize the additive Haar measure $\mu_{\Kfield}$ on $\Kfield$ so that $\OK$ has mass $1$. The unique maximal ideal of $\OK$ is
\[
\PK \coloneqq \{x\in\Kfield\colon|x|_\Kfield<1\},
\]
and the residue field $\OK\slash\PK$ is finite. Any element $\beta$ with \[\beta\OK=\PK\] is a \emph{uniformizer}.  

For quadratic extensions there are two cases. In the \emph{unramified} case, the residue field has size $p^2$ and $p$ is a uniformizer; in particular \[|\beta|_\Kfield=p^{-1}\] for any uniformizer $\beta$.  In the \emph{totally ramified} case, the residue field has size $p$ and any uniformizer $\beta$ has absolute value \[|\beta|_\Kfield=p^{-1/2}, \quad \text{hence}\quad |\beta^{-1}|_\Kfield=p^{1/2}.\]  The two possible value groups $|\Kfield^\times|_\Kfield$ for a quadratic field extension $\Kfield$ are $p^{\ZZ}$ in the unramified case, in which case it is unchanged from $|\Qp^\times|_p$, and $p^{\frac{1}{2}\ZZ}$ in the totally ramified case. 

Fix a uniformizer $\beta$ and a complete set of coset representatives $\mathcal{U}$ of $\OK \slash \PK$ in $\OK$.  Any element $x$ of $\Kfield$ is in $\beta^k \OK$ for some integer $k$.  As in the $\Qp$ setting, there is a unique coefficient function $a_x$ so that \[x = \sum_{i\in\ZZ} a_x(i) \beta^i,\] where the coefficients $a_x(i)$ are in $\mathcal{U}$ and there is an integer $K$ so $a_x$ has support in $[K, \infty)$.  In the example where $\Kfield$ is $\Qp$, the ring of $p$-adic integers $\Zp$ is $\OK$, the prime ideal is $p \Zp$, the residue field is the field with $p$ elements, the prime $p$ is a uniformizer, and the coset representatives may be taken to be the rational numbers \[\mathcal{U} = \{0, \ldots, p - 1\}.\]

Denote by $\dualKfield$ the Pontryagin dual of the additive group of $\Kfield$.  Fix the previously defined rank-$0$ base character $\chi$ on $\Qp$ and choose $\theta$ in $\Kfield^\times$ so that the additive character $\psi$ on $\Kfield$ given by
\[
\psi(x) = \chi\big(\tr(\theta x)\big)
\]
has $\OK$ as its kernel (rank-$0$). 
For any $y$ in $\Kfield$, define the isomorphism from $\Kfield$ to $\dualKfield$ that takes $y$ to the character $\psi_y$, defined for all $x$ in $\Kfield$ by \[\psi_y(x) = \psi(xy).\] %

%%%%%%%%%%%%%%%%%%%%%%%%%%%%%%%%%%%%%%%%%%%%%%%%%%%%%%%%%%%%%%%%%%%%%%
%%%%%%%%%%%%%%%%%%%%%%%%%%%%%%%%%%%%%%%%%%%%%%%%%%%%%%%%%%%%%%%%%%%%%%

\subsection{The unramified case}

%%%%%%%%%%%%%%%%%%%%%%%%%%%%%%%%%%%%%%%%%%%%%%%%%%%%%%%%%%%%%%%%%%%%%%
%%%%%%%%%%%%%%%%%%%%%%%%%%%%%%%%%%%%%%%%%%%%%%%%%%%%%%%%%%%%%%%%%%%%%%

In the case of an unramified field extension of $\Qp$, there is an element $\pi$ in $\Kfield$ so that for some unit $u$ in $\Zp^\times$, \[\pi^2 = u, \quad \text{and} \quad \Kfield = \Qp(\pi).\] The set $\{1,\pi\}$ is a $\Zp$-basis of $\OK$.  

Use the notation $\ve{x}$ to identify an ordered pair \[\ve{x} = (x_1, x_2)\] in $\Qp^2$, and take $T$ to be the $\Qp$-linear isomorphism \[T:\Qp^{2}\to\Kfield, \quad\text{given by} \quad T(x_1,x_2)=x_1+x_2\pi.\]  %
For any $a$ and $b$ in $\Qp$, %
\begin{equation}\label{trace}
\tr(a+b\pi) = \tr\left(\begin{pmatrix}
a & 0\\
0 & a
\end{pmatrix}+\begin{pmatrix}
0 & bu\\
b & 0
\end{pmatrix}\right) = 2a.
\end{equation} %

Take $\theta$ to be $\frac{1}{2}u^{-1}\pi$. For any $(a,b)$ in $\Qp^2$, 
\begin{align*}
\tr(\theta(a+b\pi)) &= \tfrac{1}{2}\tr(u^{-1}\pi(a+b\pi))\\
 &= \tfrac{1}{2}\tr(b+u^{-1}a\pi) = b,
\end{align*}
hence \[\psi(a+b\pi) = \chi\big(\tr(\theta(a+b\pi))\big) = \chi(b).\]

Since $\tr(\theta x)$ is in $\Zp$ for $x$ in $\OK$ but not for all $x$ in $p^{-1}\OK$, the character $\psi$ is trivial on $\OK$ and nontrivial on $p^{-1}\OK$.  It induces a dual pairing $\langle \blank, \blank\rangle$ on $\Qp^2$ in the following way:  for any $(a,b)$ and $(c,d)$ in $\Qp^2$, 
\begin{equation}\label{eq:Qp2dualPairing:UnRam}
\langle (a,b), (c,d)\rangle = \psi(T(a,b)T(c,d)) = \chi(ad+bc).
\end{equation}

Define the max-norm on $\Qp^{2}$ by
\[
\|(x_1,x_2)\|_{\max}\coloneqq \max\big(|x_1|_p,|x_2|_p\big),
\]
and parameterize its associated centered balls by their volume, so that
\[
\Ballt(2k)\coloneqq p^{-k}\Zp\times p^{-k}\Zp = \Big\{(x_1,x_2)\in\Qp^{2}\colon \|(x_1,x_2)\|_{\max}\le p^{k}\Big\}.
\]
With the rank-$0$ additive character $\psi$, the annihilator of $T(\Ballt(2k))$ is the dual ball
\[
T(\Ballt(2k))^\perp=T(p^{k}\Zp\times p^{k}\Zp)=T(\Ballt(-2k)),
\]
and so the collection of centered balls \[\Pon_{\max}(\Kfield) = \big\{T(\Ballt(2k))\colon k\in\ZZ\big\}\] is a Pontryagin filtration with respect to the isomorphism given by the character $\psi$ when $\Kfield$ is an unramified quadratic field extension of $\Qp$.

%%%%%%%%%%%%%%%%%%%%%%%%%%%%%%%%%%%%%%%%%%%%%%%%%%%%%%%%%%%%%%%%%%%%%%
%%%%%%%%%%%%%%%%%%%%%%%%%%%%%%%%%%%%%%%%%%%%%%%%%%%%%%%%%%%%%%%%%%%%%%

\subsection{The totally ramified case}

%%%%%%%%%%%%%%%%%%%%%%%%%%%%%%%%%%%%%%%%%%%%%%%%%%%%%%%%%%%%%%%%%%%%%%
%%%%%%%%%%%%%%%%%%%%%%%%%%%%%%%%%%%%%%%%%%%%%%%%%%%%%%%%%%%%%%%%%%%%%%

In the case where $\Kfield$ is a totally ramified quadratic field extension of $\Qp$, choose a uniformizer $\beta$ with \[|\beta|_\Kfield=p^{-\frac{1}{2}}\] and identify $\Qp^2$ with $\Kfield$ via the basis $\{1, \beta\}$.   Take $T$ to be the $\Qp$-linear isomorphism \[T:\Qp^{2}\to\Kfield, \quad\text{given by} \quad T(x_1,x_2)=x_1+x_2\beta.\]  There is a unit $u$ in $\Zp^\times$ so that $\beta^2$ is equal to $up$. Take $\theta$ to be $\frac{1}{2}\beta^{-1}$ so that, as in the unramified case, for any $(a,b)$ in $\Qp^2$, 
\[
\psi(a+b\beta) = \chi(b).
\]

Since $\tr(\theta x)$ is in $\Zp$ for $x$ in $\OK$ but not for all $x$ in $\beta^{-1}\OK$, the character $\psi$ is trivial on $\OK$ and nontrivial on $\beta^{-1}\OK$ (rank-$0$).  It induces a dual pairing $\langle \blank, \blank\rangle$ on $\Qp^2$ in the following way:  for any $(a,b)$ and $(c,d)$ in $\Qp^2$,
\begin{equation}\label{eq:Qp2dualPairing:TotRam}
\langle (a,b), (c,d)\rangle = \psi(T(a,b)T(c,d)) = \chi\big(ad+bc\big).
\end{equation}

Define the weighted max-norm on $\Qp^{2}$ by
\[
\|(x_1,x_2)\|_{\wt}\coloneqq \max\big(|x_1|_p,p^{-\frac{1}{2}}|x_2|_p\big).
\]
For any $k$ in $\No$ and $\varepsilon$ in $\{0,1\}$, parameterize the associated weighted max-norm balls by their volume, so that
\[
\Ballt(2k+\varepsilon) \coloneqq p^{-k}\big(\Zp\times p^{-\varepsilon}\Zp\big) = \Big\{\ve{x}\in \Qp^2\colon \|\ve{x}\| \leq p^{k+\frac{\varepsilon}{2}}\Big\}.\]  The transformation $T$ takes balls $\Ballt(1)$ and $\Ballt(-1)$, given by \[\Ballt(1) = \Zp\times p^{-1}\Zp \quad \text{and} \quad \Ballt(-1) = p\Zp\times \Zp = p\big(\Zp\times p^{-1}\Zp\big),\] to dual pairs, and \[T(\Ballt(0)) = T(\Zp \times \Zp) = T(\Ballt(0))^\perp.\]  Furthermore, for any $\varepsilon$ in $\{0,1\}$ and any $k$ and $\ell$ in $\ZZ$, \[T\big(p^k\Ballt(2\ell+\varepsilon)\big)^\perp = T\big(p^{-k}\big(\BallK(2\ell+\varepsilon)^\perp\big)\big),\] and so the collection of balls \[\Pon_{\wt}(\Kfield) \coloneqq \Big\{T\big(\Ballt\big(2k+\varepsilon\big)\big)\colon k\in\ZZ, \varepsilon\in\{0,1\}\Big\}\] is a Pontryagin filtration with respect to the isomorphism given by the character $\psi$ when $\Kfield$ is a totally ramified quadratic field extension of $\Qp$.

%%%%%%%%%%%%%%%%%%%%%%%%%%%%%%%%%%%%%%%%%%%%%%%%%%%%%%%%%%%%%%%%%%%%%%
%%%%%%%%%%%%%%%%%%%%%%%%%%%%%%%%%%%%%%%%%%%%%%%%%%%%%%%%%%%%%%%%%%%%%%

\subsection{Harmonic analysis on $\Qp^2$}

%%%%%%%%%%%%%%%%%%%%%%%%%%%%%%%%%%%%%%%%%%%%%%%%%%%%%%%%%%%%%%%%%%%%%%
%%%%%%%%%%%%%%%%%%%%%%%%%%%%%%%%%%%%%%%%%%%%%%%%%%%%%%%%%%%%%%%%%%%%%%

The shape of balls and circles in $\Qp^2$ depends on the choice of metric, but the continuous additive characters on $\Qp^2$ depend only on the topology and group structure of $\Qp^2$. With this in mind, define on $\Qp^2$ for any $h$ in $(0,1]$ the norm 
\[
\|(x_1,x_2)\|_{h}\coloneqq \max\Big(|x_1|_p,p^{h-1}|x_2|_p\Big).
\]
In the case that $h$ is $1$, this is the max-norm.  Refer to it as the weighted max-norm when $h$ is not $1$.  It is helpful to make the distinction because of the structural difference between the two settings.

For any two pairs $(a,b)$ and $(c,d)$ in $\Qp$, define a dual pairing on $\Qp^2$ by \eqref{eq:Qp2dualPairing:UnRam} (or  \eqref{eq:Qp2dualPairing:TotRam}). Since the $\Qp$-linear transformation $T$ in either the ramified or unramified field setting transfers structure between the field extensions and $\Qp^2$, for any $k$ in $\ZZ$, if $h$ is in $(0,1]$, then \begin{equation}\label{Eq:DualBallnoe:identification:qp2}\Ballt(2k)^\perp = \Ballt(-2k) = p^{k}\Zp\times p^{k}\Zp.\end{equation}  If $h$ is in $(0,1)$, then \begin{equation}\label{Eq:DualBalle:identification:qp2}\Ballt(2k+1)^\perp = \Ballt(-2k-1) = p^{k+1}\Zp\times p^{k}\Zp.\end{equation}   Henceforth, take $h$ to be in $(0,1)$.  

Equalities \eqref{Eq:DualBallnoe:identification:qp2} and \eqref{Eq:DualBalle:identification:qp2} together guarantee that \[\Pon_{\rm c}(\Qp^2) = \{\Ballt(2k)\colon k\in \ZZ\}\quad \text{and}\quad \Pon_{\rm f}(\Qp^2) = \{\Ballt(2k+\varepsilon)\colon k\in \ZZ, \varepsilon\in\{0,1\}\}\] are both Pontryagin filtrations of $\Qp^2$ where the isomorphism between $\Qp^2$ and its Pontryagin dual is the one induced by the dual pairing $\langle \blank, \blank\rangle$.  Equip $(\Qp^2, \|\blank\|_{\max})$ with $\Pon_{\rm c}(\Qp^2)$, and $(\Qp^2, \|\blank\|_h)$ with $\Pon_{\rm f}(\Qp^2)$, when $h$ is not $1$.  These are, respectively, the \emph{coarse} and \emph{fine} Pontryagin filtrations of $\Qp^2$.  

It is useful to relate the volume parameter for identifying centered balls in $\Qp^2$ to the radius of the ball.  For any ball $B$ in $\Qp^2$ in either the max-norm or weighted max-norm settings, denote by $\rad(B)$ the radius of $B$.  Equations \eqref{Eq:DualBallnoe:identification:qp2} and \eqref{Eq:DualBalle:identification:qp2} together imply Proposition~\ref{Prop:DualityRadius:Continuous}.

\begin{proposition}\label{Prop:DualityRadius:Continuous}
In the case of the max-norm or weighted max-norm, \[\rad\!\big(\Ballt(2k)\big) = p^{k} \quad \text{and} \quad \rad\!\big(\Ballt(2k)^\perp\big) = p^{-k}.\]  In the case of the weighted max-norm, \[\rad\!\big(\Ballt(2k+1)\big) = p^{k+h} \quad \text{and} \quad \rad\!\big(\Ballt(2k+1)^\perp\big) = p^{-k+(h-1)}.\] 
\end{proposition}

Normalize the Haar measure on $(\Qp^2,+)$ to give $\Zp^2$ unit mass.  The Fourier transform $\Fourier$ and inverse Fourier transform $\invFourier$ are the unitary operators on $L^{2}(\Qp^2)$ defined, respectively, on $L^1(\Qp^2)\cap L^2(\Qp^2)$ by
\[
\Fourier f(\ve{y})\coloneqq\int_{\Qp^2}\langle-\ve{x},\ve{y}\rangle f(\ve{x})\,\d \ve{x} \quad \text{and} \quad \invFourier g(\ve{x})\coloneqq\int_{\Qp^2}\langle\ve{x},\ve{y}\rangle g(\ve{y})\,\d \ve{y}.
\]
Define for any $k$ in $\ZZ$ the \emph{coarse shell} $\Shelltc(2k)$ of the ball $\Ballt(2k)$ by \begin{equation*}\Shelltc(2k) = \Ballt(2k)\smallsetminus\Ballt(2k-2),\end{equation*} so that \begin{equation}\label{eq:CoarseShellVolume}\begin{cases}\Vol(\Ballt(2k)) = p^{2k}\\\Vol(\Shelltc(2k)) = (1-p^{-2})\Vol(\Ballt(2k)).\end{cases}\end{equation}
Define for any $k$ in $\ZZ$ and $\varepsilon$ in $\{0,1\}$ the \emph{fine shell} $\Shelltf(2k+\varepsilon)$ of the ball $\Ballt(2k+\varepsilon)$ by \[\Shelltf(2k+\varepsilon) = \Ballt(2k+\varepsilon)\smallsetminus\Ballt(2k+\varepsilon-1),\] so that \begin{equation}\label{eq:FineShellVolume}\begin{cases}\Vol(\Ballt(2k+\varepsilon)) = p^{2k+\varepsilon}\\\Vol(\Shelltf(2k+\varepsilon)) = (1-p^{-1})\Vol(\Ballt(2k+\varepsilon)).\end{cases}\end{equation}

%%%%%%%%%%%%%%%%%%%%%%%%%%%%%%%%%%%%%%%%%%%%%%%%%%%%%%%%%%%%%%%%%%%%%%
%%%%%%%%%%%%%%%%%%%%%%%%%%%%%%%%%%%%%%%%%%%%%%%%%%%%%%%%%%%%%%%%%%%%%%

\subsection{The single-component discrete groups}

%%%%%%%%%%%%%%%%%%%%%%%%%%%%%%%%%%%%%%%%%%%%%%%%%%%%%%%%%%%%%%%%%%%%%%
%%%%%%%%%%%%%%%%%%%%%%%%%%%%%%%%%%%%%%%%%%%%%%%%%%%%%%%%%%%%%%%%%%%%%%

For any $m$ in $\No$, denote by $\Gm$ the quotient group $\Qp\slash p^m\Zp$, endowed with the counting measure $\mu_m$ that assigns mass $p^{-m}$ to each singleton set.  For any $x$ in $\Qp$, write
\[
[x]_m = x+p^m\Zp.
\]
Define the absolute value on $\Gm$ by
\[
|[x]_m|=
\begin{cases}
|x|_p & \text{if } [x]_m\neq[0]_m,\\[2pt]
0,& \text{if }  [x]_m=[0]_m,
\end{cases}
\]
and the metric by \[d(g_1,g_2) = |g_1-g_2|.\] For any $k$ in $\{-m,-m+1, \dots\}$, define the centered balls and their shells by
\[
\Ballm(k) = \{g\in \Gm\colon |g|\le p^k\} \quad \text{and} \quad \Shellm(k) = \{g\in \Gm\colon |g|=p^k\},
\]
with %
\[
\Ballm(-m)=\Shellm(-m)=\{[0]_m\}.\] %
Volumes with respect to $\mu_m$ are
\[
\vol\big(\Ballm(k)\big)=p^{k} \quad \text{and} \quad \vol\big(\Shellm(k)\big)=p^{k}\left(1-\tfrac{1}{p}\right).
\]

The pairing $\langle \cdot, \cdot \rangle_m$ on $\Gm \times p^{-m}\Zp$ given by %
\begin{equation}\label{eq:pairing-1d}
\langle [x]_m,y\rangle_m = \chi(x y)%
\end{equation} %
identifies the Pontryagin dual $\widehat{\Gm}$ with $p^{-m}\Zp$.  Local constancy of $\chi$ ensures that \eqref{eq:pairing-1d} is independent of choice of representative and Proposition~\ref{Prop:DualityScaling:DiscreteA} gives the relationship between balls and dual balls under the given dual pairing and implies that the collections of balls %
\[
\Pon\big(\Gm\big) = \{\Ballm(i)\colon i\in -m+\mathbb N_0\} \quad \text{and} \quad \Pon\big(\widehat{\Gm}\big) = \{p^{-m+i}\Zp\colon i\in\mathbb N_0\}.
\]
form a Pontryagin filtration for $\Gm$.  

\begin{proposition}\label{Prop:DualityScaling:DiscreteA} 
For any $k$ and $m$ in $\No$, \[\Big(\Ballm(k)\Big)^\perp = p^{k-m}\Zp \quad \text{and} \quad \Big(p^{k-m}\Zp\Big)^\perp = \Ballm(k).\] 
\end{proposition}

Define
\[
Q_m\colon \G\to \Gm \quad \text{by} \quad Q_m\big(x+\Zp\big)=p^m x+p^m\Zp,
\]
and
\[
\widehat Q_m\colon \Zp\to p^{-m}\Zp  \quad \text{by} \quad  \widehat{Q}_m(x) = p^{-m}x\]

\begin{proposition}
The product $Q_m\times \widehat Q_m$ takes the Pontryagin filtration $\big(\!\Pon(\G),\Pon(\Zp)\big)$ to $\big(\!\Pon(\Gm),\Pon(p^{-m}\Zp)\big)$.
\end{proposition}

The Fourier transform $\mathcal F_m$ takes $L^2(\Gm)$ to $L^2(p^{-m}\Zp)$ and the inverse Fourier transform $\mathcal F^{-1}_m$ takes $L^2(p^{-m}\Zp)$ to $L^2(\Gm)$.  They are given for any $f$ in $L^2(\Gm)$ and any $g$ in $L^2(p^{-m}\mathds Z_p)$ by \[(\mathcal F_mf)(y) = \int_{\Gm}\langle -[x]_m,y\rangle_m f([x]_m)\,{\rm d}[x]_m \quad \text{and} \quad (\mathcal F_m^{-1}g)\left([x]_m\right) = \int_{p^{-m}\mathds Z_p}\langle [x]_m,y\rangle_m g(y)\,{\rm d}y.\]

Balls in $\Qp$ of radius at least $p^{-m}$ decompose as disjoint unions of cosets of $p^m\Zp$, and the map \[x\mapsto [x]_m\] is constant on each such coset. These facts imply Lemma~\ref{StateSpaces:Intmoverballs}.

\begin{lemma}\label{StateSpaces:Intmoverballs}
For any integrable function $f$ from $\Gm$ to $\mathds{C}$ and any ball $B$ in $\Qp$ of radius at least $p^{-m}$,
\[
\int_{[B]_m} f([x]_m)\,\d[x]_m\;=\;\int_{B} f([x]_m)\,\d x.
\]
\end{lemma}

%%%%%%%%%%%%%%%%%%%%%%%%%%%%%%%%%%%%%%%%%%%%%%%%%%%%%%%%%%%%%%%%%%%%%%
%%%%%%%%%%%%%%%%%%%%%%%%%%%%%%%%%%%%%%%%%%%%%%%%%%%%%%%%%%%%%%%%%%%%%%

\subsection{The discrete groups with two components}

%%%%%%%%%%%%%%%%%%%%%%%%%%%%%%%%%%%%%%%%%%%%%%%%%%%%%%%%%%%%%%%%%%%%%%
%%%%%%%%%%%%%%%%%%%%%%%%%%%%%%%%%%%%%%%%%%%%%%%%%%%%%%%%%%%%%%%%%%%%%%

For any $m$ to be in $\No$ denote by $(\Gm)^2$ the group \[(\Gm)^2 = (\Qp\slash p^m\Zp)\times (\Qp\slash p^m\Zp),\] endowed with component-wise addition.  Given any \[\ve{x} = (x_1,x_2)\in\Qp^2 \quad \text{and}  \quad \ve{y}=(y_1,y_2)\in(p^{-m}\Zp)^2,\] define the dual pairing by
\begin{equation}\label{eq:pairing-2d}
\langle \x_m,\ve{y}\rangle_m = \chi\big(x_1 y_1+x_2 y_2\big),
\end{equation}
which identifies $\widehat{(\Gm)^2}$ with $(p^{-m}\Zp)^2$ equipped with the Haar measure induced by inclusion in $\Qp^2$. 

For $\varepsilon$ in $\{0,1\}$ and $k$ in $\No$, set
\[
\BBtm(2k+\varepsilon) = \big(p^{-k}\Zp\slash p^m\Zp\big)\times\big(p^{-k-\varepsilon}\Zp\slash p^m\Zp\big).
\]
Define coarse and fine shells, respectively, by
\[
\SStmc(2k):=\BBtm(2k)\smallsetminus \BBtm(2k-2) \quad \text{and} \quad \SStmf(2k):=\BBtm(2k)\smallsetminus \BBtm(2k-1),
\]
and for odd indices,
\[
\SStmf(2k+1):=\BBtm(2k+1)\smallsetminus \BBtm(2k).
\]
Finally, \[\BBtm(-2m)=\SStmc(-2m)=\SStmf(-2m)=\{[0]_m\}.\] Volumes of these balls and shells are given by %
\begin{equation}\label{eq:BallmVol-2d}
\begin{cases}
\Vol\!\big(\BBtm(2k+i)\big)=p^{2k+i}\\
\Vol\!\big(\SStmc(2k)\big)=(1-p^{-2})\,\Vol\!\big(\BBtm(2k)\big)\\
\Vol\!\big(\SStmf(2k+i)\big)=(1-p^{-1})\,\Vol\!\big(\BBtm(2k+i)\big),
\end{cases}
\end{equation}
for $i$ in $\{0,1\}$ and $k$ in $-m+\mathbb N_0$.

Two distinct Pontryagin filtrations for $\big((\Gm)^2, (p^{-m}\Zp)^2\big)$ are of primary significance, the coarse filtration $\big(\!\Pon_{\rm c}((\Gm)^2), \Pon_{\rm c}(p^{-m}\Zp)^2\big)$ with \[\Pon_{\rm c}((\Gm)^2) = \{\BBtm(2k)\colon k\in \No\} \quad \text{and} \quad \Pon_{\rm c}(p^{-m}\Zp)^2 = \{p^{-m+2k}\Zp^2\colon k\in \No\},\] and the fine filtration $\big(\!\Pon_{\rm f}((\Gm)^2), \Pon_{\rm f}(p^{-m}\Zp)^2\big)$ with \[\Pon_{\rm f}((\Gm)^2) = \Big\{\BBtm(2k+\varepsilon)\colon k\in \No, \varepsilon\in \{0,1\}\Big\}\] and \[\Pon_{\rm f}\big((p^{-m}\Zp)^2\big) = \Big\{p^{-m+k}\Zp^2\colon k\in \No\Big\} \cup \Big\{p^{-m+k}\Zp\times p^{-m+k-1}\Zp\colon k\in \NN\Big\}.\] %
Proposition~\ref{Prop:DualityScaling:DiscreteB} implies that the given coarse and fine filtrations are, indeed, Pontryagin filtrations for $\big((\Gm)^2, p^m\Zp^2\big)$.

\begin{proposition}\label{Prop:DualityScaling:DiscreteB} 
Take $\perp$ to be the annihilator operation that the pairing \eqref{eq:pairing-2d} defines. For any $k$ in $\No$, \[\Big(\BBtm(2k)\Big)^\perp = \Ballt(2(m-k)) = p^{-m+k}\Zp\times p^{-m+k}\Zp\] and\[\big(\BBtm(2k+1)\big)^\perp = \Ballt(2(m-k)-1) = p^{-m+k+1}\Zp \times p^{-m+k}\Zp.\] 
\end{proposition}

Define
\[
Q_m^{(2)}\colon \G^2\to (\Gm)^2 \quad \text{by} \quad Q_m^{(2)}\big(\mathbf{x}+\Zp^2\big) = p^m\mathbf{x}+p^m\Zp^2
\]
and
\[
\widehat Q_m^{(2)}\colon \Zp^2\to \big(p^{-m}\Zp\big)^2 \quad \text{by} \quad \widehat{Q}_m^{(2)}(x_1,x_2) = \big(p^{-m}x_1,p^{-m}x_2\big).
\]

\begin{proposition}
The product $Q_m^{(2)}\times \widehat Q_m^{(2)}$ takes the coarse and fine Pontryagin filtrations of $\big(\G^2,\Zp^2\big)$ to those of $\big((\Gm)^2\times(p^{-m}\Zp)^2\big)$.
\end{proposition}

Balls in $\Qp^2$ of volume at least $p^{-2m}$ decompose as disjoint unions of cosets of balls of radius $p^{-2m}$, and the map \[\ve{x}\mapsto [\ve{x}]_m\] is constant on each such coset. These facts imply Lemma~\ref{StateSpaces:Intmoverballs2}, the analog of Lemma~\ref{StateSpaces:Intmoverballs} in the two-dimensional setting.

\begin{lemma}\label{StateSpaces:Intmoverballs2}
For any integrable function $f$ from $(\Gm)^2$ to $\mathds{C}$ and any ball $B$ in $\Qp^2$ of volume at least $p^{-m}$,
\[
\int_{[B]_m} f([\ve{x}]_m)\,\d[\ve{x}]_m=\int_{B} f([\ve{x}]_m)\,\d\ve{x}.
\]
\end{lemma}

%%%%%%%%%%%%%%%%%%%%%%%%%%%%%%%%%%%%%%%%%%%%%%%%%%%%%%%%%%%%%%%%%%%%%%
%%%%%%%%%%%%%%%%%%%%%%%%%%%%%%%%%%%%%%%%%%%%%%%%%%%%%%%%%%%%%%%%%%%%%%

\subsection{Path measures}

%%%%%%%%%%%%%%%%%%%%%%%%%%%%%%%%%%%%%%%%%%%%%%%%%%%%%%%%%%%%%%%%%%%%%%
%%%%%%%%%%%%%%%%%%%%%%%%%%%%%%%%%%%%%%%%%%%%%%%%%%%%%%%%%%%%%%%%%%%%%%

Take $I$ to be either $[0,\infty)$ or $\No$. For any Polish space $\mathcal S$, denote by $F(I\colon\mathcal S)$ the set of all $\mathcal S$-valued paths with domain $I$, and by $D([0,\infty)\colon\mathcal S)$ the Skorokhod space of $\mathcal S$-valued c\`adl\`ag paths equipped with the $J_1$ topology~\cite{bil1}. Write $\Omega(I)$ to indicate either $F(I\colon\mathcal S)$ or $D([0,\infty)\colon\mathcal S)$ and denote by $\mathcal B(\mathcal S)$ the Borel $\sigma$-algebra on $\mathcal S$.

An \emph{epoch} of length $N$ in $I$ is a strictly increasing sequence
\[
e\colon \{0,\dots,N\}\to I \quad \text{with} \quad e(0)=0.
\]
A \emph{route} of length $N$ is a sequence
\[
U\colon\{0,\dots,N\}\to \mathcal B(\mathcal S).
\]
A \emph{history} is a sequence \[h=\big((t_i,U_i)\big)_{i=0}^N\] with $(t_i)$ an epoch and $(U_i)$ a route. For any history $h$, the associated \emph{simple cylinder set} in $\Omega(I)$ is
\[
\C(h)=\big\{\omega\in\Omega(I)\colon \omega(t_i)\in U_i,\;\forall i\in\{0,\dots,N\}\big\}.
\]

Specification of a simple cylinder set always involves an identification of an ambient path space. The $\sigma$-algebra generated by the simple cylinder sets is the \emph{cylinder $\sigma$-algebra}. Define $Y$ on $I\times \Omega(I)$ by \eqref{Intro:Yt} and $Y_t$ on $\Omega(I)$ by \eqref{Intro:Ydef}. For any probability measure $\Pb$ on the cylinder $\sigma$-algebra, $(\Omega(I),\Pb,Y_t)$ is a random variable and $(\Omega(I),\Pb,Y)$ is a stochastic process. Finite-dimensional distributions are the probabilities of simple cylinder sets. The convergence of processes that follows concerns processes with fixed state and parameter spaces but varying probability measures.

Assume now that $\mathcal S$ is an object of $\LCAG$ endowed with a Haar measure $\mu$. If $\{\rho(t,\cdot)\}_{t>0}$ is a convolution semigroup of probability densities on $\mathcal S$ with respect to $\mu$, then the standard construction assigns probabilities to simple cylinder sets by
\begin{align}\label{eq:ProbDetermination}
&\Pb\big(Y_{t_0}\in U_0,\dots,Y_{t_N}\in U_N\big)\\
& \hspace{.25in} = \begin{cases}
\displaystyle \int_{U_1}\!\cdots\!\int_{U_N}\rho\big(t_1, x_1\big)\prod_{i=2}^{N}\rho\big(t_i-t_{i-1},\,x_i-x_{i-1}\big)\,\d x_1\cdots\d x_N, & \text{if }0\in U_0,\\[8pt]
0, & \text{otherwise},
\end{cases}\notag
\end{align}
where $0$ is the group identity. The Kolmogorov extension theorem yields a probability measure $\Pb$ on $D([0,\infty)\colon\mathcal S)$ concentrated on paths started at the identity. When $\rho(t,\cdot)$ is the heat kernel associated with the Vladimirov operator on $\Qp$ (see \eqref{subsec:LimitProcess1D}), the triple $(D([0,\infty)\colon\Qp),\Pb,Y)$ is a $p$-adic Brownian motion.

The paper specifies random variables and processes by prescribing their finite-dimensional distributions, without constructing an underlying probability space at each step. In all cases that arise here, the Kolmogorov extension theorem applies and yields a concrete model (and, when relevant, a c\`adl\`ag version on the Skorokhod space). A reader who prefers explicit constructions can first verify the Kolmogorov consistency conditions and then invoke the theorem.

%%%%%%%%%%%%%%%%%%%%%%%%%%%%%%%%%%%%%%%%%%%%%%%%%%%%%%%%%%%%%%%%%%%%%%
%%%%%%%%%%%%%%%%%%%%%%%%%%%%%%%%%%%%%%%%%%%%%%%%%%%%%%%%%%%%%%%%%%%%%%

\subsection{Approximation of path measures}

%%%%%%%%%%%%%%%%%%%%%%%%%%%%%%%%%%%%%%%%%%%%%%%%%%%%%%%%%%%%%%%%%%%%%%
%%%%%%%%%%%%%%%%%%%%%%%%%%%%%%%%%%%%%%%%%%%%%%%%%%%%%%%%%%%%%%%%%%%%%%

Henceforth take $m$ to vary in $\No$ and $G$ to be a discrete, countably infinite, additive abelian group. Take $X$ to be a $G$-valued random variable, and $(X_i)_{i\ge1}$ to be a sequence of independent random variables, each with the same distribution as $X$. Define the partial sums
\begin{equation}\label{eq:primitive-sum}
S_n = X_0+ X_1+\cdots+ X_n,
\end{equation}
where $X_0$ is almost surely equal to the additive group identity.  The law for the random variable $S_n$ may be computed directly from the formula for $S_n$, as the $n$-fold convolution of the law for $X$.  For each element $g$ of $G$, denote by $\Prob(S_n = g)$ the probability that $S_n$ is equal to $g$.

For any history $h$ given by \[h = ((0,U_0),(n_1,U_1),\dots,(n_N,U_N)),\] define finite dimensional distributions by
\begin{align}\label{eq:PremeasureFromAbstract}
\PbG(\C(h)) =\sum_{x_0\in U_0}\cdots\sum_{x_N\in U_N}\Prob(S_{0}=x_0)\prod_{i=1}^{N}\Prob\big(S_{n_i-n_{i-1}}=x_i-x_{i-1}\big).
\end{align}
For each $n$, take \[S\colon \No\times F(\mathbb N_0\colon G)\to G \quad \text{by} \quad S(n,\omega)=\omega(n),\] so that $S(n,\cdot)$ is the random variable $S_n$.  The Kolmogorov extension theorem yields a measure $\PbG$ on the $\sigma$--algebra of cylinder sets of $F(\mathbb N_0\colon G)$.  The triple $(F(\mathbb N_0\colon G), \PbG, S)$ is a \emph{primitive random walk} on $G$ with \emph{generator} $X$.

\begin{definition}
Take $\mathcal S$ to be an object in $\LCAG$ that is additionally a complete normed vector space with norm $\|\cdot\|$. For any positive null sequence $(\delta_m)$, a sequence $(\Gamma_m)$ of injective functions from $G$ to $\mathcal S$ that satisfies these properties is a sequence of \emph{spatial embeddings}: 
\begin{enumerate}
\item[(1)] for any $x$ in $\mathcal S$ there is a $g$ in $G$ so that 
\[
0<\|\Gamma_m(g)-x\|\le \delta_m;
\]
\item[(2)] for any $g_1$ and $g_2$ in $G$, 
\[
\|\Gamma_m(g_1)-\Gamma_m(g_2)\|<\delta_m\quad \text{implies that}\quad  g_1=g_2.
\]
\end{enumerate}
\end{definition}

\begin{definition}
For any positive null sequence $(\tau_m)$, a sequence $(\iota_m)$ of \emph{spatiotemporal embeddings} of $\mathbb N_0\times G$ into $[0,\infty)\times\mathcal S$ with time scale $\tau$ and spatial embeddings $\Gamma$ is given by
\[
\iota_m\colon \mathbb N_0\times G\to[0,\infty)\times\mathcal S\qquad \text{by} \quad \iota_m(n,g)=\big(n\tau_m, \Gamma_m(g)\big).
\]
\end{definition}

Following earlier work \cite{W:Expo:24}, for any $\iota_m$, any  history $h$ for paths in $F(\mathbb N_0\colon G)$, and any $i$ in $\mathds N_0\cap [0, \ell(h)]$, initially define a history $h^m$ for paths in $F(\mathds N_0\colon G)$ by \[h^m(i) = \left(\floor*{\tfrac{e_h(i)}{\tau_m}}, \Gamma_m^{-1}(U_h(i))\right).\]  The sequence of time points for $h^m$ may fail to be an epoch since distinct time points that are sufficiently close may map to the same time point.  In such cases, remove all repeated instances of a time point except one, along with their corresponding places in the route. Replace the remaining value of the route, which corresponds to the remaining instance of the repeated time point, with the intersection of all values of the route at the repeated time point.  

The functions $\iota_m$ and the process $(F(\mathbb N_0\colon G), \PbG, S)$ together assign probabilities to the simple cylinder sets of $F([0,\infty)\colon \mathcal S)$ in the following way:  for any history $h$ for paths in $F([0,\infty)\colon V)$, \begin{equation}\label{EmbeddedProcessFDD}\Pbm(\C(h)) = \PbG(\C(h^m)).\end{equation} Since the set of simple cylinder sets is a $\pi$-system that generates the $\sigma$-algebra of cylinder sets, the measure defined by extending these probabilities to the $\sigma$-algebra is unique.

The triple $(F([0,\infty)\colon \mathcal S), \Pbm, Y)$ is a stochastic process with state space $\mathcal S$.  For appropriate choices of primitive process and sequence of spatiotemporal embedding, there will be for each $m$ a version of the process in $D([0,\infty)\colon \mathcal S)$.  Continue to write $\Pbm$ for the measure on this space that is induced by the spatiotemporal embedding.  If the sequence $(\Pbm)$ converges weakly to the probability measure $\Pb$ of a Brownian motion on $\mathcal S$, then the Brownian motion on $\mathcal S$ arises as a \emph{scaling limit} of a primitive discrete-time random walk on $G$.

%%%%%%%%%%%%%%%%%%%%%%%%%%%%%%%%%%%%%%%%%%%%%%%%%%%%%%%%%%%%%%%%%%%%%%
%%%%%%%%%%%%%%%%%%%%%%%%%%%%%%%%%%%%%%%%%%%%%%%%%%%%%%%%%%%%%%%%%%%%%%
%%%%%%%%%%%%%%%%%%%%%%%%%%%%%%%%%%%%%%%%%%%%%%%%%%%%%%%%%%%%%%%%%%%%%%
%%%%%%%%%%%%%%%%%%%%%%%%%%%%%%%%%%%%%%%%%%%%%%%%%%%%%%%%%%%%%%%%%%%%%%
%%%%%%%%%%%%%%%%%%%%%%%%%%%%%%%%%%%%%%%%%%%%%%%%%%%%%%%%%%%%%%%%%%%%%%
%%%%%%%%%%%%%%%%%%%%%%%%%%%%%%%%%%%%%%%%%%%%%%%%%%%%%%%%%%%%%%%%%%%%%%
%%%%%%%%%%%%%%%%%%%%%%%%%%%%%%%%%%%%%%%%%%%%%%%%%%%%%%%%%%%%%%%%%%%%%%
%%%%%%%%%%%%%%%%%%%%%%%%%%%%%%%%%%%%%%%%%%%%%%%%%%%%%%%%%%%%%%%%%%%%%%
%%%%%%%%%%%%%%%%%%%%%%%%%%%%%%%%%%%%%%%%%%%%%%%%%%%%%%%%%%%%%%%%%%%%%%
%%%%%%%%%%%%%%%%%%%%%%%%%%%%%%%%%%%%%%%%%%%%%%%%%%%%%%%%%%%%%%%%%%%%%%
%%%%%%%%%%%%%%%%%%%%%%%%%%%%%%%%%%%%%%%%%%%%%%%%%%%%%%%%%%%%%%%%%%%%%%
%%%%%%%%%%%%%%%%%%%%%%%%%%%%%%%%%%%%%%%%%%%%%%%%%%%%%%%%%%%%%%%%%%%%%%

\section{The primitive random walk with one component and its scaling limit}\label{Sec:One-Comp}

%%%%%%%%%%%%%%%%%%%%%%%%%%%%%%%%%%%%%%%%%%%%%%%%%%%%%%%%%%%%%%%%%%%%%%
%%%%%%%%%%%%%%%%%%%%%%%%%%%%%%%%%%%%%%%%%%%%%%%%%%%%%%%%%%%%%%%%%%%%%%
%%%%%%%%%%%%%%%%%%%%%%%%%%%%%%%%%%%%%%%%%%%%%%%%%%%%%%%%%%%%%%%%%%%%%%
%%%%%%%%%%%%%%%%%%%%%%%%%%%%%%%%%%%%%%%%%%%%%%%%%%%%%%%%%%%%%%%%%%%%%%
%%%%%%%%%%%%%%%%%%%%%%%%%%%%%%%%%%%%%%%%%%%%%%%%%%%%%%%%%%%%%%%%%%%%%%
%%%%%%%%%%%%%%%%%%%%%%%%%%%%%%%%%%%%%%%%%%%%%%%%%%%%%%%%%%%%%%%%%%%%%%
%%%%%%%%%%%%%%%%%%%%%%%%%%%%%%%%%%%%%%%%%%%%%%%%%%%%%%%%%%%%%%%%%%%%%%
%%%%%%%%%%%%%%%%%%%%%%%%%%%%%%%%%%%%%%%%%%%%%%%%%%%%%%%%%%%%%%%%%%%%%%
%%%%%%%%%%%%%%%%%%%%%%%%%%%%%%%%%%%%%%%%%%%%%%%%%%%%%%%%%%%%%%%%%%%%%%
%%%%%%%%%%%%%%%%%%%%%%%%%%%%%%%%%%%%%%%%%%%%%%%%%%%%%%%%%%%%%%%%%%%%%%
%%%%%%%%%%%%%%%%%%%%%%%%%%%%%%%%%%%%%%%%%%%%%%%%%%%%%%%%%%%%%%%%%%%%%%
%%%%%%%%%%%%%%%%%%%%%%%%%%%%%%%%%%%%%%%%%%%%%%%%%%%%%%%%%%%%%%%%%%%%%%

Although this section closely follows an earlier work \cite{WJPA} for producing a discrete time random walk and its scaling limit in $\Qp$, there are important differences.  Namely, the primitive discrete time random walk that this section constructs permits a non-zero probability for remaining at $0$ in a single step.  This section examines the effect of this non-zero probability on the continuum limit.  

\subsection{The primitive random walk}

For an appropriately chosen normalization constant $\ConXG$ and any $P_0$ in $(0,1)$, take $X$ to be a random variable that takes values in $\G$ and has probability mass function $\rho_X$ given by \begin{align}\label{eq:probmass:1-d}\rho_X([x]) = \begin{cases} \frac{\ConXG}{p^{kb}} \frac{1}{\Vol(\Shello(k))} &\text{ if } [x] \in \Shello(k),\; \forall k\in\NN\\ P_0 &\text{ if } [x] = [0].\end{cases}\end{align} Uniformity of $\rho_X$ on shells implies that \[\Prob\big(X\in\Shello(i)\big) = \begin{cases} \frac{\ConXG}{p^{ib}} &\text{ if } i\ne 0\\ P_0 &\text{ if } i=0.\end{cases}\]  Sum the geometric series to obtain the equalities %
\begin{align*}
1  = P_0 + \ConXG\Big(\frac{1}{p^b} + \frac{1}{p^{2b}} + \cdots\Big) = P_0 + \frac{\ConXG}{(p^b - 1)},
\end{align*}
which imply that %
\begin{equation}
\ConXG = (1 - P_0)(p^b - 1).
\end{equation}
Write %
\begin{equation*}
\rho(i) = \frac{\ConXG p}{p^{(i+1)b}(p-1)},
\end{equation*} %
so that %
\begin{align}\label{eq:rhodiff1-D}
\rho(i) - \rho(i+1) &= \ConXG\Big(1 - \frac{1}{p^{b+1}}\Big)\frac{p}{p^{(i+1)b}(p-1)}\notag\\&=  (1-P_0)(p^b-1)\Big(\frac{p^{b+1}-1}{p^{b+1}}\Big)\frac{p}{p^{(i+1)b}(p-1)}.
\end{align}

The law for $X$ generalizes the law given earlier \cite{WJPA}, with the goal in mind of producing a limiting Brownian motion with a range of potential diffusion constants.  Denote by $\alpha(P_0)$ the constant %
\begin{equation}\label{eq:alpha:allcases}
\alpha(P_0) = \frac{(1 - P_0)\big(p^{b+1}-1\big)}{p^b(p-1)},
\end{equation} %
and by $\phi_X$ the characteristic function for $X$.

\begin{proposition}\label{CharFunX1DdifD}%
For any $y$ in $\Zp$,  %
\[\phi_X(y) = 1 - \alpha(P_0)|y|_p^b.\]%
\end{proposition}

\begin{proof}
Decompose the integral that defines the Fourier transform of $\phi_X$ into integrals over shells in $\G$ to obtain the equalities
\begin{align*}
\phi_X(y) &= \int_\G \langle -[x],y\rangle\rho_X([x]) \,\d[x] \\ %
&= \int_{\Ballo(0)} \langle -[x],y\rangle\rho_X([x]) \,\d[x] + \sum_{i\in\NN}\int_{\Shello(i)} \langle -[x],y\rangle\rho_X([x]) \,\d[x]\\ %
&= P_0\int_{\Ballo(0)} \langle [x],y\rangle \,\d[x] + \sum_{i\in\NN}\rho(i)\int_{\Shello(i)} \langle [x],y\rangle \,\d[x]\\ %
&= P_0\int_{\Ballo(0)} \langle [x],y\rangle \,\d[x] + \sum_{i\in\NN}\rho(i)\Bigg(\int_{\Ballo(i)} \langle [x],y\rangle \,\d[x] - \int_{\Ballo(i-1)} \langle [x],y\rangle \,\d[x]\Bigg)\\ %
&= P_0 - \rho(1)\int_{\Ballo(0)} \langle [x],y\rangle \,\d[x] + \sum_{i\in\NN}(\rho(i)-\rho(i+1))\int_{\Ballo(i)} \langle [x],y\rangle \,\d[x].%\\ %
\end{align*}
Lemma~\ref{lem:LCAG:Int} and \eqref{eq:rhodiff1-D} together imply that %
\begin{equation}
\phi_X(y) = P_0 - \frac{\ConXG}{p^{b}}\frac{1}{p-1} + \Big(\frac{p^{b+1}-1}{p^{b+1}}\Big)\frac{p}{p-1}\ConXG\sum_{i\in\NN}\frac{1}{p^{ib}} {\mathds 1}_{\Ball(-i)}(y),
\end{equation}
and so for any $k$ in $\NN$ and $y$ in $\Shell(-k)$, %
\begin{align*}
\phi_X(y)  &=  P_0 - \frac{\ConXG}{p^{b}}\frac{1}{p-1} + \Big(\frac{p^{b+1}-1}{p^{b+1}}\Big)\frac{p}{p-1}\ConXG\sum_{i=1}^k\frac{1}{p^{ib}}\\
&=1 - \Bigg(\frac{(1 - P_0)\big(p^{b+1}-1\big)}{p^b(p-1)}\Bigg)\frac{1}{p^{bk}}.
\end{align*}

\end{proof}

Take $X_0$ to be a random variable that is almost surely equal to $[0]$ and $(X_i)$ to be a sequence of independent random variables with the same distribution as $X$.  For any $n$ in $\No$, take \[S_n = X_0 + X_1 + \cdots + X_n\] and denote by $\rho^\ast(n, \cdot)$ the probability mass function for $S_n$.  Denote by $\phi(i)$ the quantity \[\phi(i) = \phi_X(y),\] where $y$ is any element of $\Shell(-i)$.

\begin{proposition}\label{Rhon1DdifD}
For any $g$ in $\G$,
\begin{align*}\rho^\ast(n, g) &= (1-\alpha)^n\mathds 1_{\Ballo(0)}(g) + \sum_{i\in \mathds N}\Bigg(\Big(1-\frac{\alpha}{p^{ib}}\Big)^n - \Big(1-\frac{\alpha}{p^{(i-1)b}}\Big)^n\Bigg)\frac{1}{p^{i}}\mathds 1_{\Ballo(i)}(g).\end{align*}
\end{proposition}

\begin{proof}
The characteristic function for $S_n$ is the $n$-fold product of the characteristic function for $X$, which implies that $\rho^\ast(n, g)$ is given by %
\begin{align*}%
\rho^\ast(n, g) &= \int_{\Gd} \phi_{X}(y)^n\langle g,y\rangle \,\d y \\  &= \sum_{i=0}^\infty \int_{\Shell(-i)} \phi_{X}(y)^n\langle g,y\rangle \,\d y 
 \\  &= \sum_{i=0}^\infty \phi(i)^n\int_{\Shell(-i)}\langle g,y\rangle \,\d y
 \\  &= \sum_{i=0}^\infty \phi(i)^n\Bigg(\int_{\Ball(-i)}\langle g,y\rangle \,\d y - \int_{\Ball(-i-1)}\langle g,y\rangle \,\d y\Bigg)
\\ & = (1-\alpha)^n\mathds 1_{\Ballo(0)}(g) \\& \qquad + \sum_{i\in \mathds N}\Bigg(\Big(1-\frac{\alpha}{p^{ib}}\Big)^n - \Big(1-\frac{\alpha}{p^{(i-1)b}}\Big)^n\Bigg)\frac{1}{p^{i}}\mathds 1_{\Ballo(i)}(g).
\end{align*}  
\end{proof}

For any $r$ in $(0,b)$, denote by $\EXG\big[|S_n|^r\big]$ the expected value of the random variable $|S_n|^r$ with respect to the measure $\PbG$ on the space of all paths in $\G$.  Take $C(r,1)$ to be the constant \[C(r,1) = \frac{p^{r+1}-p^r}{p^{r+1}-1},\]

\begin{proposition}\label{ExpV1DdifD}
For any $r$ in $(0,b)$,
\[\EXG\big[|S_n|^r\big]  = C(r,1)\sum_{i\in\mathds N}\Bigg(\Big(1-\frac{\alpha}{p^{ib}}\Big)^n - \Big(1-\frac{\alpha}{p^{(i-1)b}}\Big)^n\Bigg)\big(p^{ir}-p^{-i}\big).\]
\end{proposition}

\begin{proof}
Follow the calculations in the earlier settings \cite[Theorem~3.1]{WJPA} to obtain the equalities %
\begin{align*}%
\EXG\big[|S_n|^r\big] &= \int_\G |g|^r\rho^\ast(n, g)\,\d g\\ & = \sum_{k\in \No}  \int_{\Shello(k)} |g|^r\rho^\ast(n, g)\,\d g%
\\ & = \sum_{k\in \NN}  \int_{\Shello(k)} |g|^r\Bigg\{\sum_{i\in \NN}\Bigg(\Big(1-\frac{\alpha}{p^{ib}}\Big)^n - \Big(1-\frac{\alpha}{p^{(i-1)b}}\Big)^n\Bigg)\frac{1}{p^{i}}\mathds 1_{\Ballo(i)}(g)\Bigg\}\,\d g%
\\ & = \sum_{i\in \NN}\Bigg(\Big(1-\frac{\alpha}{p^{ib}}\Big)^n - \Big(1-\frac{\alpha}{p^{(i-1)b}}\Big)^n\Bigg)\frac{1}{p^{i}}\sum_{k\in \NN}  \int_{\Shello(k)} |g|^r\mathds 1_{\Ballo(i)}(g)\,\d g
\\ & = \sum_{i\in \NN}\Bigg(\Big(1-\frac{\alpha}{p^{ib}}\Big)^n - \Big(1-\frac{\alpha}{p^{(i-1)b}}\Big)^n\Bigg)\frac{1}{p^{i}}\sum_{k\in \NN}  p^{kr}\int_{\Shello(k)} \mathds 1_{\Ballo(i)}(g)\,\d g
\\ & = \sum_{i\in \NN}\Bigg(\Big(1-\frac{\alpha}{p^{ib}}\Big)^n - \Big(1-\frac{\alpha}{p^{(i-1)b}}\Big)^n\Bigg)\Big(1-\frac{1}{p}\Big)\frac{1}{p^{i}}\big(p^rp + \cdots + p^{ir}p^i\big)
\\ & = \frac{p^{r+1}-p^r}{{p^{r+1}-1}}\sum_{i\in \NN}\Bigg(\Big(1-\frac{\alpha}{p^{ib}}\Big)^n - \Big(1-\frac{\alpha}{p^{(i-1)b}}\Big)^n\Bigg)\big(p^{ir}-p^{-i}\big).
\end{align*}%
\end{proof}

It is helpful to identify a more general inequality than the earlier works established.

\begin{lemma}\label{maxnorm:prim-mom-est-lemma}
For any prime $p$ and any positive real numbers $\alpha$, $b$, and $r$ so that \[0 < r < b, \quad \alpha \in (0,2), \quad \text{and} \quad \frac{\alpha}{p^b} < 1,\] and any natural number $d$, there is a constant $K$ so that \begin{align}\label{GenPrimMomentEst}&\sum_{i>1}\Bigg(\Big(1-\frac{\alpha}{p^{ib}}\Big)^n - \Big(1-\frac{\alpha}{p^{(i-1)b}}\Big)^n\Bigg)\big(p^{ir}-p^{-2i}\big) \leq Kn^{\frac{r}{b}}.\end{align}
\end{lemma}

The proof of Lemma~\ref{maxnorm:prim-mom-est-lemma} closely follows earlier arguments of Weisbart \cite[Theorem~3.1]{WJPA}, and Pierce, et. al. \cite[Theorem~4.2]{W:Expo:24}, and so is omitted. The idea is to compare the given sum with a lower Riemann sum approximation of a beta function to bound the integral by a quotient of gamma functions.

\begin{theorem}\label{MomEstX1DdifD}
For any $n$ in $\No$ and any $r$ in $(0,b)$, there is a constant $K$ so that \begin{equation}\label{eq:MomEstX1DdifD}\EXG\big[|S_n|^r\big] \leq Kn^{\frac{r}{b}}.\end{equation}
\end{theorem}

\begin{proof}
Take \[\varepsilon = \max\Big(1-\frac{\alpha}{p^{b}}, |1-\alpha|\Big)\] to obtain the inequality \begin{equation} \Bigg|\Bigg(\Big(1-\frac{\alpha}{p^{b}}\Big)^n - \Big(1-\alpha\Big)^n\Bigg)\big(p^r-p^{-2}\big)\Bigg| \leq 2\varepsilon^np^r.\end{equation} The equality \[\lim_{n\to 0} 2\varepsilon^np^rn^{\frac{r}{b}} = 0\] implies that there is a constant $C$ so that \begin{equation}\label{MomEstX1DdifDpreEq}\Bigg|\Bigg(\Big(1-\frac{\alpha}{p^{b}}\Big)^n - \Big(1-\alpha\Big)^n\Bigg)\big(p^r-p^{-2}\big)\Bigg| \leq Cn^{\frac{r}{b}}.\end{equation}

Write \begin{equation*} I(n) = \sum_{i>1}\Bigg(\Big(1-\frac{\alpha}{p^{(i+1)b}}\Big)^n - \Big(1-\frac{\alpha}{p^{ib}}\Big)^n\Bigg)\big(p^{(i+1)r}-p^{-2(i+1)}\big)\end{equation*} to obtain the equality \[\EXG\big[|S_n|^r\big] = \Bigg(\Big(1-\frac{\alpha}{p^{b}}\Big)^n - \Big(1-\alpha\Big)^n\Bigg)\big(p^r-p^{-2}\big) + I(n),\]  so that Lemma~\ref{maxnorm:prim-mom-est-lemma} and \eqref{MomEstX1DdifDpreEq} together imply \eqref{eq:MomEstX1DdifD}.
\end{proof}

%%%%%%%%%%%%%%%%%%%%%%%%%%%%%%%%%%%%%%%%%%%%%%%%%%%%%%%%%%%%%%%%%%%%%%
%%%%%%%%%%%%%%%%%%%%%%%%%%%%%%%%%%%%%%%%%%%%%%%%%%%%%%%%%%%%%%%%%%%%%%

\subsection{Embedding the primitive process}

%%%%%%%%%%%%%%%%%%%%%%%%%%%%%%%%%%%%%%%%%%%%%%%%%%%%%%%%%%%%%%%%%%%%%%
%%%%%%%%%%%%%%%%%%%%%%%%%%%%%%%%%%%%%%%%%%%%%%%%%%%%%%%%%%%%%%%%%%%%%%

Define the process $S^{(m)}$ with state space $\Gm$ by \[S^{(m)} = Q_m\circ S.\]  This process has generator $X^{(m)}$ that is given by \[X^{(m)} = Q_m\circ X,\] with law given by $\rho_{X^{(m)}}$, where \begin{equation}\label{eq:newdensitydiscrete}\rho_{X^{(m)}}\big(Q_m([x])\big) = \rho_{X^{(m)}}\big([p^mx]_m\big) = p^{m}\rho_{X}([x]).\end{equation} For each $n$, the law for $S_n^{(m)}$ is similarly given by \begin{equation}\label{eq:newdensitydiscretenstep}\rho_m\big(n,Q_m([x])\big) = \rho_m\big(n,[p^mx]_m\big) = p^{m}\rho_\ast(n,[x]).\end{equation}

The probabilities for the cylinder sets that $S^{(m)}$ determines extend to a probability measure $\Pbm$ on $F(\No\colon \Gm)$. Denote by $\EXm$ the expected value with respect to $\Pbm$.

\begin{proposition}\label{MomEstX1DdifD}
There is a constant $K$ independent of $m$ and $n$, so that for any $m$ and $n$ in $\No$
\[\EXm\Big[\big|S^{(m)}_n\big|^r\Big] \leq Kp^{-rm}n^{\frac{r}{b}}.\]
\end{proposition}

\begin{proof}
Integrate $\big|S^{(m)}_n\big|^r$ and use Lemma~\ref{StateSpaces:Intmoverballs} to obtain the equalities
\begin{align*}
\EXm\Big[\big|S^{(m)}_n\big|^r\Big] = \int_{\Gm} |[x]_m|^r\rho_m\big(n,[x]_m\big) \,{\rm d}[x]_m
 = \int_{\Qp} |[x]_m|^r\rho_m\big(n,[x]_m\big) \,{\rm d}x.%
\end{align*}
Equation \eqref{eq:newdensitydiscretenstep} implies that 
\begin{align*}%
\EXm\Big[\big|S^{(m)}_n\big|^r\Big]  = \int_{\Qp} |[x]_m|^r\rho_m\big(n,[p^{-m}x]\big) p^m\,{\rm d}x%\\%
 = \int_{\Qp} |[p^{m}x]_m|^r\rho_\ast\big(n,[x]\big)p^m \,{\rm d}(p^mx),%
\end{align*}
and a change of variables implies that
\begin{align*}%
\EXm\Big[\big|S^{(m)}_n\big|^r\Big] & = \int_{\Qp} |p^m[x]|^r\rho_\ast\big(n,[x]\big) \,{\rm d}x\\%
& = p^{-rm}\EXG\big[|S_n|^r\big] \leq p^{-rm}Kn^{\frac{r}{b}}.
\end{align*}

\end{proof}

Denote by $\gamma_m$  the function that takes $\Gm$ to $\Qp$ by \[\gamma_m\colon x+p^m\Zp \mapsto \sum_{k<m}a_x(k)p^k\] and take $\Gamma_m$ to be the function \[\Gamma_m = \gamma_m\circ Q_m.\]  Denote by $\iota_m$ the function that takes $\mathds N\times \G$ to $[0,\infty)\times \Qp$ by \[\iota_m\colon (n,[x])\mapsto \big(\lambda(m)n, \Gamma_m([x])\big),\] where $\lambda(m)$ is some appropriately chosen sequence that tends to infinity as $m$ tends to infinity.  The terms of the sequence $(\lambda(m))$ are reciprocals of the terms of a sequence of time scales.

The function $\iota_m$ acts on the process $S_n$ to produce a continuous time process in $\Qp$ by using \eqref{EmbeddedProcessFDD} to identify probabilities for the simple cylinder sets of $F([0,\infty)\colon \Qp)$ with those of $F(\No\colon \G)$.  Denote again by $\Pbm$ the probability measure on $F([0,\infty)\colon \Qp)$ obtained in this way.  For any $t$ in $[0,\infty)$,
\begin{equation}\label{eq:YttoSmExp-1D}
\EXm\big[|Y_t|^r\big]  = \EXm\Big[\big|S^{(m)}_{\floor{\lambda(m) t}}\big|^r\Big]. 
\end{equation}

Specialize the sequence of time scales so that for some positive real number $D$, the reciprocal sequence $(\lambda(m))$ is given by \begin{equation}\label{eq:lambdam:param}\lambda(m) = Dp^{mb},\end{equation} and maintain this specialization henceforth.

\begin{theorem}\label{MomEstX1DdifDQp}
For any $m$ in $\No$ and $t$ in $[0,\infty)$, there is a constant $K$ so that \[\EXm\big[|Y_t|^r\big] \leq Kt^{\frac{r}{b}}.\]
\end{theorem}

\begin{proof}
Theorem~\ref{MomEstX1DdifD} and \eqref{eq:YttoSmExp-1D} together imply that there is a constant $K^\prime$ so that %
\begin{align*}
\EXm\big[|Y_t|^r\big] & = \EXm\Big[\big|S^{(m)}_{\floor{\lambda t}}\big|^r\Big]\\
& \leq p^{-rm}K^\prime(\floor{\lambda(m)t})^{\frac{r}{b}}\\
& \leq p^{-rm}K^\prime(Dp^{mb}t)^{\frac{r}{b}} = K^\prime D^{\frac{r}{b}}t^{\frac{r}{b}}.
\end{align*}

\end{proof}

Use Theorem~\ref{MomEstX1DdifDQp} to show that the sequence of measures $(\Pbm)$ satisfies the criterion of Chentsov \cite{cent} exactly as in the earlier work \cite[Theorem~4.6]{WJPA} to obtain Proposition~\ref{tight:1D}.

\begin{proposition}\label{tight:1D}
The stochastic process $(F([0,\infty)\colon \Qp), \Pbm, Y)$ has a version with paths in $D([0,\infty)\colon \Qp)$, a process $(D([0,\infty)\colon \Qp), \Pbm, Y)$.  The sequence of measures $(\Pbm)$ with paths in $D([0,\infty)\colon \Qp)$ is uniformly tight.
\end{proposition}

The proof of Proposition~\ref{prop:qp:concenration} is identical to the proof of the corresponding proposition in the prior work \cite[Theorem~4.7]{WJPA}, so it is omitted here.

\begin{proposition}\label{prop:qp:concenration}
The measure $\Pbm$ is concentrated on the subset of paths in $D([0, \infty)\colon \mathds Q_p)$ that are valued in $\Gamma_m(G)$ and are, for each natural number $n$, constant on the intervals $\big[(n-1)\tau_m, n\tau_m\big)$.
\end{proposition}

\subsection{The limiting processes}\label{subsec:LimitProcess1D}

The books by Kochubei \cite{Kochubei:Book:2001} and Z\'{u}\~{n}iga-Galindo \cite{Zun1} are both excellent references that provide additional background on parabolic equations over non-Archimedean fields and the theory of $p$-adic diffusion.

For any $f$ in $L^2(\Qp)$ denote by $\widehat{f}$ its Fourier transform.  For any positive real number $b$, define the Vladimirov operator on $L^2(\Qp)$ by
\[
\Delta_b f = \mathcal F^{-1}\big(\,|\cdot|_p^{b}\widehat f\:\big).
\]
With the domain \[\mathcal D(\Delta_b) = \Big\{f\in L^2(\Qp)\colon |\cdot|_p^{\,b}\,\widehat f\in L^2(\Qp)\Big\},\] the operator $\Delta_b$ is self-adjoint.
For any $t$ in $(0,\infty)$ and any function $g$ on $(0,\infty)\times \Qp$, denote by $g_t$ the function that takes any $x$ in $\Qp$ to $g(t,x)$.  Take $\Delta_b$ to act in the space variable so that \[(\Delta_b g)(t,x)=(\Delta_bg_t)(x)\] for each $g$ so that $g_t$ is in $\mathcal D(\Delta_b)$ for each $t$ in $(0,\infty)$.  The Fourier transform and its inverse similarly act in the space variable for functions defined on $(0,\infty)\times \Qp$.

Take $\sigma$ to be a positive real number. The heat kernel
\[
\rho(t,x)=\big(\mathcal F^{-1} \e^{-\sigma t|\cdot|_p^{b}}\big)(x)
\]
is the (unique) fundamental solution to \eqref{Intro:HeatEquation}, and following Varadarajan with only minor modification \cite{Varadarajan:LMP:1997},
\begin{equation}\label{ScalLim:Equation:Qppdf}
\rho(t,x)=\sum_{k\in\Z} p^{k}\left(\e^{-\sigma t p^{kb}}-\e^{-\sigma t p^{(k+1)b}}\right)\indicator{\Ball(-k)}(x).
\end{equation}
The family $\{\rho(t,\cdot)\}_{t>0}$ is a convolution semigroup of probability densities.  These kernels determine a probability measure $\Pb$ on $D([0,\infty)\colon\Qp)$ concentrated on paths that are initially at $0$. The resulting process $(D([0,\infty)\colon\Qp),\Pb,Y)$ is a $p$-adic Brownian motion.

\subsection{Convergence of the processes}

For each $t$ in $(0, \infty)$, there is a sequence $(t_m)$ so that for each $m$, \[\floor{\lambda(m)t} = \floor{\lambda(m)t_m} = \lambda(m)t_m, \quad \text{and so} \quad |t_m-t| < \frac{1}{\lambda(m)}.\] Since $(\lambda_m)$ diverges to infinity, $(t_m)$ converges to $t$.

To prove the convergence of the measures $(\Pbm)$ to some limiting measure requires determining the convergence of the finite dimensional distributions of the measures. It is useful to compute the characteristic function $\phi_{X^{(m)}}$ for the random variable $X^{(m)}$, since it may be used to determine the law for $S_n^{(m)}$ for each $n$.  The characteristic function has a simple, closed form, where the law does not.

\begin{proposition}\label{prop:1-dChar}
For any $y$ in $p^{-m}\Zp$, the characteristic function $\phi_{X^{(m)}}$ for $X^{(m)}$ is given by the equality 
\[\phi_{X^{(m)}}(y) = 1 - \frac{\alpha(P_0)|y|^b}{p^{mb}}.\]
\end{proposition}

\begin{proof}
Take the Fourier transform of $\rho_{X^{(m)}}$ and use Lemma~\ref{StateSpaces:Intmoverballs} to obtain the equalities
\begin{align*} 
\phi_{X^{(m)}}(y) = \int_{\Gm} \langle -[x]_m, y\rangle_m\rho_{X^{(m)}}([x]_m) \,\d [x]_m%
 = \int_{\Qp} \langle -[x]_m, y\rangle_m\rho_{X^{(m)}}\big([x]_m\big) \,\d x.%
\end{align*}
Equation \eqref{eq:newdensitydiscrete} implies that
\begin{align*} 
\phi_{X^{(m)}}(y) = \int_{\Qp} \langle -[x]_m, y\rangle_m\rho_X\big([p^{-m}x]\big)p^m \,\d x.%\\ %
\end{align*}
A change of variables and Proposition~\ref{CharFunX1DdifD} together imply that
\begin{align*} 
\phi_{X^{(m)}}(y) & = \int_{\Qp} \langle -[p^{m}x]_m, y\rangle_m\rho_X([x])p^m \,\d (p^{m}x) \\
& =\int_{\Qp} \langle -[x], p^my\rangle\rho_X([x]) \,\d x  %
= 1 - \frac{\alpha(P_0)|y|^b}{p^{mb}}.
\end{align*}
\end{proof}

The Fourier transform takes an $n$-fold convolution of functions to an $n$-fold product of the Fourier transforms of the functions.  For each $n$, the characteristic function $\phi^{(m)}_n$ for $S_n^{(m)}$ is, therefore, an $n$-fold product.

\begin{corollary}
For each natural number $n$, 
\begin{equation*}
\phi^{(m)}_n(y) = \left(1 - \frac{\alpha(P_0)|y|^b}{p^{mb}}\right)^n.
\end{equation*}
\end{corollary}

Take $E_m$ to be the function that is defined for any $(n,y)$ in $\mathds N_0\times \Qp$ by \[E_m(n, y) = \begin{cases}  
\left(1 - \frac{\alpha(P_0)|y|^b}{p^{mb}}\right)^{n} &\mbox{if }|y|_p \leq p^{m}\\0 &\mbox{if }|y|_p > p^m.\end{cases}\]

Take the inverse Fourier transform of the characteristic function for $S^{(m)}_{n}$ and use \eqref{eq:pairing-1d} to obtain the equalities%
\begin{align*}
\rho_m(n,[x]_m) &= \int_{p^{-m}\Zp} \langle [x]_m, y\rangle_m \left(1-\frac{\alpha(P_0) |y|_p^b}{p^{mb}}\right)^{n}\,\d y\\
&= \int_{\Qp} \chi(xy) E_m(n,y)\,\d y.
\end{align*}

For each positive real number $t$, denote ambiguously by $\rho_m(t,[\cdot]_m)$ the function that is given by \[\rho_m(t,[x]_m) = \rho_m(\lambda(m)t_m, [x]_m).\] For any $y$ in $\Qp$, \begin{equation}\label{DiffusionConst1-D}E_m(\lambda(m)t_m, y)) \to \e^{-D\alpha(P_0)t|y|_p^b}.\end{equation}  With this in mind, take \[D = \frac{\sigma}{\alpha(P_0)}\] to precisely determine the asymptotic behavior of $\lambda(m)$.

\begin{lemma}\label{lem:unifcon1-d:A}
For any positive real number $t$, the sequence of functions $(\rho_m(t,[\cdot]_m))$ converges uniformly on $\Qp$ to the probability density function $\rho(t, \cdot)$.
\end{lemma}

\begin{proof}
For any $x$ in $\Qp$, %
\begin{align*}
\left|\rho(t,x) - \rho_m(t,[x]_m)\right| & = \left|\int_{\Qp} \chi(xy) \e^{-\sigma t |y|_p^b}\,\d y - \int_{\Qp} \chi(xy) E_m(\lambda(m)t_m,y)\,\d y\right|\\%
& \leq \int_{\Qp} \left|\e^{-\sigma t |y|_p^b}-E_m(\lambda(m)t_m,y)\right|\,\d y \to 0
\end{align*}
Independence of the integral to the right of the inequality on $x$ implies the uniform convergence of $\rho(t, \cdot)$ to $\rho_m(t, [\cdot]_m)$ on $\Qp$.

\end{proof}

Denote by $H_R$ the set of \emph{restricted histories} for paths in $D([0, \infty)\colon \Qp)$, the set of all histories whose route is a finite sequence of balls.

\begin{proposition}
For any restricted history $h$ in $H_R$, \[\Pbm(\C(h)) \to \Pb(\C(h)).\]
\end{proposition}

\begin{proof}
Take $h$ to be any restricted history and without loss in generality suppose that $U_h(0)$ is the set $\{0\}$.  Simplify the notation by writing \[e_h = (t_0, \dots, t_k) \quad {\rm and}\quad U_h = (\{0\}, U_1, \dots, U_k).\]  For any $i$ in $\{0, \dots, k\}$, denote by $r_i$ the radius of $U_i$. For any $m$ so that $p^{-m}$ is less than $\min\{r_1, \dots, r_k\}$, Lemma~\ref{StateSpaces:Intmoverballs} and the uniform convergence given by Lemma~\ref{lem:unifcon1-d:A} together imply that %
\begin{align*}
\Pbm(\C(h)) &= \Pbm\big(Y_{t_1}\in U_1, Y_{t_2}\in U_2, \dots, Y_{t_n}\in U_n\big)\\
&= \int_{U_1} \cdots \int_{U_k} \prod_{i\in\{1, \dots, k\}}\rho_m\big(t_i-t_{i-1}, [x_i -x_{i-1}]_m\big)\,{\rm d}[x_k]_m\cdots {\rm d}[x_1]_m\\
&= \int_{U_1} \cdots \int_{U_k} \prod_{i\in\{1, \dots, k\}}\rho_m\big(t_i-t_{i-1}, [x_i -x_{i-1}]_m\big)\,{\rm d}x_k\cdots {\rm d}x_1\\
&\to  \int_{U_1}\cdots \int_{U_k} \prod_{i\in\{1, \dots, k\}}\rho\big(t_i-t_{i-1}, x_i-x_{i-1}\big)\,{\rm d}x_k\cdots {\rm d}x_1 = \Pb(\C(h)).
\end{align*}

\end{proof}

Since $\C(H_{R})$ is a $\pi$-system that generates the cylinder $\sigma$-algebra in $D([0, \infty)\colon \Qp)$, the convergence holds for any set in the cylinder $\sigma$-algebra.  Together with the uniform tightness of the family of measures $\{\Pbm\colon m\in \mathds N_0\}$ that Proposition~\ref{tight:1D} guarantees, the convergence of the finite dimensional distributions of the measure $\Pbm$ to those of $\Pb$ implies Theorem~\ref{Sec:Con:Theorem:MAIN}.

\begin{theorem}\label{Sec:Con:Theorem:MAIN}
The sequence of measures $(\Pbm)$ converges weakly to $\Pb$ in $D([0, \infty)\colon \Qp)$.  Furthermore, the diffusion constant for the limiting process is \[D\alpha(P_0) = \frac{D(1 - P_0)\big(p^{b+1}-1\big)}{p^b(p-1)}.\]
\end{theorem}

%%%%%%%%%%%%%%%%%%%%%%%%%%%%%%%%%%%%%%%%%%%%%%%%%%%%%%%%%%%%%%%%%%%%%%
%%%%%%%%%%%%%%%%%%%%%%%%%%%%%%%%%%%%%%%%%%%%%%%%%%%%%%%%%%%%%%%%%%%%%%

%%%%%%%%%%%%%%%%%%%%%%%%%%%%%%%%%%%%%%%%%%%%%%%%%%%%%%%%%%%%%%%%%%%%%%
%%%%%%%%%%%%%%%%%%%%%%%%%%%%%%%%%%%%%%%%%%%%%%%%%%%%%%%%%%%%%%%%%%%%%%
%%%%%%%%%%%%%%%%%%%%%%%%%%%%%%%%%%%%%%%%%%%%%%%%%%%%%%%%%%%%%%%%%%%%%%
%%%%%%%%%%%%%%%%%%%%%%%%%%%%%%%%%%%%%%%%%%%%%%%%%%%%%%%%%%%%%%%%%%%%%%
%%%%%%%%%%%%%%%%%%%%%%%%%%%%%%%%%%%%%%%%%%%%%%%%%%%%%%%%%%%%%%%%%%%%%%
%%%%%%%%%%%%%%%%%%%%%%%%%%%%%%%%%%%%%%%%%%%%%%%%%%%%%%%%%%%%%%%%%%%%%%
%%%%%%%%%%%%%%%%%%%%%%%%%%%%%%%%%%%%%%%%%%%%%%%%%%%%%%%%%%%%%%%%%%%%%%
%%%%%%%%%%%%%%%%%%%%%%%%%%%%%%%%%%%%%%%%%%%%%%%%%%%%%%%%%%%%%%%%%%%%%%
%%%%%%%%%%%%%%%%%%%%%%%%%%%%%%%%%%%%%%%%%%%%%%%%%%%%%%%%%%%%%%%%%%%%%%
%%%%%%%%%%%%%%%%%%%%%%%%%%%%%%%%%%%%%%%%%%%%%%%%%%%%%%%%%%%%%%%%%%%%%%
%%%%%%%%%%%%%%%%%%%%%%%%%%%%%%%%%%%%%%%%%%%%%%%%%%%%%%%%%%%%%%%%%%%%%%
%%%%%%%%%%%%%%%%%%%%%%%%%%%%%%%%%%%%%%%%%%%%%%%%%%%%%%%%%%%%%%%%%%%%%%

\section{The isotropic process and its approximation}\label{Sec:Prim}

%%%%%%%%%%%%%%%%%%%%%%%%%%%%%%%%%%%%%%%%%%%%%%%%%%%%%%%%%%%%%%%%%%%%%%
%%%%%%%%%%%%%%%%%%%%%%%%%%%%%%%%%%%%%%%%%%%%%%%%%%%%%%%%%%%%%%%%%%%%%%
%%%%%%%%%%%%%%%%%%%%%%%%%%%%%%%%%%%%%%%%%%%%%%%%%%%%%%%%%%%%%%%%%%%%%%
%%%%%%%%%%%%%%%%%%%%%%%%%%%%%%%%%%%%%%%%%%%%%%%%%%%%%%%%%%%%%%%%%%%%%%
%%%%%%%%%%%%%%%%%%%%%%%%%%%%%%%%%%%%%%%%%%%%%%%%%%%%%%%%%%%%%%%%%%%%%%
%%%%%%%%%%%%%%%%%%%%%%%%%%%%%%%%%%%%%%%%%%%%%%%%%%%%%%%%%%%%%%%%%%%%%%
%%%%%%%%%%%%%%%%%%%%%%%%%%%%%%%%%%%%%%%%%%%%%%%%%%%%%%%%%%%%%%%%%%%%%%
%%%%%%%%%%%%%%%%%%%%%%%%%%%%%%%%%%%%%%%%%%%%%%%%%%%%%%%%%%%%%%%%%%%%%%
%%%%%%%%%%%%%%%%%%%%%%%%%%%%%%%%%%%%%%%%%%%%%%%%%%%%%%%%%%%%%%%%%%%%%%
%%%%%%%%%%%%%%%%%%%%%%%%%%%%%%%%%%%%%%%%%%%%%%%%%%%%%%%%%%%%%%%%%%%%%%
%%%%%%%%%%%%%%%%%%%%%%%%%%%%%%%%%%%%%%%%%%%%%%%%%%%%%%%%%%%%%%%%%%%%%%
%%%%%%%%%%%%%%%%%%%%%%%%%%%%%%%%%%%%%%%%%%%%%%%%%%%%%%%%%%%%%%%%%%%%%%

Utilize the max-norm on $\Qp^2$ to define the function $\|\cdot\|$ on $\G^2$ given for each $\x$ in $\Qp$ by \[\|\x\| = \begin{cases}\|\ve{x}\|_{\max} & \text{ if } \x \ne [\ve{0}]\\0 & \text{ if } \x = [\ve{0}].\end{cases}\]  The canonical inclusion of $\Zp^2$ into $\Qp^2$ and the max-norm on $\Qp^2$ induce an ultrametric on $\Zp^2$.  Equip $\G^2$ with its coarse Pontryagin filteration.

%%%%%%%%%%%%%%%%%%%%%%%%%%%%%%%%%%%%%%%%%%%%%%%%%%%%%%%%%%%%%%%%%%%%%%
%%%%%%%%%%%%%%%%%%%%%%%%%%%%%%%%%%%%%%%%%%%%%%%%%%%%%%%%%%%%%%%%%%%%%%

\subsection{A primitive random walk with two components}

%%%%%%%%%%%%%%%%%%%%%%%%%%%%%%%%%%%%%%%%%%%%%%%%%%%%%%%%%%%%%%%%%%%%%%
%%%%%%%%%%%%%%%%%%%%%%%%%%%%%%%%%%%%%%%%%%%%%%%%%%%%%%%%%%%%%%%%%%%%%%

Take $\Shelloc(2i)$ to be coarse shells for the balls $\BBo(2i)$.  For some appropriately chosen normalization constant $\ConMG$, the random variable $\ve{X}$ with probability mass function $\rho_{\ve{X}}$ given by \[\Prob\Big(\ve{X}\in\Shelloc(2i)\Big) = \begin{cases} \frac{\ConMG}{p^{ib}} &\text{ if } i\ne 0\\ 0 &\text{ if } i=0\end{cases}\] generates a primitive discrete time random walk on $\G^2$.  To determine $\ConMG$, sum the geometric series to see that \begin{align*}1 &= \ConMG\left(\frac{1}{p^b} + \frac{1}{p^{2b}} + \cdots\right) = \frac{\ConMG}{p^b -1}, 
\end{align*}
hence \begin{equation}\label{ConMG}\ConMG = p^b-1.\end{equation}

 Take $(\ve{X}_i)$ to be a sequence of independent random variables, each with the same distribution as $\ve{X}$.  Take $\ve{X_0}$ to be equal to $[\ve{0}]$ almost surely, and \[\ve{S}_n = \ve{X}_0 + \ve{X}_1 + \cdots + \ve{X}_n.\] Specialize the calculations in the general case \cite[Proposition~4.2]{W:Expo:24} to determine the characteristic function $\phi_{\max}$ to be \[\phi_{\max}(\y)  = 1- \alpha_{\max}\|\y\|^b, \quad \text{where} \quad \alpha_{\max} = \frac{p^{b+2}-1}{p^b(p^2-1)}.\]  Continue with the specification of the general case to see that for any pair $(n,g)$ in $\mathds N\times \G^2$, \begin{align*}%
\rho^\ast(n, g) &= (1-\alpha_{\max})^n\mathds 1_{\BBo(0)}(g) \\& \qquad + \sum_{i\in \mathds N}\Big(\Big(1-\frac{\alpha_{\max}}{p^{ib}}\Big)^n - \Big(1-\frac{\alpha_{\max}}{p^{(i-1)b}}\Big)^n\Big)\frac{1}{p^{2i}}\mathds 1_{\BBo(2i)}(g).\end{align*}  The probability mass function $\rho^\ast$ induces on $F(\No\colon \G^2)$ the probability measure $\PbG$.  The stochastic process $(F(\No\colon \G^2), \PbG, S)$ is the \emph{coarse primitive process}. 

Take $C(r,d)$ to be the constant \[C(r,d) = \left(\frac{p^{r+2}-p^r}{p^{r+2}-1}\right),\] and $\EXG$ to denote the expected value with respect to the measure $\PbG$ to obtain first \cite{W:Expo:24} the equality \begin{align}\label{sec:MaxPrimitiveOrderEstimates}
\EXG\big[\|\ve{S}_n\|^r\big] &=\sum_{i\in\mathds N}C(r,d)\big(p^{ir}-p^{-2i}\big)\Big(\Big(1-\frac{\alpha_{\max}}{p^{ib}}\Big)^n - \Big(1-\frac{\alpha_{\max}}{p^{(i-1)b}}\Big)^n\Big)\end{align} and then use Lemma~\ref{maxnorm:prim-mom-est-lemma} to obtain the inequality \begin{align}\label{sec:MaxPrimitiveOrderEstimates}\EXG\big[\|\ve{S}_n\|^r\big] \leq Kn^{\frac{r}{b}}.\end{align}

%%%%%%%%%%%%%%%%%%%%%%%%%%%%%%%%%%%%%%%%%%%%%%%%%%%%%%%%%%%%%%%%%%%%%%
%%%%%%%%%%%%%%%%%%%%%%%%%%%%%%%%%%%%%%%%%%%%%%%%%%%%%%%%%%%%%%%%%%%%%%

\subsection{Embedding the primitive process}

%%%%%%%%%%%%%%%%%%%%%%%%%%%%%%%%%%%%%%%%%%%%%%%%%%%%%%%%%%%%%%%%%%%%%%
%%%%%%%%%%%%%%%%%%%%%%%%%%%%%%%%%%%%%%%%%%%%%%%%%%%%%%%%%%%%%%%%%%%%%%

Define the process $\ve{S}^{(m)}$ with state space $(\Gm)^2$ by \[\ve{S}^{(m)} = Q_m^{(2)}\circ \ve{S}.\]  This process has generator $\ve{X}^{(m)}$ that is given by \[\ve{X}^{(m)} = Q_m^{(2)}\circ \ve{X},\] and has a law given by $\rho_{\ve{X}^{(m)}}$, where %
\begin{equation}\label{eq:newdensitydiscrete2d}
\rho_{\ve{X}^{(m)}}\big(Q_m^{(2)}([\ve{x}])\big) = \rho_{\ve{X}^{(m)}}\big([p^m\ve{x}]_m\big) = p^{2m}\rho_{\ve{X}}([\ve{x}]).
\end{equation} 
For each $n$, the law $\rho_m\big(n,\cdot\big)$ for the random variable $\ve{S}_n^{(m)}$ is similarly given by %
\begin{equation}\label{eq:newdensitydiscretenstep2d}
\rho_m\big(n,Q_m^{(2)}([\ve{x}])\big) = \rho_m\big(n,[p^m\ve{x}]_m\big) = p^{2m}\rho_\ast(n,[\ve{x}]).
\end{equation}

The probabilities for the cylinder sets that $\ve{S}^{(m)}$ determines extends to a probability measure $\Pbm$ on $F(\No\colon (\Gm)^2)$. Denote by $\EXm$ the expected value with respect to $\Pbm$.  Follow the analogous proof in the one dimensional setting to obtain Proposition~\ref{MomEstX2DMdifD}.

\begin{proposition}\label{MomEstX2DMdifD}
There is a constant $K$ independent of $m$ and $n$, so that for any $m$ and $n$ in $\No$,
\[\EXm\Big[\|\ve{S}^{(m)}_n\|_{\max}^r\Big] \leq Kp^{-rm}n^{\frac{r}{b}}.\]
\end{proposition}

Take $\Gamma_m^{(2)}$ to be the function \[\Gamma_m^{(2)} = \Gamma_m \times \Gamma_m.\]  Denote by $\iota_m^{(2)}$ the function that takes $\mathds N\times \G^2$ to $[0,\infty)\times \Qp^2$ by \[\iota_m^{(2)}\colon (n,[\ve{x}])\mapsto \big(\lambda(m)n, \Gamma_m^{(2)}([\ve{x}])\big),\] where $\lambda(m)$ is chosen as in the dimension one case, so that for some positive real number $D$, \[\lambda(m) = Dp^{mb}.\]

The function $\iota_m^{(2)}$ acts on the process $\ve{S}_n$ to produce a continuous time process in $\Qp^2$ by using \eqref{EmbeddedProcessFDD} to identify probabilities for the simple cylinder sets of $F([0,\infty)\colon \Qp^2)$ with those of $F(\No\colon \G^2)$.  Denote again by $\Pbm$ the probability measure on $F([0,\infty)\colon \Qp^2)$ obtained in this way.

\begin{theorem}\label{MomEstX2DMdifDQp}
For any $m$ in $\No$ and $t$ in $[0,\infty)$, there is a constant $K$ so that \[\EXm\big[\|\ve{Y}_t\|_{\max}^r\big] \leq Kt^{\frac{r}{b}}.\]
\end{theorem}

Again specialize the earlier results \cite[Propositions~5.2 and 5.3]{W:Expo:24} to obtain Propositions~\ref{tight:2DM} and \ref{prop:qp:concenration2M}.

\begin{proposition}\label{tight:2DM}
The stochastic process $(F([0,\infty)\colon \Qp^2), \Pbm, \ve{Y})$ has a version with paths in $D([0,\infty)\colon \Qp^2)$, a process $(D([0,\infty)\colon \Qp^2), \Pbm, \ve{Y})$.  The sequence of measures $(\Pbm)$ with paths in $D([0,\infty)\colon \Qp^2)$ is uniformly tight.
\end{proposition}

\begin{proposition}\label{prop:qp:concenration2M}
The measure $\Pbm$ is concentrated on the subset of paths in $D([0, \infty)\colon \mathds Q_p^2)$ that are valued in $\Gamma_m^{(2)}(\G^2)$ and are, for each natural number $n$, constant on the intervals $\big[(n-1)\tau_m, n\tau_m\big)$.
\end{proposition}

%%%%%%%%%%%%%%%%%%%%%%%%%%%%%%%%%%%%%%%%%%%%%%%%%%%%%%%%%%%%%%%%%%%%%%
%%%%%%%%%%%%%%%%%%%%%%%%%%%%%%%%%%%%%%%%%%%%%%%%%%%%%%%%%%%%%%%%%%%%%%

\subsection{The limiting processes}\label{subsec:LimitProcess2DMax}

%%%%%%%%%%%%%%%%%%%%%%%%%%%%%%%%%%%%%%%%%%%%%%%%%%%%%%%%%%%%%%%%%%%%%%
%%%%%%%%%%%%%%%%%%%%%%%%%%%%%%%%%%%%%%%%%%%%%%%%%%%%%%%%%%%%%%%%%%%%%%

The isotropic Vladimirov operator on $L^2(\Qp^2)$ is given by
\[
\Delta_{b} f = \mathcal F^{-1}\big(\|\cdot\|_{\max}\widehat f\;\big),
\]
where $\widehat f$ is the Fourier transform of $f$. The isotropic diffusion equation in $\Qp^2$ is the equation \eqref{Intro:HeatEquation}, with $\Delta_b$ as the isotropic Vladimirov operator, and $x$ replaced by $\ve{x}$. 
With the domain \[\mathcal D(\Delta_{b}) = \Big\{f\in L^2(\Qp^2)\colon \|\cdot\|^b_{\max}\widehat f\in L^2(\Qp^2)\Big\},\] the operator $\Delta_b$ is self-adjoint.  Extend both $\Delta_{b}$ and the Fourier transform as in the one dimensional case to act on functions of a (non-negative) time variable and spatial variable.

\begin{lemma}\label{lemma:continuousEstimateMoment}
For any real number $r$ in $(0, b)$ and and positive real number $a$,
\[\sum_{k\in\ZZ} \Big(\e^{-ap^{kb}}-\e^{-ap^{(k+1)b}}\Big) p^{-kr} \leq p^{r}a^{\frac{r}{b}}\Gamma\Big(1-\frac{r}{b}\Big) \]
\end{lemma}

\begin{proof}
For any real numbers $y_k$ and $y_{k+1}$ where $y_{k+1}$ is larger than $y_k$, \begin{equation}\label{expdiff-to-int:A}\e^{-ay_k}-\e^{-ay_{k+1}} = \int_{y_{k}}^{y_{k+1}} a\e^{-ay}\,\d y.\end{equation} For any non-negative integrable function $f$ that is defined and decreasing on $[y_k, y_{k+1}]$, \begin{equation}\label{expdiff-to-int:B}\Big(\e^{-ay_k}-\e^{-ay_{k+1}}\Big)f(y_{k+1}) \leq \int_{y_{k}}^{y_{k+1}} a\e^{-ay}f(y)\,\d y.\end{equation}

For each $k$ in $\ZZ$ and each $y$ in $(0,\infty)$, take \[y_k = p^{kb} \quad \text{and} \quad f(y) = y^{-\frac{r}{b}}\] to obtain the inequality \begin{equation}\label{expdiff-to-int:C}\Big(\e^{-ap^{bk}}-\e^{-ap^{b(k+1)}}\Big)p^{-r(k+1)} \leq \int_{p^{kb}}^{p^{(k+1)b}} a\e^{-ay}y^{-\frac{r}{b}}\,\d y,\end{equation} hence %
\begin{align}\label{expdiff-to-int:C}
\Big(\e^{-ap^{bk}}-\e^{-ap^{b(k+1)}}\Big)p^{-rk}\leq p^{r}\int_{p^{kb}}^{p^{(k+1)b}} a\e^{-ay}y^{-\frac{r}{b}}\,\d y%
 = p^{r}a^{\frac{r}{b}}\int_{ap^{kb}}^{ap^{(k+1)b}} \e^{-u}u^{-\frac{r}{b}}\,\d u.
\end{align} %
Use \eqref{EXY:cont2D} and \eqref{expdiff-to-int:C} to obtain

\begin{align}\label{EXY:cont2D:final}
\sum_{k\in\ZZ} \Big(\e^{-ap^{kb}}-\e^{-ap^{(k+1)b}}\Big) p^{-kr}\leq p^ra^{\frac{r}{b}}\sum_{k\in\ZZ}\int_{ap^{kb}}^{ap^{(k+1)b}} \e^{-u}u^{-\frac{r}{b}}\,\d u%
= p^ra^{\frac{r}{b}}\Gamma\Big(1-\frac{r}{b}\Big).
\end{align}

\end{proof}

The heat kernel
\[
\rho(t,\ve{x})=\big(\mathcal F^{-1} \e^{-\sigma t\|\cdot\|_{\max}^{b}}\big)(\ve{x})
\]
is the fundamental solution to the isotropic diffusion equation. 

\begin{proposition}\label{Prop:Density:2dM}
For each positive $t$, \[\rho(t,\ve{x}) = \sum_{k\in\ZZ} p^{2k}\Big(\e^{-\sigma t\alpha_{\max}p^{kb}}-\e^{-\sigma t\alpha_{\max}p^{(k+1)b}}\Big)\indicator{\Ballt(-2k)}(\ve{x}).\] Furthermore, the family $\{\rho(t,\cdot)\}_{t>0}$ is convolution semigroup of probability density functions on $\Qp$ that for any $n$ in $[1,\infty)$ are $L^n(\Qp^2)$.
\end{proposition}

\begin{proof}
For each positive $t$, compute the inverse of Fourier transform by integrating over circles to obtain the equalities
\begin{align}\label{vsvAdapted2Dh}
\rho(t,\ve{x}) &= \big(\mathcal F^{-1} \e^{-\sigma t\alpha_{\max}\|\cdot\|_{\max}^{b}}\big)(\ve{x})\notag\\
&= \int_{\Qp^2}\langle \ve{x}, \ve{y}\rangle \e^{-\sigma t\alpha_{\max}\|\ve{y}\|_{\max}^{b}}\,{\rm d}\ve{y}\notag\\
&= \sum_{k\in\ZZ} \e^{-\sigma t\alpha_{\max}p^{kb}}\int_{\Shelltc(2k)}\langle \ve{x}, \ve{y}\rangle\,{\rm d}\ve{y}\notag\\
&= \sum_{k\in\ZZ} p^{2k}\Big(\e^{-\sigma t\alpha_{\max}p^{kb}}-\e^{-\sigma t\alpha_{\max}p^{(k+1)b}}\Big)\indicator{\Ballt(-2k)}(\ve{x}).%\notag\\
\end{align}%

Lemma~\ref{lemma:continuousEstimateMoment} implies that this sum is convergent for any $x$, and in fact, that $\rho(t,\cdot)$ is $L^n(\Qp)$ for every $n$ in $[1,\infty)$.  The given family of functions forms a convolution semigroup because the Fourier transform takes products to convolutions.  Positivity of each term in the sum that defines the functions in the family imply that the functions are positive, and since the Fourier transforms of these functions at $\ve{0}$ are equal to $1$, they each have unit mass.
\end{proof}

The family $\{\rho(t,\cdot)\}_{t>0}$ determines a probability measure $\Pb$ on $F([0,\infty)\colon\Qp^2)$, concentrated on paths that are initially at $\ve{0}$. To show that the process $\big(F([0,\infty)\colon\Qp^2), \Pbm, Y_t\big)$ has a version with paths in $D([0,\infty)\colon\Qp^2)$, it suffices to establish Proposition~\ref{Prop:EXY:cont2D} and then follow Varadarajan \cite[Theorem 4]{Varadarajan:LMP:1997}.

\begin{proposition}\label{Prop:EXY:cont2D}
For any $r$ in $(0,b)$ and positive real number $t$, there is a constant $K$ that is independent of $t$ so that \[\EX[\|\ve{Y}_t\|_{\max}^r] \leq K t^{\frac{r}{b}}.\]
\end{proposition}

\begin{proof}
Proposition~\ref{Prop:Density:2dM} implies that %
\begin{align}\label{EXY:cont2D}
\EX[\|\ve{Y}_t\|_{\max}^r] & =  \int_{\Qp^2} \|\ve{x}\|_{\max}^r\rho(t,\ve{x})\,{\rm d}\ve{x}\notag\\
& =  \sum_{k\in\ZZ} p^{2k}\Big(\e^{-\sigma t\alpha_{\max}p^{kb}}-\e^{-\sigma t\alpha_{\max}p^{(k+1)b}}\Big)\int_{\Qp^2} \|\ve{x}\|_{\max}^r\indicator{\Ballt(-2k)}(\ve{x})\,{\rm d}\ve{x}\notag\\
& \leq  \sum_{k\in\ZZ} \Big(\e^{-\sigma t\alpha_{\max}p^{kb}}-\e^{-\sigma t\alpha_{\max}p^{(k+1)b}}\Big) p^{-kr}.
\end{align}
Take $a$ to be $\sigma t\alpha_0$ and use Lemma~\ref{lemma:continuousEstimateMoment} to obtain the inequality.

\end{proof}

The resulting process $(D([0,\infty)\colon\Qp),\Pb,\ve{Y})$ is a Brownian motion in $\Qp^2$.  The earlier work \cite{RW:JFAA:2023} of the authors examined (as a special case) Brownian motion in this two dimensional setting.  The component processes of this process are dependent for all time, but they are diffusion processes with the same diffusion constant and Vladimirov exponent $\sigma$.

\subsection{Convergence of the processes}

For each $t$ in $(0, \infty)$, define the sequence $(t_m)$ as before in the one dimensional setting, and determine the characteristic function $\phi_{\ve{X}^{(m)}}$ for $\ve{X}^{(m)}$ as before.  The proofs of Proposition~\ref{prop:2-dMChar} and its corollary are identical to those in the perturbed case to be presented in the next section, with the simplification that  the constant $\alpha_{\max}$ replaces the term $\alpha(\|\y\|)$ that appears later.

\begin{proposition}\label{prop:2-dMChar}
For any $\y$ in $p^{-m}\Zp^2$, the characteristic function $\phi_{\ve{X}^{(m)}}$ for $\ve{X}^{(m)}$ is given by the equality 
\[\phi_{\ve{X}^{(m)}}(\y) = 1 - \frac{\alpha_{\max}\|\y\|_{\max}^b}{p^{mb}}.\]
\end{proposition}

\begin{corollary}
For each natural number $n$, 
\begin{equation*}
\phi^{(m)}_n(\y) = \left(1 - \frac{\alpha_{\max}\|\y\|_{\max}^b}{p^{mb}}\right)^n.
\end{equation*}
\end{corollary}

Take $E_m$ to be the function that is defined for any $(n,\y)$ in $\mathds N_0\times \Qp$ by \[E_m(n, \y) = \begin{cases}  
\left(1 - \frac{\alpha_{\max}\|\y\|_{\max}^b}{p^{mb}}\right)^{n} &\mbox{if }\|\ve{y}\|_{\max} \leq p^{m}\\0 &\mbox{if }\|\ve{y}\|_{\max} > p^m.\end{cases}\]

Take the inverse Fourier transform of the characteristic function for $S^{(m)}_{n}$ and use \eqref{eq:pairing-2d} to obtain the equalities%
\begin{align*}
\rho_{n}^{(m)}(\x_m) &= \int_{p^{-m}\Zp^2} \langle \x_m, \y\rangle_m \left(1-\frac{\alpha_{\max}\|\y\|_{\max}^b}{p^{mb}}\right)^{n}\,\d \y\\
&= \int_{\Qp^2} \langle \ve{x}, \y\rangle E_m(n,\y)\,\d \y.
\end{align*}

Define the function $\rho_m(t, [\cdot]_m)$ for each $\ve{x}$ in $\Qp^2$ by \[\rho_m(t, [\cdot]_m) \colon \ve{x}\mapsto \rho_m\big(\lambda(m)t_m,\x_m\big).\] The proofs of Lemma~\ref{lem:unifcon2-dM} and Theorem~\ref{Sec:Con:Theorem:MAIN2M} similarly follow those from the anisotropic dimension two setting, to be presented in the next section, but already are essentially special cases of the earlier work \cite[Lemma~5.1, Proposition~5.4]{W:Expo:24}.

\begin{lemma}\label{lem:unifcon2-dM}
The sequence of functions $(\rho_m(t,[\cdot]_m))$ converges uniformly on $\Qp$ to the probability density function $\rho(t, \cdot)$.
\end{lemma}

\begin{theorem}\label{Sec:Con:Theorem:MAIN2M}
The sequence of measures $(\Pbm)$ converges weakly to $\Pb$ in $D([0, \infty)\colon \Qp^2)$.
\end{theorem}

%%%%%%%%%%%%%%%%%%%%%%%%%%%%%%%%%%%%%%%%%%%%%%%%%%%%%%%%%%%%%%%%%%%%%%
%%%%%%%%%%%%%%%%%%%%%%%%%%%%%%%%%%%%%%%%%%%%%%%%%%%%%%%%%%%%%%%%%%%%%%

%%%%%%%%%%%%%%%%%%%%%%%%%%%%%%%%%%%%%%%%%%%%%%%%%%%%%%%%%%%%%%%%%%%%%%
%%%%%%%%%%%%%%%%%%%%%%%%%%%%%%%%%%%%%%%%%%%%%%%%%%%%%%%%%%%%%%%%%%%%%%
%%%%%%%%%%%%%%%%%%%%%%%%%%%%%%%%%%%%%%%%%%%%%%%%%%%%%%%%%%%%%%%%%%%%%%
%%%%%%%%%%%%%%%%%%%%%%%%%%%%%%%%%%%%%%%%%%%%%%%%%%%%%%%%%%%%%%%%%%%%%%
%%%%%%%%%%%%%%%%%%%%%%%%%%%%%%%%%%%%%%%%%%%%%%%%%%%%%%%%%%%%%%%%%%%%%%
%%%%%%%%%%%%%%%%%%%%%%%%%%%%%%%%%%%%%%%%%%%%%%%%%%%%%%%%%%%%%%%%%%%%%%
%%%%%%%%%%%%%%%%%%%%%%%%%%%%%%%%%%%%%%%%%%%%%%%%%%%%%%%%%%%%%%%%%%%%%%
%%%%%%%%%%%%%%%%%%%%%%%%%%%%%%%%%%%%%%%%%%%%%%%%%%%%%%%%%%%%%%%%%%%%%%
%%%%%%%%%%%%%%%%%%%%%%%%%%%%%%%%%%%%%%%%%%%%%%%%%%%%%%%%%%%%%%%%%%%%%%
%%%%%%%%%%%%%%%%%%%%%%%%%%%%%%%%%%%%%%%%%%%%%%%%%%%%%%%%%%%%%%%%%%%%%%
%%%%%%%%%%%%%%%%%%%%%%%%%%%%%%%%%%%%%%%%%%%%%%%%%%%%%%%%%%%%%%%%%%%%%%
%%%%%%%%%%%%%%%%%%%%%%%%%%%%%%%%%%%%%%%%%%%%%%%%%%%%%%%%%%%%%%%%%%%%%%

\section{The anisotropic process and its approximation}\label{Sec:Perturbed}

%%%%%%%%%%%%%%%%%%%%%%%%%%%%%%%%%%%%%%%%%%%%%%%%%%%%%%%%%%%%%%%%%%%%%%
%%%%%%%%%%%%%%%%%%%%%%%%%%%%%%%%%%%%%%%%%%%%%%%%%%%%%%%%%%%%%%%%%%%%%%
%%%%%%%%%%%%%%%%%%%%%%%%%%%%%%%%%%%%%%%%%%%%%%%%%%%%%%%%%%%%%%%%%%%%%%
%%%%%%%%%%%%%%%%%%%%%%%%%%%%%%%%%%%%%%%%%%%%%%%%%%%%%%%%%%%%%%%%%%%%%%
%%%%%%%%%%%%%%%%%%%%%%%%%%%%%%%%%%%%%%%%%%%%%%%%%%%%%%%%%%%%%%%%%%%%%%
%%%%%%%%%%%%%%%%%%%%%%%%%%%%%%%%%%%%%%%%%%%%%%%%%%%%%%%%%%%%%%%%%%%%%%
%%%%%%%%%%%%%%%%%%%%%%%%%%%%%%%%%%%%%%%%%%%%%%%%%%%%%%%%%%%%%%%%%%%%%%
%%%%%%%%%%%%%%%%%%%%%%%%%%%%%%%%%%%%%%%%%%%%%%%%%%%%%%%%%%%%%%%%%%%%%%
%%%%%%%%%%%%%%%%%%%%%%%%%%%%%%%%%%%%%%%%%%%%%%%%%%%%%%%%%%%%%%%%%%%%%%
%%%%%%%%%%%%%%%%%%%%%%%%%%%%%%%%%%%%%%%%%%%%%%%%%%%%%%%%%%%%%%%%%%%%%%
%%%%%%%%%%%%%%%%%%%%%%%%%%%%%%%%%%%%%%%%%%%%%%%%%%%%%%%%%%%%%%%%%%%%%%
%%%%%%%%%%%%%%%%%%%%%%%%%%%%%%%%%%%%%%%%%%%%%%%%%%%%%%%%%%%%%%%%%%%%%%

%%%%%%%%%%%%%%%%%%%%%%%%%%%%%%%%%%%%%%%%%%%%%%%%%%%%%%%%%%%%%%%%%%%%%%
%%%%%%%%%%%%%%%%%%%%%%%%%%%%%%%%%%%%%%%%%%%%%%%%%%%%%%%%%%%%%%%%%%%%%%

\subsection{The perturbed primitive process}

%%%%%%%%%%%%%%%%%%%%%%%%%%%%%%%%%%%%%%%%%%%%%%%%%%%%%%%%%%%%%%%%%%%%%%
%%%%%%%%%%%%%%%%%%%%%%%%%%%%%%%%%%%%%%%%%%%%%%%%%%%%%%%%%%%%%%%%%%%%%%

The weighted max-norm uniquely identifies the balls in $\G^2$ and $\Zp^2$ in the fine Pontryagin filtration $\big(\Pon_{\rm f}(\G^2), \Pon_{\rm f}(\Zp^2)\big)$ by their radii.  For each $[(x_1,x_2)]$ in $\G^2$, define \[\|[(x_1, x_2)]\|_h = \begin{cases}\|(x_1,x_2)\|_h & \text{ if } (x_1, x_2) \not\in \Zp\times \Zp\\0 & \text{ otherwise.}\end{cases}\] For each $i$ in $\No$, write \[r_i = \rad\Big(\BBo(i)\Big) \quad \text{and} \quad \hat{r}_i = \rad\Big(\BBo(i)^\perp\Big),\] 
so that for each $k$ in $\No$, \begin{equation}\label{eq:riandhatri}r_i = \begin{cases}p^{k+\h}  &\text{ if } i = 2k+1, \; k\\[.00in] p^{k+1} &\text{ if } i = 2(k+1),\\[.00in]0 & \text{ if }i = 0.\end{cases} \quad \text{and} \quad \hat{r}_i = \begin{cases}p^{-k-1+\h}  &\text{ if } i = 2k+1\\[.00in] p^{-(k+1)} &\text{ if } i = 2(k+1)\\[.00in]1 & \text{ if }i = 0.\end{cases}\end{equation}

Fixing specified probabilities associated to the fine shells of the balls in $\Pon_{\rm f}(\G^2)$ specifies a probability mass function $\rho_{\ve{X}}$ for a random variable $\ve{X}$ that gives rise to the increments of a random walk on $\G^2$ by 
\[\rho_{\ve{X}}(\x) = \begin{cases}\frac{\ConTd}{r_k^b} \frac{1}{\Vol\big(\Shellof(k)\big)} &\text{ if } \x \in \Shellof(k), \; \forall k\in \NN\\[.05in] 0 &\text{ if } \x \in \Shellof(0),\end{cases}\] 
so that \[\Prob\Big(\ve{X}\in\Shellof(i)\Big) = \begin{cases}\frac{\ConTd}{r_i^b}  &\text{ if } i\in \NN\\[.05in]0 &\text{ if } i = 0. \end{cases}\] %
The equalities %
\begin{align*}
1 = \ConTd\Bigg\{\frac{1}{p^{\h b}} + \frac{1}{p^{b}} + \frac{1}{p^{(1+\h)b}}  + \frac{1}{p^{2b}} + \cdots \Bigg\}
= \ConTd\Bigg(\frac{1}{p^{\h b}}+ \frac{1}{p^{b}}\Bigg)\frac{p^b}{p^b - 1},
\end{align*}
imply that %
\begin{equation}\label{Gamma:def}\ConTd = \Bigg(\frac{1}{p^{\h b}}+ \frac{1}{p^{b}}\Bigg)^{-1}\frac{p^b-1}{p^b}.\end{equation}

For each $i$ in $\No$, write $\rho(i)$ to indicate the quantity $\rho_X(\x)$, where $\x$ is in $\Shellof(i)$. Denote %
\[\alpha_0 = 1 + \frac{1}{p^{b}+p^{\h b}}\frac{p^b-1}{p-1}\quad \text{and} \quad \alpha_1 = \Bigg\{1  + \frac{p^{\h b}}{p^{\h b} + p^{b}}\frac{(p^b-1)p}{p-1}\Bigg\}\frac{p^{\h b}}{p^{2b}}.\] For any $k$ in $\ZZ$, define the function $\alpha$ on $\Qp^2$ (hence by restriction on $\Zp^2$) by \[\alpha(\ve{y}) = \begin{cases}\alpha_0 &\text{ if } \ve{y} \in\Shelltf(-2k)\\ \alpha_1 &\text{ if } \ve{y} \in \Shelltf(-2k-1).\end{cases}\]  As a useful shorthand, for any $i$ in $\ZZ$ write \[\alpha(2i) = \alpha_0 \quad \text{and} \quad \alpha(2i+1) = \alpha_1.\] A straightforward calculation proves Lemma~\ref{alpha10values}.

\begin{lemma}\label{alpha10values}
For any $b$ in $(0,\infty)$ and $h$ in $(0,1)$, the quantities $\alpha_0 -1$, $\alpha_1$, and $\alpha_0p^{-b}$ are in $(0,1)$.
\end{lemma}

Denote by $\phi_{\ve{X}}$ the characteristic function for the random variable $\ve{X}$.

\begin{proposition}\label{prop:charh}
For any $\y$ in $\Zp^2$ and $\varepsilon$ in $\{0,1\}$, \[\phi_{\ve{X}}(\y) = 1-\alpha_h(\y)\|\y\|_h^b.\]
\end{proposition}

\begin{proof}
Take the Fourier transform of the probability mass function $\rho_X$ to obtain for any $\y$ in $\Zp$ the equalities 
\begin{align*}
\phi_{\ve{X}}(\y) & = \int_{\G^2} \langle-\x, \y\rangle \rho_X(\x)\,\d\x %
\\ & = \ConTd\sum_{i\in\NN} \int_{\Shellof(i)} \langle-\x, \y\rangle \frac{1}{\Vol\big(\Shellof(i)\big)} \frac{1}{r_i^b} \,\d\x
\\ & = \ConTd\sum_{i\in\NN} \frac{1}{\Vol\big(\Shellof(i)\big)} \frac{1}{r_i^b}\left(\int_{\BBo(i)} \langle-\x, \y\rangle \,\d\x - \int_{\BBo(i-1)} \langle-\x, \y\rangle \,\d\x\right)
\\ & = -\frac{1}{\Vol\big(\Shellof(1)\big)} \frac{\ConTd}{r_1^b}\Vol\big(\BBo(0)\big){\mathds 1}_{\Gd}(\y) \\ &\qquad + \sum_{i\in\NN} \ConTd\left(\frac{1}{r_i^b}\frac{1}{\Vol\big(\Shellof(i)\big)} - \frac{1}{r_{i+1}^b}\frac{1}{\Vol\big(\Shellof(i+1)\big)}\right) \Vol\big(\BBo(i)\big){\mathds 1}_{\Ballt(i)^\perp}(\y)%%
\end{align*}

Use the equalities \[\frac{1}{\Vol\big(\Shellof(i)\big)} = \frac{1}{\Vol\big(\BBo(i)\big)} \frac{p}{p-1}\quad \text{and} \quad \frac{1}{\Vol\big(\Shellof(i+1)\big)} = \frac{1}{\Vol\big(\BBo(i)\big)}\frac{1}{p-1}\] to rewrite $\phi_X(\y)$ as %
\begin{align*}
\phi_{\ve{X}}(\y) & = -\frac{1}{p-1} \frac{\ConTd}{r_1^b} + \sum_{i\in \NN} \left(\frac{\ConTd}{r_i^b}\frac{p}{p-1} - \frac{\ConTd}{r_{i+1}^b}\frac{1}{p-1}\right) {\mathds 1}_{\Ballt(i)^\perp}(\y) \\
& = \ConTd\left\{-\frac{1}{p-1} \frac{1}{r_1^b} + \frac{1}{p-1}\sum_{i\in \NN} \left(\frac{p}{r_i^b} - \frac{1}{r_{i+1}^b}\right) {\mathds 1}_{\Ballt(i)^\perp}(\y)\right\}.
\end{align*} 
If $\y$ is in $\Shellof(2k)$ for some $k$ in $\NN$, then
\begin{align*}
\phi_{\ve{X}}(\y) &= \ConTd\Bigg\{-\frac{1}{p-1} \frac{1}{p^{\h b}} + \frac{1}{p-1}\Bigg(\Big(\frac{p}{p^{\h b}} - \frac{1}{p^b}\Big) + \Big(\frac{p}{p^{b}} - \frac{1}{p^{(1+\h)b}}\Big) + \Big(\frac{p}{p^{(1+\h)b}} - \frac{1}{p^{2b}}\Big) \\ & \qquad + \Big(\frac{p}{p^{2b}} - \frac{1}{p^{(2+\h)b}}\Big)+ \Big(\frac{p}{p^{(2+\h)b}} - \frac{1}{p^{3b}}\Big) + \cdots + \Big(\frac{p}{p^{kb}} - \frac{1}{p^{(k+\h)b}}\Big)\Bigg)\Bigg\}\\%
& = 1 - \Bigg\{1 + \frac{1}{p^{b}+p^{\h b}}\frac{p^b-1}{p-1}\Bigg\}\frac{1}{p^{kb}} = 1 - \alpha_0\|\ve{y}\|_h^b.
\end{align*}
If $\y$ is in $\Shellof(2k+1)$ for some $k$ in $\NN$, then
\begin{align*}
\phi_{\ve{X}}(\y) & = \ConTd\Bigg\{-\frac{1}{p-1}\frac{1}{p^{\h b}} + \frac{1}{p-1}\Bigg(\Big(\frac{p}{p^{\h b}} - \frac{1}{p^b}\Big) + \Big(\frac{p}{p^{b}} - \frac{1}{p^{(1+\h)b}}\Big)\\
 & \qquad + \Big(\frac{p}{p^{(1+\h)b}} - \frac{1}{p^{2b}}\Big) + \Big(\frac{p}{p^{2b}} - \frac{1}{p^{(2+\h)b}}\Big) + \Big(\frac{p}{p^{(2+\h)b}} - \frac{1}{p^{3b}}\Big)\\
& \qquad + \cdots + \Big(\frac{p}{p^{kb}} - \frac{1}{p^{(k+\h)b}}\Big) + \Big(\frac{p}{p^{(k+\h)b}} - \frac{1}{p^{(k+1)b}}\Big)\Bigg)\Bigg\}\\%
& = 1 - \Bigg\{1  + \frac{p^{\h b}}{p^{\h b} + p^{b}}\frac{(p^b-1)p}{p-1}\Bigg\}\frac{p^{\h b}}{p^{2b}}\frac{1}{p^{(k+(\h -1))b}}  = 1 - \alpha_1\|\y\|_h^b.%%
\end{align*}
Consolidate the results to obtain the equalities \[\begin{cases}\phi_{\ve{X}}(\y) = 1-\alpha(\y)\|\y\|_h^b&\text {}\\\phi_{\ve{X}}(\ve{0}) = 1,& \text{}\end{cases}\] since $\phi_{\ve{X}}$ is the Fourier transform of a probability mass function.

\end{proof}

Take $({\ve{X}}_i)$ to be a sequence of independent random variables with the same distribution as ${\ve{X}}$, the random variable ${\ve{X}}_0$ to be almost surely equal to $[\ve{0}]$, and \[{\ve{S}}_n = {\ve{X}}_0 + {\ve{X}}_1 + \cdots + {\ve{X}}_n.\] The probability mass function $\rho_\ast(n,\cdot)$ for ${\ve{S}}_n$ is given by  
\begin{align*}
\rho_\ast(n,\x) & = \int_{\Zp^2} \langle \x, \y\rangle \phi_X(\y)^n\;\d\y
\\& = \sum_{i\in \No} \int_{\Shelltf(-i)} \Big(1- \alpha(\y)\|\y\|^b\Big)^n\langle \x, \y\rangle\;\d\y
\\& = \sum_{i\in \No} \Big(1- \alpha(i)\hat{r}_i^{b}\Big)^n\Big(\int_{\Ballt(-i)} \langle \x, \y\rangle\;\d\y - \int_{\Ballt(-i-1)} \langle \x, \y\rangle\;\d\y\Big)
\\& = \Big(1- \alpha(0)\Big)^n\int_{(\Zp)^2} \langle \x, \y\rangle\;\d\y\\ & \qquad + \sum_{i\in \No} \Big(\Big(1- \alpha(i+1)\hat{r}_{i+1}^{-b}\Big)^n - \Big(1- \alpha(i)\hat{r}_i^{-b}\Big)^n\Big)\int_{\Ballt(-i-1)} \langle \x, \y\rangle\;\d\y
\\& = \Big(1- \alpha(0)\Big)^n{\mathds 1}_{\Ballt(0)}(\x) \\&\qquad+ \sum_{i\in \No} \Big(\Big(1- \alpha(i+1)\hat{r}_{i+1}^{-b}\Big)^n - \Big(1- \alpha(i)\hat{r}_i^{-b}\Big)^n\Big)\\&\hspace{3in}\cdot\Vol\Big(\Ballt(-i-1)\Big){\mathds 1}_{\BBo(i+1)}(\x)
\\& = \Big(1- \alpha(0)\Big)^n{\mathds 1}_{\BBo(0)}(\x)%
\\&\qquad+ \sum_{k\in \No} \Big(\Big(1- \alpha(2k+1)\hat{r}_{2k+1}^{-b}\Big)^n - \Big(1- \alpha(2k)\hat{r}_{2k}^{-b}\Big)^n\Big)
\\&\hspace{3in}\cdot\Vol\Big(\Ballt(-2k-1)\Big){\mathds 1}_{\BBo(2k+1)}(\x)
\\&\qquad+ \sum_{k\in \No} \Big(\Big(1- \alpha(2k+2)\hat{r}_{2k+2}^{-b}\Big)^n - \Big(1- \alpha(2k+1)\hat{r}_{2k+1}^{-b}\Big)^n\Big)
\\&\hspace{3in}\cdot\Vol\Big(\Ballt(-2k-2)\Big){\mathds 1}_{\BBo(2k+2)}(\x)
\end{align*}
Use the formula for the volume of balls in $\Zp^2$ and \eqref{eq:riandhatri} to obtain Proposition~\ref{Prop:FormulaForChar:WtCase}.

\begin{proposition}\label{Prop:FormulaForChar:WtCase}
For any $\x$ in $\Zp^2$, %
\begin{align*}
\rho_\ast(n,\x) & =\Big(1- \alpha(0)\Big)^n{\mathds 1}_{\BBo(0)}(\x)%
\\&\qquad+ \sum_{k\in \No} \Big(\Big(1- \alpha_1p^{-(k+1-h)b}\Big)^n - \Big(1- \alpha_0p^{-kb}\Big)^n\Big)p^{-(2k+1)}{\mathds 1}_{\BBo(2k+1)}(\x)
\\&\qquad+ \sum_{k\in \No} \Big(\Big(1- \alpha_0p^{-(k+1)b}\Big)^n - \Big(1- \alpha_1p^{-(k+1-h)b}\Big)^n\Big)p^{-2(k+1)}{\mathds 1}_{\BBo(2k+2)}(\x).
\end{align*}
\end{proposition}

Denote by $\PbG$ the measure on the space $F(\No\colon \G^2)$ that the convolution semigroup of probability mass functions $(\rho_\ast(n,\cdot))$ determines.  The stochastic process $(F(\No\colon \G^2), \PbG, S)$ is the \emph{fine primitive process}.

%%%%%%%%%%%%%%%%%%%%%%%%%%%%%%%%%%%%%%%%%%%%%%%%%%%%%%%%%%%%%%%%%%%%%%
%%%%%%%%%%%%%%%%%%%%%%%%%%%%%%%%%%%%%%%%%%%%%%%%%%%%%%%%%%%%%%%%%%%%%%

\subsection{Moment estimates}

%%%%%%%%%%%%%%%%%%%%%%%%%%%%%%%%%%%%%%%%%%%%%%%%%%%%%%%%%%%%%%%%%%%%%%
%%%%%%%%%%%%%%%%%%%%%%%%%%%%%%%%%%%%%%%%%%%%%%%%%%%%%%%%%%%%%%%%%%%%%%

Denote by $\EXG$ the expected value with respect to $\PbG$.

\begin{theorem}\label{thm:MomExthprim}
For any $n$ in $\No$ and any $r$ in $(0,b)$, there is a constant $K$ so that \[\EXG\big[\|\ve{S}_n\|_h^r\big] \leq Kn^{\frac{r}{b}}.\]
\end{theorem}

\begin{proof}

Compress notation by writing \[A(2k+1) = \Big(1- \alpha_1p^{-(k+1-h)b}\Big)^n - \Big(1- \alpha_0p^{-kb}\Big)^n\] and \[A(2k+2) = \Big(1- \alpha_0p^{-(k+1)b}\Big)^n - \Big(1- \alpha_1p^{-(k+1-h)b}\Big)^n\] to see that
\begin{align}\label{ineq:Aeven-Aodd}
\rho_\ast(n,\x) & = \Big(1- \alpha(0)\Big)^n{\mathds 1}_{\BBo(0)}(\x)+ \sum_{k\in \No} A(2k+1)p^{-(2k+1)}{\mathds 1}_{\BBo(2k+1)}(\x)
\notag\\&\qquad+ \sum_{k\in \No} A(2k+2)p^{-(2k+2)}{\mathds 1}_{\BBo(2k+2)}(\x)
\notag\\& \leq \Big(1- \alpha(0)\Big)^n{\mathds 1}_{\BBo(0)}(\x)+ \sum_{k\in \No} A(2k+1)p^{-(2k+1)}{\mathds 1}_{\BBo(2k+2)}(\x)
\\&\qquad+ \sum_{k\in \No} A(2k+2)p^{-(2k+2)}{\mathds 1}_{\BBo(2k+2)}(\x).\notag
\end{align}

The inequality \eqref{ineq:Aeven-Aodd} implies that
\begin{align*}
\EXG\big[\|\ve{S}_n\|_h^r\big]  & \leq \int_{\G^2} \|x\|^r\sum_{k\in \No}\Big(A(2k+1)p^{-(2k+1)} + A(2k+2)p^{-(2k+2)}\Big)\\
&\hspace{3.5in}\cdot{\mathds 1}_{\BBo(2k+2)}(\x) \,\d\x\\
& \leq \sum_{k\in \No}p^{-(2k+1)}\Big(A(2k+1)+ A(2k+2)\Big)\int_{\BBo(2k+2)} \|x\|^r \,\d\x\\
& = \sum_{k\in \No}p^{-(2k+1)}\Big(\Big(1- \alpha_0p^{-(k+1)b}\Big)^n- \Big(1- \alpha_0p^{-kb}\Big)^n\Big)
\sum_{i=1}^{2k+2}\int_{\Shellof(i)} \|x\|^r \,\d\x \\%
& \leq 2\sum_{k\in \No}p^{-(2k+1)}\Big(\Big(1- \alpha_0p^{-(k+1)b}\Big)^n- \Big(1- \alpha_0p^{-kb}\Big)^n\Big)
\sum_{i=1}^{k+1}\int_{\Shelloc(2i)} \|x\|^r \,\d\x \\
& = 2\sum_{k\in \No}p^{-(2k+1)}\Big(\Big(1- \alpha_0p^{-(k+1)b}\Big)^n- \Big(1- \alpha_0p^{-kb}\Big)^n\Big)\\
&\hspace{2in}\cdot\Big(\frac{p^2}{p^2-1}\Big)\Big(p^{1r}p^2 + p^{2r}p^4 + \cdots + p^{(k+1)r}p^{2(k+1)}\Big)\\%
& = 2p\frac{p^{r+2}}{p^{r+2}-1}\frac{p^2}{p^2-1}\sum_{k\in \No}\Big(\Big(1- \alpha_0p^{-(k+1)b}\Big)^n- \Big(1- \alpha_0p^{-kb}\Big)^n\Big)\\
&\hspace{3.5in}\cdot\Big(p^{(k+1)r}-p^{-2(k+1)}\Big).%
\end{align*}
Lemma~\ref{maxnorm:prim-mom-est-lemma} therefore implies the existence of the desired upper bound on the expected value.

\end{proof}

%%%%%%%%%%%%%%%%%%%%%%%%%%%%%%%%%%%%%%%%%%%%%%%%%%%%%%%%%%%%%%%%%%%%%%
%%%%%%%%%%%%%%%%%%%%%%%%%%%%%%%%%%%%%%%%%%%%%%%%%%%%%%%%%%%%%%%%%%%%%%

\subsection{Embedding the primitive process}

%%%%%%%%%%%%%%%%%%%%%%%%%%%%%%%%%%%%%%%%%%%%%%%%%%%%%%%%%%%%%%%%%%%%%%
%%%%%%%%%%%%%%%%%%%%%%%%%%%%%%%%%%%%%%%%%%%%%%%%%%%%%%%%%%%%%%%%%%%%%%

Define the process $\ve{S}^{(m)}$ with state space $(\Gm)^2$ by \[\ve{S}^{(m)} = Q_m^{(2)}\circ \ve{S}.\]  This process has generator $\ve{X}^{(m)}$ that is given by \[\ve{X}^{(m)} = Q_m^{(2)}\circ \ve{X},\] and has a law given by \eqref{eq:newdensitydiscrete2d}, but with respect to the probabilities defined on the fine shells. %
For each $n$, the law $\rho_m\big(n,\cdot\big)$ for the random variable $\ve{S}_n^{(m)}$ is similarly given by \eqref{eq:newdensitydiscretenstep2d}, but again with respect to the probabilities for $\ve{X}$ defined on the fine shells.%

The probabilities that $\ve{S}^{(m)}$ determines for the simple cylinder sets extends to a probability measure $\Pbm$ on $F(\No\colon (\Gm)^2)$. Denote by $\EXm$ the expected value with respect to $\Pbm$.

\begin{proposition}\label{MomEstX1DdifDh}
There is a constant $K$ independent of $m$ and $n$ so that for any $m$ and $n$ in $\No$, if $r$ is in $(0,b)$, then
\[\EXm\Big[\|\ve{S}^{(m)}_n\|^r_h\Big] \leq Kp^{-rm}n^{\frac{r}{b}}.\]
\end{proposition}

\begin{proof}
Integrate $\|\ve{S}^{(m)}_n\|^r_h$ and use the analog of \eqref{eq:newdensitydiscretenstep2d} in the case where the probabilities are with respect to the fine shells to obtain the equalities
\begin{align*}
\EXm\Big[\|\ve{S}^{(m)}_n\|^r_h\Big] = \int_{(\Gm)^2} \|[\ve{x}]_m\|^r_h\rho_m(n,[\ve{x}]_m) \,{\rm d}[\ve{x}]_m = \int_{\Qp^2} \|[\ve{x}]_m\|^r_h\rho_m\big(n,[\ve{x}]_m\big) \,{\rm d}\ve{x}.%
\end{align*}
Equation~\eqref{eq:newdensitydiscretenstep2d} together with a change of variables implies that 
\begin{align*}
\EXm\Big[\|\ve{S}^{(m)}_n\|^r_h\Big] & = \int_{\Qp^2} \|[\ve{x}]_m\|^r_h\rho_\ast\big(n,[p^{-m}\ve{x}]\big) p^{2m}\,{\rm d}\ve{x}\\%
& = \int_{\Qp^2} \|[\ve{x}]_m\|^r_h\rho_\ast\big(n,[p^{-m}\ve{x}]\big)\,{\rm d}(p^{-m}\ve{x})
 = \int_{\Qp^2} \|[p^{m}\ve{x}]_m\|^r_h\rho_\ast\big(n,[\ve{x}]\big) \,{\rm d}\ve{x}.
 \end{align*}
 Theorem~\ref{thm:MomExthprim} implies that
 \begin{align*}
\EXm\Big[\|\ve{S}^{(m)}_n\|^r_h\Big]  = p^{-rm}\EXG\big[\|\ve{S}_n\|^r_h\big] \leq p^{-rm}Kn^{\frac{r}{b}}.
\end{align*}

\end{proof}

Denote by $\Gamma_m^{(2)}$ and $\iota_m^{(2)}$ the same functions as in the two dimensional max-norm case.  The function $\iota_m^{(2)}$ acts on the fine primitive process $\ve{S}_n$ to produce a continuous time process in $\Qp^2$ by using \eqref{EmbeddedProcessFDD}.  Denote again by $\Pbm$ the probability measure on $F([0,\infty)\colon \Qp^2)$ obtained in this way, and by $\EXm$ the expected value with respect to $\Pbm$.

\begin{theorem}\label{MomEstX1DdifDQp}
For any $m$ in $\No$ and $t$ in $[0,\infty)$, there is a constant $K$ so that \[\EXm\big[\|\ve{Y}_t\|_h^r\big] \leq Kt^{\frac{r}{b}}.\]
\end{theorem}

\begin{proof}
Proposition~\ref{MomEstX1DdifDh} and \eqref{eq:YttoSmExp-1D} together imply that there is a constant $K^\prime$ so that %
\begin{align*}
\EXm\big[\|\ve{Y}_t\|^r_h\big] & = \EXm\Big[\big\|\ve{S}^{(m)}_{\floor{\lambda(m) t}}\big\|^r_h\Big]\\
& \leq p^{-rm}K^\prime(\floor{\lambda(m)t})^{\frac{r}{b}}\\
& \leq p^{-rm}K^\prime(Dp^{mb}t)^{\frac{r}{b}} = K^\prime D^{\frac{r}{b}}t^{\frac{r}{b}}.
\end{align*}

\end{proof}

The proofs of Propositions~\ref{tight:2D} and \ref{prop:qp:concenration2} follow the same reasoning as those of Propositions~\ref{tight:1D} and \ref{prop:qp:concenration}.

\begin{proposition}\label{tight:2D}
The stochastic process $(F([0,\infty)\colon \Qp^2), \Pbm, \ve{Y})$ has a version with paths in $D([0,\infty)\colon \Qp^2)$, a process $(D([0,\infty)\colon \Qp^2), \Pbm, \ve{Y})$.  The sequence of measures $(\Pbm)$ with paths in $D([0,\infty)\colon \Qp^2)$ is uniformly tight.
\end{proposition}

\begin{proposition}\label{prop:qp:concenration2}
The measure $\Pbm$ is concentrated on the subset of paths in $D([0, \infty)\colon \mathds Q_p^2)$ that are valued in $\Gamma_m^{(2)}(\G^2)$ and are, for each natural number $n$, constant on the intervals $\big[(n-1)\tau_m, n\tau_m\big)$.
\end{proposition}

%%%%%%%%%%%%%%%%%%%%%%%%%%%%%%%%%%%%%%%%%%%%%%%%%%%%%%%%%%%%%%%%%%%%%%
%%%%%%%%%%%%%%%%%%%%%%%%%%%%%%%%%%%%%%%%%%%%%%%%%%%%%%%%%%%%%%%%%%%%%%

\subsection{The limiting processes}\label{subsec:LimitProcess2Dh}

%%%%%%%%%%%%%%%%%%%%%%%%%%%%%%%%%%%%%%%%%%%%%%%%%%%%%%%%%%%%%%%%%%%%%%
%%%%%%%%%%%%%%%%%%%%%%%%%%%%%%%%%%%%%%%%%%%%%%%%%%%%%%%%%%%%%%%%%%%%%%

Define the function $M_{h,b}$ for any $\y$ in $\Qp^2$ by \[M_{h,b}(\y) = \alpha(\y)\|\y\|_h^b.\]  The anisotropic Vladimirov operator on $L^2(\Qp^2)$ is given by
\[
\Delta_{h,b} f = \mathcal F^{-1}\big(M_{h,b}\widehat f\;\big)
\]
and has an associated diffusion equation%
\begin{equation}\label{Intro:HeatEquation2dh}
\begin{cases}
\displaystyle \frac{\d}{\d t}\,\rho(t,\ve{x}) = -\sigma\Delta_{h,b}\,\rho(t,\ve{x})\\[4pt]
\rho(0,\cdot)=\delta_0.
\end{cases}
\end{equation}
As in the one dimensional case, the operator $\Delta_{h,b}$ is self-adjoint on its domain \[\mathcal D(\Delta_{h,b}) = \Big\{f\in L^2(\Qp^2)\colon M_{h,b}\widehat f\in L^2(\Qp^2)\Big\}.\] Extend $\Delta_{h,b}$ as in the one dimensional case to act on functions of a non-negative time variable and spatial variable.

The heat kernel
\[
\rho(t,\ve{x})=\big(\mathcal F^{-1} \e^{-\sigma t\alpha(\cdot)\|\cdot\|_h^{b}}\big)(\ve{x})
\]
is the fundamental solution to \eqref{Intro:HeatEquation2dh}, and $\{\rho(t,\cdot)\}_{t>0}$ is a convolution semigroup of probability densities.  Showing that this semigroup determines a probability measure $\Pb$ on $D([0,\infty)\colon\Qp^2)$, concentrated on paths that are initially at $\ve{0}$, requires a more involved argument than in the one dimensional \cite{Varadarajan:LMP:1997} or two dimensional max-norm setting.

\begin{proposition}
For any positive real number $t$, the family $\{\rho(t,\cdot)\}_{t>0}$ forms a convolution semigroup of probability density functions that defines a measure $\Pbm$ on $F([0,\infty)\colon \Qp^2)$.  Furthermore, for each $\ve{x}$, \[\rho(t,\ve{x}) \leq p\sum_{k\in\ZZ} \Big(\e^{-\sigma t\alpha_0p^{kb}}-\e^{-\sigma t\alpha_0p^{(k+1)b}}\Big) p^{-kr}.\]
\end{proposition}

\begin{proof}

Write the inverse Fourier transform of the exponential as an integral over fine shells to obtain the equalities
\begin{align*}
\rho(t,\ve{x}) &= \big(\mathcal F^{-1} \e^{-\sigma t\alpha(\cdot)\|\cdot\|_h^{b}}\big)(\ve{x}) \\ &= \int_{\Qp^2}\langle -\ve{x}, \ve{y}\rangle \e^{-\sigma t\alpha(\ve{y})\|\ve{y}\|_h^{b}}\,{\rm d}\ve{y} = \sum_{k\in\ZZ} \e^{-\sigma t\alpha(k)r_k^{b}}\int_{\Shelltf(k)}\langle -\ve{x}, \ve{y}\rangle\,{\rm d}\ve{y}.
\end{align*}
Reorder terms to obtain the equalities
\begin{align}\label{vsvAdapted2Dh}
\rho(t,\ve{x}) &= \sum_{k\in\ZZ} \e^{-\sigma t\alpha(k)r_k^{b}}\Bigg(\int_{\Ballt(k)}\langle -\ve{x}, \ve{y}\rangle\,{\rm d}\ve{y}-\int_{\Ballt(k-1)}\langle -\ve{x}, \ve{y}\rangle\,{\rm d}\ve{y}\Bigg)\notag\\
&= \sum_{k\in\ZZ} \Big(\e^{-\sigma t\alpha(k)r_k^{b}}-\e^{-\sigma t\alpha(k+1)r_{k+1}^{b}}\Big)\int_{\Ballt(k)}\langle -\ve{x}, \ve{y}\rangle\,{\rm d}\ve{y}\notag\\
&=  \sum_{k\in\ZZ} p^{2k}\Big(\e^{-\sigma t\alpha_0p^{kb}}-\e^{-\sigma t\alpha_1p^{(k+h)b}}\Big)\indicator{\Ballt(-2k)}(\ve{x})\notag\\
& \qquad + \sum_{k\in\ZZ} p^{2k+1}\Big(\e^{-\sigma t\alpha_1p^{(k+h)b}}-\e^{-\sigma t\alpha_0p^{(k+1)b}}\Big)\indicator{\Ballt(-2k-1)}(\ve{x}).
\end{align}
For each positive $t$, each term of the sum that defines $\rho(t, \ve{x})$ is non-negative, so $\rho(t, \ve{x})$ is non-negative.  

Positivity of the differences in the sum in \eqref{vsvAdapted2Dh} implies that 
\begin{align}\label{vsvAdapted2Dh:Est}
\rho(t,\ve{x}) &\leq  \sum_{k\in\ZZ} p^{2k+1}\Big(\e^{-\sigma t\alpha_0p^{kb}}-\e^{-\sigma t\alpha_1p^{(k+h)b}}\Big)\indicator{\Ballt(-2k)}(\ve{x})\notag\\
& \qquad + \sum_{k\in\ZZ} p^{2k+1}\Big(\e^{-\sigma t\alpha_1p^{(k+h)b}}-\e^{-\sigma t\alpha_0p^{(k+1)b}}\Big)\indicator{\Ballt(-2k)}(\ve{x})\notag\\
&=  p\sum_{k\in\ZZ} p^{2k}\Big(\e^{-\sigma t\alpha_0p^{kb}}-\e^{-\sigma t\alpha_0p^{(k+1)b}}\Big)\indicator{\Ballt(-2k)}(\ve{x}).
\end{align}

The function $\rho(t, \cdot)$ has unit mass as the inverse Fourier transform of a function that is $1$ at the origin, and so is a probability density function.  Since the Fourier transform takes products to convolutions, the family forms a semigroup under the convolution operation.  
\end{proof}

Take $\Pb$ to be the measure on $F([0,\infty)\colon \Qp^2)$ that $\{\rho(t,\cdot)\}_{t>0}$ induces by \eqref{eq:ProbDetermination}, and take $\EX$ to be the expected value with respect to the probability density function.

\begin{proposition}
For any $r$ in $(0,b)$, \[\EX[\|\ve{Y}_t\|_h^r] \leq  p^{-kr} \leq p^{1+r}a^{\frac{r}{b}}\Gamma\Big(1-\frac{r}{b}\Big).\]
\end{proposition}

\begin{proof}
For any $r$ in $(0,b)$, Equation~\eqref{vsvAdapted2Dh:Est} implies that%
\begin{align}\label{EXY:cont2D}
\EX[\|\ve{Y}_t\|_h^r] & =  \int_{\Qp^2} \|\ve{x}\|_h^r\rho(t,\ve{x})\,{\rm d}\ve{x}\notag\\
& \leq  p\int_{\Qp^2} \|\ve{x}\|_h^r\sum_{k\in\ZZ} p^{2k}\Big(\e^{-\sigma t\alpha_0p^{kb}}-\e^{-\sigma t\alpha_0p^{(k+1)b}}\Big)\indicator{\Ballt(-2k)}(\ve{x})\,{\rm d}\ve{x}\notag\\
& = p\sum_{k\in\ZZ} p^{2k}\Big(\e^{-\sigma t\alpha_0p^{kb}}-\e^{-\sigma t\alpha_0p^{(k+1)b}}\Big)\int_{\Qp^2} \|\ve{x}\|_h^r\indicator{\Ballt(-2k)}(\ve{x})\,{\rm d}\ve{x}\notag\\
& \leq p\sum_{k\in\ZZ} p^{2k}\Big(\e^{-\sigma t\alpha_0p^{kb}}-\e^{-\sigma t\alpha_0p^{(k+1)b}}\Big)\int_{\Qp^2} \|\ve{x}\|_{\max}^r\indicator{\Ballt(-2k)}(\ve{x})\,{\rm d}\ve{x}\notag\\
& = p^{r+1}\frac{p^2-1}{p^{r+2}-1}\sum_{k\in\ZZ} \Big(\e^{-\sigma t\alpha_0p^{kb}}-\e^{-\sigma t\alpha_0p^{(k+1)b}}\Big) p^{-kr}.
\end{align}

Take \[a = \sigma t\alpha_0\] so that Lemma~\ref{lemma:continuousEstimateMoment} and \eqref{EXY:cont2D} together imply the proposition.

\end{proof}

\subsection{Convergence of the processes}

For each $t$ in $(0, \infty)$, define the sequence $(t_m)$ as before in the one dimensional setting. 

\begin{proposition}\label{prop:1-dChar}
For any $\y$ in $p^{-m}\Zp^2$, the characteristic function $\phi_{\ve{X}^{(m)}}$ for $\ve{X}^{(m)}$ is given by the equality 
\[\phi_{\ve{X}^{(m)}}(\y) = 1 - \frac{\alpha_h(\y)\|\y\|_h^b}{p^{mb}}.\]
\end{proposition}

\begin{proof}
Take the Fourier transform of $\rho_{X^{(m)}}$ and use Lemma~\ref{StateSpaces:Intmoverballs2} to obtain the equalities
\begin{align*} 
\phi_{\ve{X}^{(m)}}(\y) & = \int_{(\Gm)^2} \langle [\ve{x}]_m, \y\rangle_m\rho_{\ve{X}^{(m)}}([\ve{x}]_m) \,\d [\ve{x}]_m = \int_{\Qp^2} \langle [\ve{x}]_m, \y\rangle_m\rho_{\ve{X}^{(m)}}([\ve{x}]_m)\,\d \ve{x}.%
\end{align*}
Equation~\eqref{eq:newdensitydiscrete2d} followed by a change of variables and then \eqref{eq:pairing-2d} together imply that
\begin{align*} %
\phi_{\ve{X}^{(m)}}(\y) & = \int_{\Qp^2} \langle [\ve{x}]_m, \y\rangle_m\rho([p^{-m}\ve{x}])p^{2m}\,\d \ve{x}\\ %
& = \int_{\Qp^2} \langle [p^m\ve{x}]_m, \y\rangle_m\rho([\ve{x}])p^{2m}\,\d (p^m\ve{x})\\ %
& = \int_{\Qp^2} \langle [\ve{x}], p^m\y\rangle\rho([x]) \,\d \ve{x} %
= 1 - \frac{\alpha_h(\y)\|\y\|_h^b}{p^{mb}},
\end{align*}
 where Proposition~\ref{prop:charh} implies the ultimate equality.%
\end{proof}

The Fourier transform takes a $n$-fold convolutions to an $n$-fold products, so for each $n$, the characteristic function $\phi^{(m)}_n$ for $\ve{X}_n^{(m)}$ is an $n$-fold product.

\begin{corollary}
For each natural number $n$, 
\begin{equation*}
\phi^{(m)}_n(\y) = \left(1 - \frac{\alpha(\y)\|\y\|_h^b}{p^{mb}}\right)^n.
\end{equation*}
\end{corollary}

Take $E_m$ to be the function that is defined for any $(n,\y)$ in $\mathds N_0\times \Qp$ by \[E_m(n, \y) = \begin{cases}  
\left(1 - \frac{\alpha(\y)\|\y\|_h^b}{p^{mb}}\right)^{n} &\mbox{if }\|\ve{y}\|_h \leq p^{m}\\0 &\mbox{if }\|\ve{y}\|_h > p^m.\end{cases}\]

Take the inverse Fourier transform of the characteristic function for $S^{(m)}_{n}$ and use \eqref{eq:pairing-2d} to obtain the equalities%Char 
\begin{align*}
\rho_m(n,\x_m) &= \int_{p^{-m}\Zp^2} \langle \x_m, \y\rangle_m \left(1-\frac{\alpha(\y) \|\y\|_h^b}{p^{mb}}\right)^{n}\,\d \y\\
&= \int_{\Qp^2} \langle \ve{x}, \y\rangle E_m(n,\y)\,\d \y.
\end{align*}

For each positive real number $t$, denote by $\rho_m(t,[\cdot]_m)$ the function \[\rho_m\big(t,\x_m\big) = \rho_m\big(\lambda(m)t_m,\x_m\big).\] For any $\y$ in $\Qp^2$, \[E_m(\lambda(m)t_m, \y)) \to \e^{-D\alpha(\y)t\|\y\|_h^b}.\]

\begin{proposition}\label{prop:unifcon1-d}
The sequence of functions $(\rho_m(t,[\cdot]_m))$ converges uniformly on $\Qp^2$ to the probability density function $\rho(t, \cdot)$.
\end{proposition}

\begin{proof}
For any $\ve{x}$ in $\Qp^2$ and any positive real number $t$, %
\begin{align*}
\left|\rho(t,\ve{x}) - \rho_m(t,\x_m)\right| & = \left|\int_{\Qp^2} \langle\ve{x},\y\rangle \e^{-D\alpha(\y) t \|\y\|_h^b}\,\d \y - \int_{\Qp^2} \langle\ve{x},\y\rangle E_m(\lambda(m)t_m,\y)\,\d \y\right|\\%
& \leq \int_{\Qp^2} \left|\e^{-D\alpha(\y) t \|\y\|_h^b}-E_m(\lambda(m)t_m,\y)\right|\,\d \y \to 0
\end{align*}
Independence of integral to the right of the inequality on $\ve{x}$ implies the uniform convergence of $\rho(t, \cdot)$ to $\rho_m(t, [\cdot]_m)$ on $\Qp^2$.

\end{proof}

Follow the proof in the one dimensional case for the weak convergence of $\Pbm$ to $\Pb$ in the present case.  There are no substantial differences.

\begin{theorem}\label{Sec:Con:Theorem:MAIN2}
The sequence of measures $(\Pbm)$ converges weakly to $\Pb$ in $D([0, \infty)\colon \Qp^2)$.
\end{theorem}

%%%%%%%%%%%%%%%%%%%%%%%%%%%%%%%%%%%%%%%%%%%%%%%%%%%%%%%%%%%%%%%%%%%%%%
%%%%%%%%%%%%%%%%%%%%%%%%%%%%%%%%%%%%%%%%%%%%%%%%%%%%%%%%%%%%%%%%%%%%%%
%%%%%%%%%%%%%%%%%%%%%%%%%%%%%%%%%%%%%%%%%%%%%%%%%%%%%%%%%%%%%%%%%%%%%%
%%%%%%%%%%%%%%%%%%%%%%%%%%%%%%%%%%%%%%%%%%%%%%%%%%%%%%%%%%%%%%%%%%%%%%
%%%%%%%%%%%%%%%%%%%%%%%%%%%%%%%%%%%%%%%%%%%%%%%%%%%%%%%%%%%%%%%%%%%%%%
%%%%%%%%%%%%%%%%%%%%%%%%%%%%%%%%%%%%%%%%%%%%%%%%%%%%%%%%%%%%%%%%%%%%%%
%%%%%%%%%%%%%%%%%%%%%%%%%%%%%%%%%%%%%%%%%%%%%%%%%%%%%%%%%%%%%%%%%%%%%%
%%%%%%%%%%%%%%%%%%%%%%%%%%%%%%%%%%%%%%%%%%%%%%%%%%%%%%%%%%%%%%%%%%%%%%
%%%%%%%%%%%%%%%%%%%%%%%%%%%%%%%%%%%%%%%%%%%%%%%%%%%%%%%%%%%%%%%%%%%%%%
%%%%%%%%%%%%%%%%%%%%%%%%%%%%%%%%%%%%%%%%%%%%%%%%%%%%%%%%%%%%%%%%%%%%%%
%%%%%%%%%%%%%%%%%%%%%%%%%%%%%%%%%%%%%%%%%%%%%%%%%%%%%%%%%%%%%%%%%%%%%%
%%%%%%%%%%%%%%%%%%%%%%%%%%%%%%%%%%%%%%%%%%%%%%%%%%%%%%%%%%%%%%%%%%%%%%

\section{Component processes}\label{Sec:Components}

%%%%%%%%%%%%%%%%%%%%%%%%%%%%%%%%%%%%%%%%%%%%%%%%%%%%%%%%%%%%%%%%%%%%%%
%%%%%%%%%%%%%%%%%%%%%%%%%%%%%%%%%%%%%%%%%%%%%%%%%%%%%%%%%%%%%%%%%%%%%%
%%%%%%%%%%%%%%%%%%%%%%%%%%%%%%%%%%%%%%%%%%%%%%%%%%%%%%%%%%%%%%%%%%%%%%
%%%%%%%%%%%%%%%%%%%%%%%%%%%%%%%%%%%%%%%%%%%%%%%%%%%%%%%%%%%%%%%%%%%%%%
%%%%%%%%%%%%%%%%%%%%%%%%%%%%%%%%%%%%%%%%%%%%%%%%%%%%%%%%%%%%%%%%%%%%%%
%%%%%%%%%%%%%%%%%%%%%%%%%%%%%%%%%%%%%%%%%%%%%%%%%%%%%%%%%%%%%%%%%%%%%%
%%%%%%%%%%%%%%%%%%%%%%%%%%%%%%%%%%%%%%%%%%%%%%%%%%%%%%%%%%%%%%%%%%%%%%
%%%%%%%%%%%%%%%%%%%%%%%%%%%%%%%%%%%%%%%%%%%%%%%%%%%%%%%%%%%%%%%%%%%%%%
%%%%%%%%%%%%%%%%%%%%%%%%%%%%%%%%%%%%%%%%%%%%%%%%%%%%%%%%%%%%%%%%%%%%%%
%%%%%%%%%%%%%%%%%%%%%%%%%%%%%%%%%%%%%%%%%%%%%%%%%%%%%%%%%%%%%%%%%%%%%%
%%%%%%%%%%%%%%%%%%%%%%%%%%%%%%%%%%%%%%%%%%%%%%%%%%%%%%%%%%%%%%%%%%%%%%
%%%%%%%%%%%%%%%%%%%%%%%%%%%%%%%%%%%%%%%%%%%%%%%%%%%%%%%%%%%%%%%%%%%%%%

%%%%%%%%%%%%%%%%%%%%%%%%%%%%%%%%%%%%%%%%%%%%%%%%%%%%%%%%%%%%%%%%%%%%%%
%%%%%%%%%%%%%%%%%%%%%%%%%%%%%%%%%%%%%%%%%%%%%%%%%%%%%%%%%%%%%%%%%%%%%%

\subsection{Components of the isotropic process}

%%%%%%%%%%%%%%%%%%%%%%%%%%%%%%%%%%%%%%%%%%%%%%%%%%%%%%%%%%%%%%%%%%%%%%
%%%%%%%%%%%%%%%%%%%%%%%%%%%%%%%%%%%%%%%%%%%%%%%%%%%%%%%%%%%%%%%%%%%%%%

Take $\ve{X}$ to be the generator for the max-norm primitive process and write \[\ve{X} = (X^{(1)}, X^{(2)}),\] so that $X^{(1)}$ and $X^{(2)}$ are the generators of the component processes.

\begin{proposition}\label{prop:componentsmaxprob}
For any $k$ in $\NN$ and $i$ in $\{1,2\}$,
\begin{equation}\label{eq:aprop:componentsmaxprob}
\Prob\big(X^{(i)}\in\Shello(k)\big) = \frac{(p^b-1)\Big(p-1\Big)}{\Big(p^2-1\Big)}\Bigg(1 + (p-1)\frac{p^{(b+1)}}{p^{(b+1)}-1}\Bigg)\frac{1}{p^{kb}}
\end{equation}
and
\begin{equation}\label{eq:bprop:componentsmaxprob}
\Prob\big(X^{(i)}\in\Ballo(0)\big) = \frac{p(p-1)}{(p^2-1)}\frac{(p^b-1)}{(p^{b+1}-1)}  \coloneqq P_0(\max).
\end{equation}
\end{proposition}

\begin{proof}
For any $k$ in $\NN$, the probability $\Prob\big(X^{(2)}\in\Shello(k)\big)$ is given by 
\begin{align*}
\Prob(X^{(2)}\in\Shello(k)) &= \Prob\big(X^{(1)}\in\Ballo(k-1)\cap X^{(2)}\in\Shello(k)\big) \\& \qquad \qquad + \Prob\big(X^{(1)}\in\Ballo(k)^c\cap X^{(2)}\in\Shello(k)\big)\\%
&= \Prob\big(X^{(1)}\in\Ballo(k-1)\cap X^{(2)}\in\Shello(k)\big) \\& \qquad \qquad + \sum_{\ell = k}^\infty \Prob\big(X^{(1)}\in\Shello(\ell)\cap X^{(2)}\in\Shello(k)\big)\\
&= \Prob\big(\ve{X}\in\Shelloc(2k)\big)\frac{\Vol\big(\Ballo(k-1) \times \Shello(k)\big)}{\Vol\big(\Shelloc(2k)\big)} \\& \qquad \qquad + \sum_{\ell = k}^\infty \Prob\big(\ve{X}\in\Shelloc(2\ell)\big)\frac{\Vol\big(\Shello(\ell) \times \Shello(k)\big)}{\Vol\big(\Shelloc(2\ell)\big)}\\%
& = \frac{\ConMG}{p^{kb}}\frac{1}{p}\frac{\Big(1-\frac{1}{p}\Big)}{\Big(1-\frac{1}{p^2}\Big)} +  \sum_{\ell = k}^\infty \frac{\ConMG}{p^{\ell b}} \frac{p^k}{p^\ell}\frac{\Big(1-\frac{1}{p}\Big)^2}{\Big(1-\frac{1}{p^2}\Big)}\\ 
& = \frac{\ConMG\Big(p-1\Big)}{\Big(p^2-1\Big)}\Bigg(1 + (p-1)\frac{p^{(b+1)}}{p^{(b+1)}-1}\Bigg)\frac{1}{p^{kb}}.
\end{align*}
Use the value of $\ConMG$ determined by \eqref{ConMG} to obtain \eqref{eq:aprop:componentsmaxprob}.

The component processes have a non-zero probability of remaining at the origin in a single step.  Calculate this probability by
\begin{align*}%\label{ProbX2max}
\Prob\big(X^{(2)}\in\Ballo(0)\big) & = \sum_{\ell = 1}^\infty \Prob\big(X^{(1)}\in\Shello(\ell)\cap X^{(2)}\in\Ballo(0)\big)\\
& = \sum_{\ell = 1}^\infty \Prob\big(\ve{X}\in\Shelloc(2\ell)\big)\frac{\Vol\big(X^{(1)}\in\Shello(\ell)\cap X^{(2)}\in\Ballo(0)\big)}{\Vol\big(\Shelloc(2\ell)\big)}\\
& = \sum_{\ell = 1}^\infty \frac{\ConMG}{p^{\ell b}} \frac{1}{p^\ell} \frac{(1-\frac{1}{p})}{(1-\frac{1}{p^2})}
 =  \ConMG\frac{p(p-1)}{(p^2-1)}\frac{1}{(p^{b+1}-1)}.
\end{align*}
Symmetry in the calculations implies that %
\begin{equation*}%\label{ProbX1max}
\Prob\big(X^{(1)}\in\Ballo(0)\big) = \Prob\big(X^{(2)}\in\Ballo(0)\big).
\end{equation*} %
\end{proof}

%%%%%%%%%%%%%%%%%%%%%%%%%%%%%%%%%%%%%%%%%%%%%%%%%%%%%%%%%%%%%%%%%%%%%%
%%%%%%%%%%%%%%%%%%%%%%%%%%%%%%%%%%%%%%%%%%%%%%%%%%%%%%%%%%%%%%%%%%%%%%

\subsection{Components of the anisotropic processes}

%%%%%%%%%%%%%%%%%%%%%%%%%%%%%%%%%%%%%%%%%%%%%%%%%%%%%%%%%%%%%%%%%%%%%%
%%%%%%%%%%%%%%%%%%%%%%%%%%%%%%%%%%%%%%%%%%%%%%%%%%%%%%%%%%%%%%%%%%%%%%

Take $\ve{X}$ to be the generator for the anisotropic primitive process with parameter $h$ in $(0,1)$ and write \[\ve{X} = \big(X^{(1)}, X^{(2)}\big),\] so that $X^{(1)}$ and $X^{(2)}$ are the generators of the component processes.  

\begin{proposition}\label{prop:componentsmaxprob}
For any $k$ in $\NN$ and $i$ in $\{1,2\}$,
\begin{equation}\label{eq:1aprop:componentsmaxprob}
\Prob\big(X^{(1)}\in\Shello(k)\big) = \frac{p^b-1}{p^b + p^{hb}}\Bigg(p^{\h b} + \Big(1-\frac{1}{p}\Big)\frac{p^{b+1}}{p^{b+1}-1}\Bigg)\frac{1}{p^{kb}},
\end{equation}
\begin{equation}\label{eq:2aprop:componentsmaxprob}
\Prob\big(X^{(2)}\in\Shello(k)\big) = \frac{p^b-1}{p^b + p^{hb}}\Bigg(p^{\h b} + \frac{(p-1)p^{2hb}}{p^{b+1}-1}\Bigg)\frac{1}{p^{(k-1+h)b}},
\end{equation}
\begin{equation}\label{eq:1bprop:componentsmaxprob}
\Prob(X^{(1)}\in\Ballo(0)) = \frac{p^{b+1}}{p^b + p^{hb}} \frac{p^b-1}{p^{b+1} - 1} \coloneqq P_0(1),
\end{equation}
and
\begin{equation}\label{eq:2bprop:componentsmaxprob}
\Prob(X^{(2)}\in\Ballo(0)) = \frac{p^{hb}}{p^b + p^{hb}}\frac{p^b-1}{p^{(b+1)}-1} \coloneqq P_0(2).
\end{equation}
\end{proposition}

\begin{proof}
For any $k$ in $\NN$, the probability $\Prob\big(X^{(1)}\in\Shello(k)\big)$ is given by 
\begin{align}\label{ProbX1h}
\Prob\big(X^{(1)}\in\Shello(k)\big) &= \Prob\big(X_1\in\Shello(k)\cap X^{(2)}\in\Ballo(k)\big) \notag\\ & \qquad \qquad + \Prob\big(X^{(1)}\in\Shello(k)\cap X^{(2)}\in\Ballo(k)^c\big)\notag\\%
&= \Prob\big(X^{(1)}\in\Shello(k)\cap X^{(2)}\in\Ballo(k)\big) \notag\\ & \qquad \qquad + \sum_{\ell = k}^\infty \Prob\big(X^{(1)}\in\Shello(k)\cap X^{(2)}\in\Shello(\ell+1)\big)\notag\\
& = \frac{\ConTd}{p^{kb}} + \sum_{\ell = k}^\infty \frac{\ConTd}{p^{(\ell + h)b}} \frac{p^k}{p^\ell}\Big(1-\frac{1}{p}\Big)\notag\\ 
&= \ConTd\Bigg(1 + \Big(1-\frac{1}{p}\Big)\frac{1}{p^{hb}}\frac{p^{b+1}}{p^{b+1}-1}\Bigg)\frac{1}{p^{kb}}.
\end{align}
For any $\ell$ in $\NN$, the equality \[\Prob\big(X^{(1)}\in\Ballo(\ell-1)\cap X^{(2)}\in\Shello(\ell)\big) = \frac{\ConTd}{p^{(\ell-1 + h)b}}\] and the uniformity of the distribution of $\ve{X}$ in $\Ballo(\ell-1)\times\Shello(\ell)$ together imply that
\begin{align}\label{eq:X1hintstep}
\Prob(X^{(1)}\in\Ballo(0)) & = \Prob\big(X^{(1)}\in\Ballo(0)\cap X^{(2)}\in\Ballo(0)\big)\notag\\ &\qquad \qquad \qquad + \sum_{\ell = 1}^\infty \Prob\big(X^{(1)}\in\Ballo(0)\cap X^{(2)}\in\Shello(\ell)\big)\notag\\
& = \sum_{\ell = 1}^\infty \frac{\ConTd}{p^{(\ell-1 + h)b}} \frac{1}{p^{\ell-1}} %
= \frac{\ConTd}{p^{hb}}\frac{p^{b+1}}{p^{b+1}-1}.
\end{align}
Use \eqref{Gamma:def}, \eqref{ProbX1h}, and \eqref{ProbX2h} to obtain \eqref{eq:1aprop:componentsmaxprob} and \eqref{eq:2aprop:componentsmaxprob}.  
%%%%%%%%%%%%%%%%%%%%%%%%%%%%%%%%%%%%%%%%%%%%%%%%%%%%%%%%%%%%%%%%%%%%%%
%%%%%%%%%%%%%%%%%%%%%%%%%%%%%%%%%%%%%%%%%%%%%%%%%%%%%%%%%%%%%%%%%%%%%%

For any $k$ in $\NN$, 

\begin{align}\label{ProbX2h}
\Prob(X^{(2)}\in\Shello(k)) &= \Prob\big(X^{(1)}\in\Ballo(k-1)\cap X^{(2)}\in\Shello(k)\big) \notag\\& \qquad \qquad+ \Prob\big(X^{(1)}\in\Ballo(k-1)^c\cap X^{(2)}\in\Shello(k)\big)\notag\\%
&= \Prob\big(X^{(1)}\in\Ballo(k-1)\cap X^{(2)}\in\Shello(k)\big) \notag\\& \qquad \qquad + \sum_{\ell = k}^\infty \Prob\big(X^{(1)}\in\Shello(\ell)\cap X^{(2)}\in\Shello(k)\big)\notag\\
& = \frac{\ConTd}{p^{(k-1+h)b}} + \sum_{\ell = k}^\infty \frac{\ConTd}{p^{\ell b}} \frac{p^k}{p^\ell}\Big(1-\frac{1}{p}\Big)\notag\\ 
&=  \ConTd\Bigg(1 + \frac{(p-1)p^{hb}}{p^{b+1}-1}\Bigg)\frac{1}{p^{(k-1+h)b}}.
\end{align}
The equalities \eqref{Gamma:def} %
and \eqref{eq:X1hintstep} together imply \eqref{eq:1bprop:componentsmaxprob}.
Furthermore,
\begin{align}\label{ProbX2h0}
\Prob(X^{(2)}\in\Ballo(0)) & = \sum_{\ell = 1}^\infty \Prob\big(X^{(1)}\in\Shello(\ell)\cap X^{(2)}\in\Ballo(0)\big)\notag\\
& = \sum_{\ell = 1}^\infty \frac{\ConTd}{p^{\ell b}} \frac{1}{p^\ell}
= \frac{\ConTd}{p^{(b+1)}-1}.
\end{align}
The equalities \eqref{ProbX2h} and \eqref{ProbX2h0} together with \eqref{Gamma:def} imply \eqref{eq:2bprop:componentsmaxprob}. 

\end{proof}

%%%%%%%%%%%%%%%%%%%%%%%%%%%%%%%%%%%%%%%%%%%%%%%%%%%%%%%%%%%%%%%%%%%%%%
%%%%%%%%%%%%%%%%%%%%%%%%%%%%%%%%%%%%%%%%%%%%%%%%%%%%%%%%%%%%%%%%%%%%%%

\subsection{Diffusion constants}

%%%%%%%%%%%%%%%%%%%%%%%%%%%%%%%%%%%%%%%%%%%%%%%%%%%%%%%%%%%%%%%%%%%%%%
%%%%%%%%%%%%%%%%%%%%%%%%%%%%%%%%%%%%%%%%%%%%%%%%%%%%%%%%%%%%%%%%%%%%%%

Both components of the coarse primitive process have generator $X_{\max}$.  For each $h$, denote by $X^{(1)}$ and $X^{(2)}$ the components of the generator of the fine primitive process.  Denote by $\sigma(\max)$, $\sigma(h,1)$, and $\sigma(h,2)$ the diffusion coefficients for the limiting Brownian motions to which $X_{\max}$, $X^{(1)}$, and $X^{(2)}$ respectively give rise.  The probability for remaining at the origin is different for the different generators, leading to different diffusion constants for the limiting processes.  Although not directly studied here, differing diffusion constants lead to the concrete effect of different first exit time probabilities from balls \cite[Theorem~3.1]{Weisbart:2021}.  

With the scaling parameter $\lambda(m)$ given by \[\lambda(m) = Dp^{mb},\] Equations~\eqref{eq:alpha:allcases} and \eqref{DiffusionConst1-D} together imply that the diffusion constant $\sigma$ for each of the processes satisfies the equality %
\begin{align}\label{eq:sigma:6}
\sigma = D(1 - P_0)\frac{p^{b+1}-1}{p^b(p-1)}.
\end{align}
Take $P_0$ to be either $P_0(\max)$ in \eqref{eq:bprop:componentsmaxprob}, or the probabilities $P_0(1)$, or $P_0(2)$ given by Proposition~\ref{prop:componentsmaxprob} and use \eqref{eq:sigma:6} to obtain Proposition~\ref{prop:componentCrit}.

\begin{proposition}\label{prop:componentCrit}
For any $h$ in $(0,1)$, \[\sigma(h,1) = \frac{D}{p-1}\left(\frac{p\,(p^{hb}+1)}{p^{b}+p^{hb}}-\frac{1}{p^{b}}\right), \quad \sigma(h,2) = D\left(1+\frac{p^b-1}{(p-1)(p^b+p^{hb})}\right),\]
and
\begin{equation}\label{prop:componentCritmax}\sigma(\max) = D\frac{p^{b+2} -1}{p^b(p^2-1)}.\end{equation}
\end{proposition}

\begin{theorem}\label{thm:componentCrit}
For any fixed prime $p$ and $b$ in $(0, \infty)$, there is a fixed gap between each of the diffusion coefficients $\sigma(h,1)$ and $\sigma(h,2)$and the diffusion coefficient $\sigma(\max)$ as $h$ varies in $(0,1)$.
\end{theorem}

\begin{proof}

A straightforward calculation shows that $\sigma(h,1)$ increases in $h$, that $\sigma(h,2)$ decreases in $h$, and that for all $b$ and $h$, \[\sigma(h,1) < \sigma(\max) < \sigma(h,2).\] Take left limits to obtain the equalities, 
\begin{equation*}
\sigma(1^-,1) \coloneqq \lim_{h\to 1^-}\sigma(h,1) = D\left(\frac{p^{b+1}+p-2}{2p^b(p-1)}\right)
\end{equation*}
and
\begin{equation*}
\sigma(1^-,2)  \coloneqq  \lim_{h\to 1^-}\sigma(h,2) =  D\left(1+\frac{p^b-1}{2p^b(p-1)}\right).
\end{equation*}
 Compute differences and use \eqref{prop:componentCritmax} to see that
 \begin{align*}
 \sigma(\max) - \sigma(1^-,1) = D\frac{p^{b}-1}{2p^{b-1}(p+1)} \quad \text{and} \quad \sigma(1^-,2) - \sigma(\max) = D\frac{p^{b}-1}{2p^{b}(p+1)}.
 \end{align*}
Both the first and second differences are larger than $0$ for any positive $b$, equal to $0$ in the limit as $b$ tends to $0$, increasing in $b$, and they tend to $\frac{Dp}{2(p+1)}$ and $\frac{D}{2(p+1)}$, respectively, as $b$ tends to infinity.

\end{proof}

As should be expected, for any $h$, any $b$, and any $p$, $\sigma(\max)$ is the convex combination %
\begin{align*}
\sigma(\max) & = \sigma(h,1)\Big(\frac{1}{p+1}\Big) + \sigma(h,2)\Big(\frac{p}{p+1}\Big).
\end{align*} %
The weighting on the summands comes from the differing masses of consecutive shells.  Two fine shells lose distinction at the right endpoint of $h$, with the smaller shell having $\frac{1}{p}$ times the mass of the larger. The coefficients of summands are those required to uniformly weight the probabilities of the limiting shell subsets.

%%%%%%%%%%%%%%%%%%%%%%%%%%%%%%%%%%%%%%%%%%%%%%%%%%%%%%%%%%%%%%%%%%%%%%
%%%%%%%%%%%%%%%%%%%%%%%%%%%%%%%%%%%%%%%%%%%%%%%%%%%%%%%%%%%%%%%%%%%%%%
%%%%%%%%%%%%%%%%%%%%%%%%%%%%%%%%%%%%%%%%%%%%%%%%%%%%%%%%%%%%%%%%%%%%%%
%%%%%%%%%%%%%%%%%%%%%%%%%%%%%%%%%%%%%%%%%%%%%%%%%%%%%%%%%%%%%%%%%%%%%%
%%%%%%%%%%%%%%%%%%%%%%%%%%%%%%%%%%%%%%%%%%%%%%%%%%%%%%%%%%%%%%%%%%%%%%
%%%%%%%%%%%%%%%%%%%%%%%%%%%%%%%%%%%%%%%%%%%%%%%%%%%%%%%%%%%%%%%%%%%%%%
%%%%%%%%%%%%%%%%%%%%%%%%%%%%%%%%%%%%%%%%%%%%%%%%%%%%%%%%%%%%%%%%%%%%%%
%%%%%%%%%%%%%%%%%%%%%%%%%%%%%%%%%%%%%%%%%%%%%%%%%%%%%%%%%%%%%%%%%%%%%%
%%%%%%%%%%%%%%%%%%%%%%%%%%%%%%%%%%%%%%%%%%%%%%%%%%%%%%%%%%%%%%%%%%%%%%
%%%%%%%%%%%%%%%%%%%%%%%%%%%%%%%%%%%%%%%%%%%%%%%%%%%%%%%%%%%%%%%%%%%%%%
%%%%%%%%%%%%%%%%%%%%%%%%%%%%%%%%%%%%%%%%%%%%%%%%%%%%%%%%%%%%%%%%%%%%%%
%%%%%%%%%%%%%%%%%%%%%%%%%%%%%%%%%%%%%%%%%%%%%%%%%%%%%%%%%%%%%%%%%%%%%%


\begin{thebibliography}{0}

%
\bibitem{Abdesselam:arXiv:2018} Abdesselam, A.: \textsl{Towards three-dimensional conformal probability}. P-Adic Num Ultrametr Anal Appl 10, 233–252 (2018).

%
\bibitem{Auffinger:arXiv:2022} Auffinger, A., Montanari, A., Subag, E.: \textsl{Optimization of random high-dimensional functions: Structure and algorithms}. arXiv:2206.10217, (2022).

%
\bibitem{Avetisov_Bikulov_Kozyrev:JPA:1999}  Avetisov, V.A., Bikulov, A.H., Kozyrev, S.V.: \textsl{Application of $p$-adic analysis to models of spontaneous breaking of replica symmetry.} J. Phys. A: Math. Gen. 32 (50), 8785--8791, (1999).

%
\bibitem{Avetisov:JPA:2002} Avetisov, V.A., Bikulov, A.H., Kozyrev, S.V., Osipov, V.A.: \textsl{$p$-Adic Models of Ultrametric Diffusion Constrained by Hierarchical Energy Landscapes}. \emph{J.~Phys.~A: Math. Gen.} 35, no.~2, 177--189, (2002).

%
\bibitem{Avetisov:JPA:2003} Avetisov, V.A., Bikulov, A.Kh., Osipov, V.Al.: \textsl{$p$-adic description of characteristic relaxation in complex systems}. \emph{J.~Phys.~A: Math. Gen.} 36, no.~15, 4239--4246, (2003).

%
\bibitem{ABZ:ProcSteklov:2014} Avetisov, V.A., Bikulov, A.K., Zubarev, A.P.: \textsl{Ultrametric random walk and dynamics of protein molecules}. Proc. Steklov Inst. Math. 285, 3-25, (2014).

%
\bibitem{BW} Bakken, E., Weisbart, D.:  \textsl{$p$\,-Adic brownian motion as a limit of discrete time random walks}. Commun. Math. Phys. 369, 371-402, (2019).

%
\bibitem{bil1} Billingsley, P.: \textsl{Convergence of Probability Measures, Second Edition}. John Wiley $\&$ Sons, (1999).

%
\bibitem{Bik:UAA:2010} Bikulov, A.K.: \textsl{Problem of the first passage time for $p$-adic diffusion}. $p$-adic Numbers, Ultrametric Analysis, and Applications, 2(2), 89-99, (2010).

%
\bibitem{BZ:Physica:2021} Bikulov, A.K., Zubarev, A.P.: \textsl{Ultrametric theory of conformational dynamics of protein molecules in a functional state and the description of experiments on the kinetics of CO binding to myoglobin}. Physica A: Statistical Mechanics and its Applications, 583, 126280, (2021).

%
\bibitem{Bradley:arXiv:2024} Bradley, P.E.: \textsl{Local ultrametric approximation of graph-distance based Laplacian diffusion}. arXiv:2412.20591, (2024).

%
\bibitem{Bradley:JoC:2025} Bradley, P.E.: \textsl{On the Local Ultrametricity of Finite Metric Data}. \emph{J.~Classification}, (2025).

%
\bibitem{CartanEilenberg:HA:1956} Cartan, H., Eilenberg, S.: \textsl{Homological Algebra}. Princeton University Press, Princeton, (1956).

%
\bibitem{cent} Chentsov, N.N.: \textsl{Weak convergence of stochastic processes whose trajectories have no discontinuities of the second kind and the ``heuristic'' approach to the Kolmogorov--Smirnov tests}. Theory of Probability \& Its Applications, 1(1):140-144, (1956).

%
\bibitem{Dragovich:PNUAA:2009} Dragovich, B., Khrennikov, A.Yu., Kozyrev, S.V., Volovich, I.V.: \textsl{On $p$-adic mathematical physics}. \emph{p-Adic Numbers, Ultrametric Anal. Appl.} 1, no.~1, 1--17, (2009).

%
\bibitem{Dragovich:PNUAA:2017} Dragovich, B., Khrennikov, A.Yu., Kozyrev, S.V., Volovich, I.V., Zelenov, E.I.: \textsl{$p$-Adic Mathematical Physics: The First 30 Years}. \emph{p-Adic Numbers, Ultrametric Anal. Appl.} 9, no.~2, 87--121, (2017).

%
\bibitem{Dragovich:Biosystems:2021} Dragovich, B., Khrennikov, A.Yu., Kozyrev, S.V., Misic, N.Z.: \textsl{$p$-Adic mathematics and theoretical biology}. \emph{Biosystems} 199, 104288, (2021).

%
\bibitem{Easo:JMP:2024} Easo, P., Hutchcroft, T., Kurrek, J.: \textsl{Double-exponential susceptibility growth in Dyson's hierarchical model with $\lvert x-y\rvert^{-2}$ interaction.} \emph{J. Math. Phys.} 65, no.~2, 023301, (2024).

%
\bibitem{Fuchs:IAG:1970} Fuchs, L.: \textsl{Infinite Abelian Groups}. Pure and Applied Mathematics, Vol.~1. Academic Press, (1970).

%
\bibitem{Hutchcroft:arXiv:2025I} Hutchcroft, T.: \textsl{Critical long-range percolation I: High effective dimension}. arXiv:2508.18807, (2025).

%
\bibitem{Hutchcroft:arXiv:2025II} Hutchcroft, T.: \textsl{Critical long-range percolation II: Low effective dimension}. arXiv:2508.18808, (2025).

%
\bibitem{Hutchcroft:arXiv:2025III} Hutchcroft, T.: \textsl{Critical long-range percolation III: The upper critical dimension}. arXiv:2508.18809, (2025).

%
\bibitem{Khrennikov:JFAA:2016b} Khrennikov, A., Oleschko, K., Correa Lopez, M.J.: \textsl{Application of $p$-adic wavelets to model reaction–diffusion dynamics in random porous media}. \emph{J.~Fourier Anal.~Appl.} 22, no.~4, 809--822, (2016).

%
\bibitem{Khrennikov:Entropy:2016} Khrennikov, A., Oleschko, K., Correa Lopez, M.J.: \textsl{Modeling fluid’s dynamics with master equations in ultrametric spaces representing the treelike structure of capillary networks}. Entropy, 18, Art. 249, (2016).

%
\bibitem{Khrennikov:medRxiv:2020} Khrennikov, A.: \textsl{Ultrametric model for COVID-19 dynamics: an attempt to explain slow approaching herd immunity in Sweden}. \emph{medRxiv} preprint 2020.07.04.20146209, (2020).

%
\bibitem{Kochubei:Book:2001} Kochubei, A.N.: \textsl{Pseudo-Differential Equations and Stochastics over non-Archimedean Fields.} Monographs and Textbooks in Pure and Applied Mathematics 244 (Marcel Dekker Inc., New York, 2001).

%
\bibitem{Kochubei:UkrMathJ:2018} Kochubei, A.N.: \textsl{Linear and Nonlinear Heat Equations on a $p$-Adic Ball}. \emph{Ukr. Math. J.} 70, 217--231, (2018).

%
\bibitem{Macy:npjComplexity:2024} Macy, M.W., Szymanski, B.K., Holyst, J.A.: \textsl{The Ising model celebrates a century of interdisciplinary contributions.} \emph{npj Complexity} 1, no.~1, 10, (2024).

%
\bibitem{Megard:WSLNP:nd} Megard, M., Parisi, G., Virasoo, M.A.: \textsl{Spin Glass Theory and Beyond.} \emph{World Scientific Lecture Notes in Physics}, (1987).

%
\bibitem{Panchenko:AnnMath:2013} Panchenko, D.: \textsl{The Parisi ultrametricity conjecture.} \emph{Ann.~of Math.} 177, 383--393, (2013).

%
\bibitem{Peski:arXiv:2021} Van Peski, R.: \textsl{Limits and fluctuations of $p$-adic random matrix products.} arXiv:2011.09356, (2021).

%
\bibitem{Peski:arXiv:2023} Van Peski, R.: \textsl{Reflecting Poisson walks and dynamical universality in $p$-adic random matrix theory.} arXiv:2312.11702, (2023).

%
\bibitem{Peski:arXiv:2024a} Van Peski, R.: \textsl{Local limits in $p$-adic random matrix theory.} arXiv:2310.12275, (2024).

%
\bibitem{Peski:arXiv:2024b} Van Peski, R.: \textsl{What is a $p$-adic Dyson Brownian motion?} arXiv:2309.02865, (2024).

%
\bibitem{W:Expo:24} Pierce, T., Rajkumar, R., Stine, A., Weisbart, D., and Yassine, A.: \textsl{Brownian motion in a vector space over a local field is a scaling limit.} Expo. Math., 42(6):125607, (2024).

%
\bibitem{Pierce_Weisbart:JSP:2025} Pierce, T., Weisbart, D.: \textsl{Brownian motion in the $p$-adic integers is a limit of discrete time random walks.} J. Stat. Phys. 192, 104 (2025). 

%
\bibitem{RW:JFAA:2023} Rajkumar, R., Weisbart, D.: \textsl{Components and exit times of Brownian motion in two or more $p$-adic dimensions.} J. Fourier Anal. Appl. 29 (6), 75, (2023).

%
\bibitem{Varadarajan:LMP:1997} Varadarajan, V.S.: \textsl{Path integrals for a class of $p$-adic Schr{\"o}dinger equations.} Lett. Math. Phys. 39, no. 2, 97--106, (1997).

%
\bibitem{Volov:PNUAA:2020} Volov, V.T., Zubarev, A.P.: \textsl{Toward Ultrametric Modeling of the Epidemic Spread.} $p$-Adic Numbers, Ultrametric Anal. Appl. 12, no.~3, 247--258, (2020).

%
\bibitem{vvz} Vladimirov, V.S., Volovich, I.V., Zelenov, E.I.: \textsl{$p$-Adic Analysis and Mathematical Physics.}  World Scientific, (1994).

%
\bibitem{Weisbart:2021} Weisbart, D.: \textsl{Estimates of certain exit probabilities for $p$-adic Brownian bridges.}  J. Theor. Probab., (2021).

%
\bibitem{WJPA} Weisbart, D.: \emph{$p$-Adic Brownian Motion is a Scaling Limit.} Journal of Physics A: Mathematical and Theoretical, (2024). 

%
\bibitem {Zun1}  Z\'{u}\~{n}iga-Galindo, W.A.: \textsl{Pseudodifferential Equations Over Non-Archimedean Spaces.} Lecture Notes in Mathematics (2174), Springer; 1st ed., (2016).

%
\bibitem{ZunigaGalindo:JNMP:2024} Z\'u\~niga-Galindo, W.A., Zambrano-Luna, B.A., Dibba, B.: \textsl{Hierarchical Neural Networks, $p$-Adic PDEs, and Applications to Image Processing.} \emph{J.~Nonlinear Math.~Phys.} 31, no.~1, 63, (2024).

\end{thebibliography}
\end{document}